\documentclass[reqno,a4paper,10 pt]{amsart}
\usepackage{amsmath, amssymb, bbm}
\usepackage[english]{babel}
\usepackage[latin1]{inputenc}
\usepackage{braket}
\usepackage[numbers]{natbib}
\usepackage[colorinlistoftodos]{todonotes}
\usepackage[draft,authormarkup=none]{changes}
\usepackage{enumerate}
\usepackage{mathrsfs}
\usepackage{mathtools}
\usepackage{dsfont}
\usepackage[inline]{enumitem}
\usepackage[nice]{nicefrac}
\usepackage[toc,page]{appendix}
\usepackage[hidelinks,draft=false]{hyperref}
\usepackage[a4paper,top=2cm,bottom=2cm,left=2cm,right=2cm,bindingoffset=5mm]{geometry}
\usepackage{verbatim}
\usepackage{tikz}

\numberwithin{equation}{section}

\newcounter{todocounter}

\allowdisplaybreaks

\newcommand*\ccd[1]{\tikz[baseline=(char.base)]{
            \node[shape=circle,draw,inner sep=2pt] (char) {#1};}}

\newtheorem*{theorem*}{Theorem}
\newtheorem{theorem}{Theorem}
\newtheorem{proposition}[theorem]{Proposition}
\newtheorem{lemma}[theorem]{Lemma}

\newtheorem{definition}[theorem]{Definition}
\newtheorem{corollary}[theorem]{Corollary}
\newtheorem{remark}[theorem]{Remark}

\newtheorem{assumption}{Assumption}

\DeclareMathOperator{\Dom}{Dom}
\DeclareMathOperator{\tr}{Tr}
\DeclareMathOperator{\supp}{supp}

\newenvironment{system}
{\left\lbrace\begin{array}{@{}l@{}}}
{\end{array}\right.}
\DeclarePairedDelimiter{\abs}{\lvert}{\rvert}
\DeclarePairedDelimiter{\norm}{\lVert}{\rVert}
\DeclarePairedDelimiter{\br}{\lbrace}{\rbrace}

\reversemarginpar

\newcommand{\R}{\mathds{R}}
\newcommand{\N}{\mathds{N}}
\newcommand{\E}{\mathds{E}}

\definecolor{capri}{rgb}{0.0, 0.75, 1.0}

\newcommand{\bE}{\mathds{E}}

\newcommand{\bN}{\mathds{N}\,}
\newcommand{\bP}{\mathds{P}}

\newcommand{\bR}{\mathds{R}}

\newcommand{\rJ}{\mathscr{J}}

\newcommand{\sA}{\mathcal{A}}
\newcommand{\sB}{\mathcal{B}}
\newcommand{\sC}{\mathcal{C}}
\newcommand{\sD}{\mathcal{D}}

\newcommand{\sF}{\mathcal{F}}

\newcommand{\sH}{\mathcal{H}}

\newcommand{\sJ}{\mathcal{J}}
\newcommand{\sK}{\mathcal{K}}
\newcommand{\sL}{\mathcal{L}}
\newcommand{\sM}{\mathcal{M}}

\newcommand{\sZ}{\mathcal{Z}}

\newcommand{\ud}{\,\mathrm{d}}

\newcommand{\ind}{\mathds{1}}

\newcommand{\VV}{\Vert}
\newcommand{\aC}{\overset{\curvearrowleft}{\sC}}
\newcommand{\dis}{\displaystyle}

\newcommand{\xphi}[2]{\left(\begin{smallmatrix}#1 \\ #2\end{smallmatrix}\right)}
\renewcommand{\epsilon}{\varepsilon}
\renewcommand{\phi}{\varphi}

\newcommand{\intt}{\int_{t}^s}
\newcommand{\etsa}{e^{(s-r)A}}

\def\<{\left\langle }
\def\>{\right\rangle }

\title{Semilinear Kolmogorov equations on the space of continuous functions via BSDEs}

\begin{document}

\author{Federica Masiero}
\address[F. Masiero]{Dipartimento di Matematica e Applicazioni, Universit\`a di Milano-Bicocca. via Cozzi 55, 20125 Milano, Italia}
\email{federica.masiero@unimib.it}

\author{Carlo Orrieri}
\address[C. Orrieri]{Dipartimento di Matematica, Universit\`a di Trento. Via Sommarive 14, 38123 Trento, Italia}
\email{carlo.orrieri@unitn.it}

\author{Gianmario Tessitore}
\address[G. Tessitore]{Dipartimento di Matematica e Applicazioni, Universit\`a di Milano-Bicocca. via Cozzi 55, 20125 Milano, Italia}
\email{gianmario.tessitore@unimib.it}

\author{Giovanni Zanco}
\address[G. Zanco]{Dipartimento di Economia e Finanza, LUISS Guido Carli. Viale Romania 32, 00197 Roma, Italia}
\email{gzanco@luiss.it}

\begin{abstract}
  We deal with a class of semilinear parabolic PDEs on the space of continuous functions
that arise, for example, as Kolmogorov equations associated to the infinite-dimensional lifting of path-dependent SDEs. 
  We investigate existence of smooth solutions through their representation via forward-backward stochastic systems, for which we provide the necessary regularity theory. 
  Because of the lack of smoothing properties of the parabolic operators at hand, solutions in general will only share the same regularity as the coefficients of the equation.
To conclude we exhibit an application to Hamilton-Jacobi-Bellman equations associated to suitable optimal control problems.
\end{abstract}

\maketitle

AMS 2010 Subject classification: 35K58, 60H10, 60H30, 93E20.

\tableofcontents
\section{Introduction}
\label{sec:intro}
The aim of this paper is to address the infinite dimensional semilinear backward Kolmogorov PDE
\begin{equation}\label{eq:kolm:diff:nonlinear}
  \begin{system}
  \frac{\partial u}{\partial t}(t,x) + Du(t,x)\left[Ax + B(t,x)\right] + \frac{1}{2}\tr_{\bR^d} \left[\Sigma\Sigma^\ast D^2u(t,x)\right] = G(t,x,u(t,x),Du(t,x)\Sigma) \ ,\\
  u(T,\cdot) = \Phi
\end{system}
\end{equation} 
in the space of continuous functions on a real interval. Under suitable assumptions on the coefficients and the terminal condition we provide existence of smooth (classical) solutions to (\ref{eq:kolm:diff:nonlinear}) through the associated forward-backward stochastic system, extending the methods introduced in 
\cite{flandoli2016infinite} for the linear equation (i.e. $G\equiv 0$).\\
PDEs of the above given form naturally arise in connection with path-dependent stochastic differential equations in finite dimension through their 
infinite-dimensional reformulation in the so-called product space framework, proposed first in \cite{delfour1972hereditary} 
for deterministic systems with delay and in \cite{chojnowska1978representation} for stochastic ones.  
In particular semilinear equations as \eqref{eq:kolm:diff:nonlinear} describe the value function of optimal control problems for stochastic path dependent state equations.\\
While the solution theory for path-dependent stochastic systems 
is classical (see e.g. the monograph \cite{mohammed1984stochastic})
at least when the coefficients are regular enough,
the study of associated PDEs is a relatively recent subject for which different approaches have been proposed in the last years.\\
The recent research activity on path-dependent functionals of stochastic processes and related PDEs originated from insight by Dupire \cite{dupire2009functional} and investigation by Cont and Fourni\'e \cite{cont2013functional}.
Due to the general lack of regularity of path-dependent functionals (most notably $\gamma \mapsto \sup_{s\in[0,T]}\vert \gamma_s\vert$) various authors have introduced different  weak notions of solutions for nonlinear path-dependent PDEs, see for example \cite{ekren2014, ekren2016, cosso2018, cosso2019osaka, Cordoni2017, dipersio2019, bayaraktar2018, zhou2018}; in many cases such PDEs arise in connection with stochastic control problems, as in \cite{guatteri2017stochastic, fuhrman2010stochastic, tang2015}. 
Nonetheless, at the current stage there is no complete theory even for regular solutions of nonlinear PDEs (existence of solutions was proved in \cite{cosso2014pdep} and \cite{cosso2019idaqp} only for coefficients with a very specific cylindrical form), and only the linear case has been extensively investigated in this sense \cite{flandoli2016infinite, digirolami2018}.
The appearance of so many different approaches is essentially motivated by the intrinsic infinite-dimensional nature of the problem which then reflects in different notions of differential for functions of paths.\\
It is by now well understood that the parabolic-type operators associated to path-dependent SDEs do not possess
smoothing properties in general (although there is a kind of partial smoothing in some problems with delay and for particular choices of the coefficients, see \cite{gozzi2017partial}): this affects the regularity of any type of solution, and makes the study of regular solutions nontrivial. 
In the case we discuss here, a precise investigation of differentiability of the forward-backward system has never been rigorously carried out before. 
The approach we develop provides, under suitable assumptions on the regularity of the coefficients and on finite-dimensionality of the noise, a solution theory for general PDEs on the space of continuous functions .
For a discussion on its relation with the functional Ito calculus we refer to \cite{zanco2015phd}\\

Let us now briefly introduce our framework and main results, sketching the general lines of the proofs. 
Fix a finite time horizon $T>0$ and consider the path-dependent SDE in $\bR^d$
\begin{equation}
\label{eq:SDEx}
\begin{system} 
\ud\xi_s = b_s(\xi_{[0,s]})\ud s + \sigma \ud W_s\ ,\quad s\in[t,T]\ ,\\
\xi_{t} = \gamma_{t},
\end{system}
\end{equation}
where $W$ is a $d_1$-dimensional Brownian Motion, $\sigma$ is a $d \times d_1$ matrix, $t\in[0,T]$ and $\gamma_{t}$ is a given function that belongs to $D([0,t];\bR^d)$ (the space of c\`adl\`ag functions on $[0,t]$, endowed with the supremum norm). 
The value of the solution process $\xi$ at time $s$ is denoted by $\xi_s$, while its path up to time $s$ is denoted by $\xi_{[0,s]}$.
The drift $b$ at time $s$ depends on the whole past trajectory of the solution $\xi_{[0,s]}$ and it is therefore given as a family $\{b_s\}_{s\in[0,T]}$
\begin{equation*}
\label{eq:driftb_intro}
  b_s : D([0,s];\bR^d) \to \bR^d\ .
\end{equation*} 
Note that for different times the drift $b$ is defined on different spaces of paths; while this is not an issue in the study of the SDE (\ref{eq:SDEx}), it becomes a delicate question for the investigation of the associated Kolmogorov PDE.
Furthermore, even if the solution to \eqref{eq:SDEx} has continuous paths from time $t$ on, it is convenient (actually unavoidable) to formulate everything in spaces of c\`adl\`ag functions.\\

The product-space reformulation  of \eqref{eq:SDEx} consists in separating the present state $\xi_t$ from the past trajectory $\xi_{[0,t)}$, rewriting the second one via a time change as a function on $[-t,0)$ and then lengthening it towards the past up to $[-T,0)$. 
In this way it is possible to distinguish between the time $t$ of the forward equation and the time variable of the past trajectory:  for any time $T$ a process
\begin{equation*}
  X_t=\left(\begin{matrix}\xi_t\\ \left\{\xi_{t+r}\right\}_{r\in[-T,0)}\end{matrix}\right)\in\bR^d\times D\left([-T,0);\bR^d\right) \ ,
\end{equation*}
is defined, whose second component is now defined on a fixed functional space.
This reformulation allows to recover Markovianity and turns out to be particularly convenient to investigate differentiability properties of the solution of the nonlinear Kolmogorov PDE.
As a drawback an additional linear term comes into play, which is given by a first order differential operator $A$ usually refer to as the \emph{generator of the delay}. 
Indeed, the process $X$ turns out to be a solution to the following  SDE (the \emph{forward equation} in what follows)
\begin{equation}
  \label{eq:SDEX}
\begin{system}
  \ud X_s=AX_s \ud s+B(s,X_s)\ud s+\Sigma \ud W_s\ ,\quad s\in(t,T]\\
  X_{t}=x\ \text{,}
\end{system}
\end{equation}
where $B$, $\Sigma$ and $x$ are suitable infinite-dimensional liftings of $b$, $\sigma$ and $\gamma_{t}$, respectively. \\

Given \eqref{eq:SDEX} it is natural to associate, at least formally, the linear backward Kolmogorov equation
\begin{equation}\label{eq:kolm:diff}
\begin{system}
\frac{\partial u}{\partial t}(t,x) + Du(t,x)\left[Ax + B(t,x)\right] + \frac{1}{2}\tr_{\bR^d} \left[\Sigma\Sigma^\ast D^2u(t,x)\right] = 0, \\
u(T,\cdot) = \Phi
\end{system}
\end{equation} 
on $[0,T]\times \left(\bR^d\times D\left([-T,0);\bR^d\right)\right)$.
The terms $Du$ and $D^2u$ denote the Fr\`echet differentials of the solution $u$ with respect to the variable $x\in \bR^d\times D\left([-T,0);\bR^d\right)$ and the terms $Du B$ and $\tr\left[\Sigma\Sigma^\ast D^2u\right]$ only depends on the $\bR^d$-component of $x$ (recall that $X$ generates from a finite-dimensional SDE).

Then, to account for a nonlinear term $G$ as  in equation \eqref{eq:kolm:diff:nonlinear}, the introduction of the following  backward SDE (BSDE) is essential
\begin{equation}
\label{eq:BSDE}
\begin{system}
\ud Y_s = G(s,X^{t,x}_s,Y^{t,x}_s,Z^{t,x}_s) \ud s + Z_s \ud W_s, \qquad s \in [t,T],\\
Y_T = \Phi(X^{t,x}(T)) \ , 
\end{system}
\end{equation}
where the variables $t$ and $x$ refer to the initial data of the forward equation satisfied by $X$.
A solution to \eqref{eq:BSDE} is a pair of processes $(Y^{t,x},Z^{t,x})$ with values in $\bR\times\bR^{d_1}$ and the (partially-coupled) system generated by \eqref{eq:SDEX} and \eqref{eq:BSDE} goes under the name of \emph{forward-backward system}.
Notice that, even if the solution $(Y,Z)$ is finite-dimensional, it depends in a nontrivial way on the forward process $X$ that takes values in an infinite-dimensional space.\\
Our main result is a version of the nonlinear Feynman-Kac formula in terms of backward SDEs.
\begin{theorem*}
The function
  \begin{equation*}
    u(t,x)=Y^{t,x}_t\ ,
  \end{equation*}
where $(Y^{t,x},Z^{t,x})$ is the unique solution of (\ref{eq:BSDE}), is a classical solution of the semilinear Kolmogorov backward equation with terminal condition $\Phi$.
\end{theorem*}
Here by classical solution we mean a function that is two times differentiable with respect to $x$ and satisfies \eqref{eq:kolm:diff:nonlinear} for every $t\in[0,T]$ and for every $x$ that belongs to the domain of $A$. 
Since the solution $u(t,x)$ is represented by $Y^{t,x}_t$, where $(Y^{t,x},Z^{t,x})$ is the solution to the (\ref{eq:BSDE}), it is crucial to study Frechet differentiability of the map $(t,x)\mapsto Y^{t,x}$, up to the second order with respect to the variable $x$.
At our best knowledge a precise investigation of differentiability, up to the second order, of the forward-backward system has never been rigorously carried out before in the generality needed here. As a matter of fact higher order differentiability for BSDEs  has only been taken into account in \cite{izumi2018higher} in a non-Markovian setting and in Malliavin sense. Although  the two arguments have several technical similarities it seems that here we cannot use the result in  \cite{izumi2018higher}, as we rely only on G\^{a}teaux and Fr\'echet differential calculus. \\
We firstly prove the above theorem in the space $L^2(-T,0;\bR^d)$ and then extend our results to $D([-t,0);\R^d)$. 
Note that requiring regularity in $L^2$-sense drastically restricts the class of models one can consider so that the $L^2$-theory has no much  
relevance by itself. Nonetheless it is a fundamental intermediate step for studying the PDE on $D$. 
As a marginal remark we notice that the $L^2$-theory can be easily adapted to get existence of classical solutions with coefficients in $L^p$, $p>2$, allowing to recover already at this level some interesting examples. \\
Once the $L^2$-theory is established, we  proceed as follows: the coefficients $B$, $G$ and $\Phi$, defined on $D\left([-T,0);\bR^d\right)$, are approximated by suitable sequences $B^n$, $G^n$, $\Phi^n$ defined on $L^2\left(-T,0;\bR^d\right)$, providing a family of solutions $u^n$ of the approximated PDEs
\begin{equation*}
\begin{system}
\frac{\partial u^n}{\partial t}(t,x) + Du^n(t,x)\left[Ax + B^n(t,x)\right] + \frac{1}{2}\tr_{\bR^d} \left[\Sigma\Sigma^\ast D^2u^n(t,x)\right] = G^n(t,x,u^n(t,x),Du^n(t,x)\Sigma), \\
u^n(T,\cdot) = \Phi^n\ .
\end{system}
\end{equation*} 
To conclude the proof we need pass to the limit as $n$ tends to infinity in each term of the PDE. While the derivation of the semilinear PDE in $L^2$ is similar to the linear case, the passage to the limit shows substantial differences with the corresponding proof for the linear case, and requires a nontrivial analysis of the convergence of the BSDE (\ref{eq:BSDE}) together with its first and second derivative.\\
The choice to work in  $D([-t,0);\R^d)$ is motivated by what we hinted at above: there are very few functions satisfying the needed regularity assumptions in $L^2$ but many significant examples, most notably those involving pointwise evaluations of the path, can be recovered switching to Banach spaces with a finer topology (see \cite{flandoli2016infinite} for a discussion of several examples meeting our assumptions). 
To this end, the choice of the space of continuous functions would seem the most natural and appropriate one. 
However the infinite dimensional reformulation mentioned above has the drawback of creating discontinuities: as an example, the lifted drift term $B$ has the form
\begin{equation*}
  B(t,x)=\left(\begin{matrix}b_t\\0\end{matrix}\right)\ 
\end{equation*}
and has to be interpreted as a c\`adl\`ag function that is non-zero only at the current time $t$. 
Consequently, the operator $A$ introduces a transport effect, explicitly visible in the mild formulation of \eqref{eq:SDEX}, shifting the discontinuity over time. 
It is therefore convenient to formulate everything in the larger space of c\`adl\`ag paths and restrict to the subspace of continuous paths when needed.\\
The role of the intermediate $L^2$ step can be informally explained as follows: the natural scheme to investigate existence of regular solutions to \eqref{eq:kolm:diff:nonlinear} consists in combining some form of It\^o formula  with a smoothness result for the solution of the forward-backward system (with respect to the initial data of the forward equation).  
However, because of the spaces we are working with and of the particular form of the noise, no It\^o formula applies to our system and we have to rely on a particular Taylor expansion that exploits the Markovianity recovered through our infinite-dimensional reformulation. 
Furthermore, to obtain the PDE we need a control over the second order term which is achievable only in $L^p$ spaces; in particular this allows to show that the second order term is concentrated on the finite-dimensional component, thus providing the trace term as it appears in the equation. 
The same result cannot be directly proved through estimates with respect to the supremum norm.

Let us finally stress that, as all the technical difficulties related to path-depedency and to the use of c\`adl\`g spaces are already present in the additive noise case, we choose to work in such a setting that considerably simplifies the technical aspects of several points.\\

We eventually apply the result to a stochastic control problem, for the state equation 
\begin{equation*}
\label{eq:SDEX:control}
\begin{system}
  \ud X^u_s=AX^u_s\ud s+B\left(s,X^u_s\right)\ud s+\Sigma u_s \ud s+\Sigma \ud W_s\ ,\quad s\in[t,T]\\
  X^u_t=x\ \text{.}
\end{system}
\end{equation*}
We aim at minimizng the cost functional $\mathscr{J}: [0,T] \times \sD \times \R^{d_1} \to \R$
\begin{equation*}
  \label{cost:abstract}
\mathscr{J}\left(  t,x,u\right) :=\E\int_{t}^{T}\left[
 L\left(s,X^{u;t,x}_s\right) +Q\left(u_s\right)\right]\ud s+\E\Upsilon\left(X^{u;t,x}_T\right) \ ,
\end{equation*}
over all {\it{admissible}} controls. The Hamilton-Jacobi-Bellman equation is related to a semilinear Kolmogorov PDE that can be solved in classical sense thanks to the results proved herein; as a consequence we are able to prove the existence of optimal controls in strong formulation.

We briefly outline the structure of the paper. Section \ref{sec:prel} contains notation and classical results on BSDE that will be used in the sequel. Section \ref{sec:problem} introduces rigorously the product space framework and the assumptions that will stand throughout the paper. In Section \ref{sec:FBSDE} we prove some results about regularity of the solution of the stochastic forward-backward system with respect to the initial data of the forward process. Up to this point results are proved in a generic Banach space $E$, and possibly specialized to particular spaces when needed. Section \ref{sec:PDEL2} is devoted to the proof of existence of a solution to the semilinear backward PDE in $\sL^2$. In Section \ref{sec:PDED} we carry out the limit procedure and prove the main result. Finally Section \ref{sec:control} contains some applications to optimal control problems.

\section{Notation and preliminaries}
\label{sec:prel}
Let $(\Omega,\sF,\bP)$ be a probability space and fix a time interval $[0,T]$.
We denote by $W_t$, $t \geq 0$,  a $d_1$-dimensional standard Brownian motion and by $\sF_t$ the associated natural filtration, completed with the null sets in $\sF_T$. 
All the measurability properties we refer to have to be intended with respect to this filtration.
In the following, given a $\sF_t$-measurable random variable with finite expectation, we denote by $\E^{\sF_s}(X_t):= \E(X_t | \sF_s)$ the conditional expectation of $X_t$ given $\sF_s$.  \\

We denote by $E$ a general Banach space, whose norm is given by $|\cdot|_E$, or simply by $|\cdot|$, when no confusion can arise.
For any pair of Banach spaces $E, F$, we write $L(E,F)$ for the space of linear and bounded operators $T:E \to F$, endowed with the operator norm.
In the special case $F = \bR$, we shorthand $E':= L(E;\bR)$. 
The operator norm is indicated  by $\| T \|_{L(E,F)}$, or $\| T \|$ if no confusion is possible.

Moreover, given two possibly different Banach spaces $E_1, E_2$ we indicate with $L(E_1,E_2;F)$ the space of bilinear maps (linear in each argument) from $E_1 \times E_2 \to F$. In the following we will identify $L(E_1,E_2;F)$ with $L(E_1; L(E_2; F))$.   

For every $p,q \geq 1$, we use the following notation for classes of random variables and stochastic processes with values in a Banach space $E$:
\begin{itemize}
\item $L^p_{\sF_T}(\Omega;E)$, the set of $\sF_T$-measurable $E$-valued random variable endowed with the norm
\[ \| X \|_{L^p_{\sF_T(\Omega;E)}}  : = \left( \E |X|^p_E \right)^{\nicefrac{1}{p}};\]  
\item $L^p(\Omega \times [0,T]; E)$, the set of progressively measurable $E$-valued processes endowed with the norm
\[	\| X \|_{L^p(\Omega \times [0,T]; E)} := \left( \E \int_0^T |X_t|_E^p \ud t \right)^{\nicefrac{1}{p}};\]
\item $L^p\left(\Omega;L^q(0,T;E)\right)$, the space of progressively measurable $E$-valued processes with the norm given by
\[  \| X \|_{L^p(\Omega;L^q(0,T;E))} :=  \left( \E \left( \int_0^T |X_t|_E^q \ud t \right)^{\nicefrac{p}{q}}\right)^{\nicefrac{1}{p}};\]
\item $L^p\left(\Omega;C([0,T];E)\right)$, the space of progressively measurable $E$-valued processes such that the map $t \mapsto X_t$ is a.s. continuous and the norm
\[ \|  X \|_{L^p(\Omega;C([0,T];E))} : = \left( \E \sup_{t \in [0,T]} |Y_t|^p \right)^{\nicefrac{1}{p}}, \]
is finite. 
\end{itemize}

If $E = \bR$, to shorten the notation we denote by  $\sK_p$ the product space
\begin{equation}
 \sK_p:=L^p\left(\Omega;C([0,T];\bR)\right)\times L^p\left(\Omega;L^2(0,T;\bR^{d_1})\right).
\end{equation}

We say that a function $R: E \to F$ belongs to $C^{n,\alpha}(E; F)$ if it is $n$-times Fr\'echet differentiable in $E$ with measurable differentials $D^j R$, $j =1, \ldots, n$, and the map $x \mapsto D^nR(x)$ is $\alpha$-H\"older continuous with measurable norm.\\
We say that a function $S:[0,T]\times E\to F$ belongs to $C^{1;n,\alpha}$ if for every $x\in E$ the map $t\mapsto S(t,x)$ is differentiable with measurable differential and for every $t\in[0,T]$ the map $x\mapsto S(t,x)$ belongs to $C^{n,\alpha}$.
For space-time functions $R = R(t,x)$ we will denote by $\frac{\partial R}{\partial t}$ the derivative w.r.t. $t$ and by $D^j R$ the Fr\'echet differentials w.r.t. $x$.\\
\smallskip
In what follows we generally use capital letters $X,Y,Z, \ldots$ to denote random variables, on the contrary we use small letters $x,y,z, \ldots$ to denote deterministic objects.
Whenever we write $a \lesssim b$, with $a,b \in \bR$, we mean that there exists a constant $c >0$ for which $a \leq c\, b$.

\subsection{BSDEs toolbox}
\label{subsec:toolbox}
Here we collect some basic results from the theory of Backward SDEs that will be useful in the sequel. We refer to \cite{pardoux2014stochastic} for a general introduction to the subject.\\
Given  a $\sF_T$-measurable real-valued random variable $\eta$ 
and a driver $g: \Omega \times [0,T] \times \bR\times\bR^{d_1}\to\bR$ which is $ \sF_t \otimes \sB([0,t]) \otimes \sB(\bR \times \bR^{d_1})$-measurable, we say that a pair of progressively measurable processes $(Y,Z) \in \sK_p$ is a solution to the BSDE associated with $(g,\eta)$ if $\bP$-a.s.
\begin{equation}\label{BSDE}
Y_t = \eta + \int_t^T g(s,Y_s,Z_s) \ud s - \int_t^T Z_s \ud W_s \ , \qquad 0 \leq  t \leq T \ .
\end{equation}
In a differential formulation, we also write that $\bP$-a.s.
\begin{equation}
-\ud Y_t = g(t,Y_t,Z_t) \ud t - Z_t \ud W_t \ , \qquad Y_T = \eta \ , \qquad 0 \leq  t \leq T \ .
\end{equation}
 
Wellposedness results and a priori estimates for solutions to \eqref{BSDE} hold under specific assumptions on the pair $(g,\eta)$.
Let us recall here a classical result with uniform Lipschitz hypothesis.

\begin{proposition}\label{prop:FT}
Let $p >1$ and  $\eta\in L^p_{\sF_T}(\Omega;\bR)$. Moreover, suppose that
\begin{itemize}
\item[(i)] There exists $L > 0 $ for which
\[\abs{g(s,y_1,z_1)-g(s,y_2,z_2)}\leq L(\abs{y_1-y_2}+\abs{z_1-z_2}),  \qquad \forall \,s \in [0,T], \, \bP\text{-a.s.} \ , \] 
for any $y_1,y_2 \in \R$ and $z_1,z_2 \in \R^{d_1}$;
\item[(ii)] $\E\left(\int_0^T\abs{g(s,0,0)}^2\ud s\right)^{\nicefrac{p}{2}}<+\infty \ .$ 
\end{itemize} 
Then the BSDE \eqref{BSDE} admits a unique solution $(Y,Z) \in \sK_p$ and for every $t \in [0,T]$ it holds
\begin{equation}
  \label{eq:BSDE_stima_classica}
  \E\sup_{s\in[t,T]}\abs{Y_s}^p+\E\left(\int_{t}^T\abs{Z_r}^2\ud r\right)^{\nicefrac{p}{2}}\leq C \E\left(\int_{t}^T\abs{g(r,0,0)}^2\ud r\right)^{\nicefrac{p}{2}}+C \E\abs{\eta}^p,
\end{equation}
where $C = C(p,L,T)$ is a positive constant.
\end{proposition}
For a proof of Proposition \ref{prop:FT} we refer to  \cite[Thm.~5.21]{pardoux2014stochastic} where (integrable) time-dependent Lipschitz constants are also taken into account.\\
In the sequel we will be interested in BSDEs associated with data $(g,\eta)$ depending on a given stochastic process.
Precisely, consider a stochastic process $X$ with values in a general Banach space $E$  and assume that $g: \Omega \times [0,T] \times E \times \bR\times\bR^{d_1}\to\bR$ and $\eta = \phi(\cdot)$, $\phi: E \to \R$, are given measurable functions. 
If we write the equation
\begin{equation}\label{BSDE_X}
Y_t = \phi(X_T) + \int_t^T g(s,X_s ,Y_s, Z_s) \ud s - \int_t^T Z_s \ud W_s \ , \qquad 0 \leq  t \leq T \ ,
\end{equation}
existence and uniqueness of a solution in $\sK_p$ is a consequence of Proposition \ref{prop:FT}.

Let us now give an explicit formula for one-dimensional BSDEs under general integrability conditions on the driver.

\begin{lemma}\label{l:linearBSDE}
Let $p >1$ and $\eta \in L^p_{\sF_T}(\Omega; \bR)$. 
Suppose that $(a_t)_{t\geq 0}$, $(b_t)_{t\geq 0}$ are bounded $\bR$-valued and $\bR^{d_1}$-valued processes, respectively, and $c \in L^p(\Omega; L^1(0,T;\bR))$, i.e.
\[\E\left( \int_0^T |c_s| \ud s\right)^p < +\infty.\]
Then the BSDE 
\begin{equation}\label{linear_BSDE}
-\ud Y_t = \left( a_t Y_t + b_t Z_t + c_t \right) \ud t - Z_t \ud W_t, \qquad Y_T = \eta,
\end{equation}
admits a unique solution $(Y,Z) \in \sK_p$.
The process $Y$ can be written as
\begin{equation}
Y_t = \Gamma_t^{-1} \E^{\sF_t} \left[ \Gamma_T \eta + \int_t^T \Gamma_s c_s \ud s, \right]
\end{equation}
where $\Gamma$ is given by the formula
\begin{equation}
\Gamma_t = \exp \left[ \int_0^t \left( a_s - \frac{1}{2}|b_s|^2 \right) \ud s + \int_0^t b_s \ud W_s \right].
\end{equation}
Moreover, by setting
\begin{equation*}
  V_t := \int_0^t |a_s| \ud s + \frac{1}{1 \wedge (p-1)}\int_0^t|b_s|^2 \ud s\ ,
\end{equation*}
there exists $C = C(p) >0$ such that 
\begin{equation}\label{estimate_linear}
\E \sup_{t \in [0,T]} \left| e^{V_t} Y_t \right|^p + \E \left( \int_0^T e^{2V_r} |Z_r|^2 \ud r \right)^{\nicefrac{p}{2}} \leq C \,\E \left| e^{V_T} \eta\right|^p + C \,\E\left( \int_0^T e^{V_r}|c_r| \ud r \right)^p.
\end{equation}
\end{lemma} 

\begin{proof}
A proof of this result can be easily derived from \cite[Prop.~5.31]{pardoux2014stochastic}, in which more general growth conditions on the coefficients are taken into account.  
\end{proof}

\section{Setting of the problem and Assumptions}
\label{sec:problem}
In this section we firstly show how PDEs of the form of \eqref{eq:kolm:diff:nonlinear}  naturally arise in connection with path-dependent stochastic dynamics. 
Even if path-dependent calculus remains our main motivation, the method we develop here applies to a wide class of equations that do not necessarily originate from path-dependent problems  (see the discussion at the end of subsection \ref{subsec:FBSDE}). 
We subsequently introduce the assumptions under which the main results will be valid.

\subsection{The forward-backward system and the PDE}
\label{subsec:FBSDE}
In what follows, for a given path $\xi$ we will denote by $\xi_t$ the value of $\xi$ at time $t$, while we will use the notation $\xi_{[0,t]}$ for the path of $\xi$ up to time $t$, that is $\xi_{[0,t]}=\br{\xi(s)}_{s\in[0,t]}$. 
We will denote by $C([0,t];\bR^d)$ and $D([0,t];\bR^d)$ the space of $\bR^d$-valued continuous and c\`adl\`ag functions, respectively, defined on the interval $[0,t]$. \\
Let us introduce the  path-dependent SDE
\begin{equation}
\tag{\ref{eq:SDEx}}
\begin{system} 
\ud\xi_s = b_s(\xi_{[0,s]})\ud s + \sigma \ud W_s\ ,\quad s\in[t,T]\ ,\\
\xi_{[0,t]} = \gamma
\end{system}
\end{equation}
where $\gamma\in D\left([0,t];\bR^d\right)$ is a given deterministic curve and the drift $b$ is a family $\{b_s\}_{s\in[0,T]}$,
\begin{equation}
\label{eq:driftb}
  b_s : D([0,s];\bR^d) \to \bR^d\ .
\end{equation}
A solution to (\ref{eq:SDEx}) will be denoted by $\xi^{\gamma,t}$.
Some authors define the drift equivalently as a map
\begin{equation*}
  b:[0,T]\times D([0,T]:\bR^d)\to\bR^d
\end{equation*}
that is \emph{non-anticipative}: $b(s,\chi)=b(s,\chi_{[0,s]})$ for every $s\in[0,T]$. 
In this setting, non-anticipativeness is assured requiring $b$ to be measurable with respect to the $\sigma$-algebra induced on $D([0,T];\bR^d)$ by the metric
\begin{equation*}
  d_\infty\left((s,\mu),(t,\chi)\right)=\abs{s-t}+\sup_{r\in[0,T]}\abs{\mu(r\wedge s)-\eta(r\wedge t)}\ .
\end{equation*}
We will prefer the first formulation (\ref{eq:driftb}) in what follow, but everything can be easily adapted to the second one, where the particular topology induced by the metric $d_\infty$ has to be taken into account.

We now introduce the product space framework (see \citep[Chap.~4]{benosussan2007representation} for a general discussion), where the present state $\xi_s$ and the past trajectory $\xi_{[0,s)}$ are seen as separate variables. 
Setting
\begin{align*}
  E_0&:=\left\{\phi\in C\left([-T,0);\bR^d\right) \colon \exists\lim_{r\uparrow 0}\phi(r)\in\bR^d\right\}\ \text{,}  \\
    D_0&:=\left\{\phi\in D\left([-T,0);\bR^d\right) \colon \exists\lim_{r\uparrow 0}\phi(r)\in\bR^d\right\}\ \text{,}  \\
\end{align*}
we define the spaces
\begin{equation}
  \label{eq:spaces}
\begin{aligned}
    \sC&:=\bR^d\times E_0\ \text{,}\\
    \aC&:=\left\{x=\xphi{y}{\phi}\in\sC \text{ s.t. }y=\lim_{r\uparrow 0}\phi(r)\right\}\text{,}\\
    \sD&:=\bR^d\times D_0\ \text{,}\\
   \sL^2&:=\bR^d\times L^2\left(-T,0;\bR^d\right)\ \text{.}
 \end{aligned}
\end{equation}
The spaces $\sC$, $\aC$ and $\sD$ are Banach spaces with respect to the norm $\VV \xphi{y}{\phi}\VV^2=\vert y\vert^2+{\VV\phi\VV}_\infty^2$, while $\sL^2$ is a Banach space with respect to the norm $\VV\xphi{y}{\phi}\VV^2=\vert y\vert^2+\VV\phi\VV_{2}^2$.
The space $\sD$, endowed with the topology given by the norm above, is not separable, but this will not undermine the methods used herein.\\
With these norms we have the natural inclusions
\begin{equation*}
  \aC\subset\sC\subset\sD\subset\sL^2
\end{equation*}
with continuous embeddings. We remark that $\aC$, $\sC$ and $\sD$ are dense in $\sL^2$ while neither $\aC$ nor $\sC$ are dense in $\sD$. 
The choice of the interval $[-T,0]$ is made in accordance with most of the classical literature on delay equations.
Note also that the space $\aC$ does not have the structure of a product space, and it is isomorphic to the space $C\left([-T,0];\bR^d\right)$.\\

The reformulation of equation (\ref{eq:SDEx}) in infinite dimensions is obtained through the family of \emph{restriction operators}
\begin{equation}\label{eq:M_t}
\begin{gathered}
  M_t\colon \sD\longrightarrow D\left([0,t];\bR^d\right)\\
           M_t\left(\xphi{y}{\phi}\right)=\phi(s-t)\ind_{[0,t)}(s)+y\ind_{\{t\}}(s)\ \text{.}
\end{gathered} 
\end{equation}
Using $M_t$ we can define the operator
\begin{equation}\label{eq:Bprod}
\begin{gathered}
 B\colon [0,T]\times \sD\to\sD\\
          B\left(t,\xphi{y}{\phi}\right)=\left(\begin{matrix} b_t\left(M_t\xphi{y}{\phi}\right) \\ 0\end{matrix}\right)\ .
\end{gathered}
\end{equation}
Note that the variable $t$ appears explicitely in $B$ even if $b$ does not depend explicitely on time; such a variable acts here as a selector for $b_t$ and $M_t$. 
The right inverse of $M_t$ is the \emph{backward extension} operators defined as
\begin{equation}
\label{eq:Lt}
\begin{gathered}
L^t\colon D([0,t];\bR^d)\longrightarrow \sD\\
           L^t(\chi)=\left(\begin{matrix}\chi(t)\\ \chi(0)\ind_{[-T,-t)}+\chi(t+\cdot)\ind_{[-t,0)}
             \end{matrix}
\right)\ \text{;}
\end{gathered}
\end{equation}
with these definitions we have that $M_tL^t\gamma=\gamma$ for every $\gamma\in D([0,t];\bR^d)$.

Finally let us introduce the operator
\begin{equation*}
\label{eq:DAE}
\Dom\left(A\right)=\left\{\xphi{y}{\phi}\in \sL^2 \colon \phi\in W^{1,2}\left(-T,0;\bR^d\right),\ y=\lim_{r\to 0^-}\phi(r)\right\}\ \text{,}
\end{equation*}
\begin{equation*}
\label{eq:A}
A=\left(
\begin{array}
[c]{cc}
0 & 0\\
0 & \frac{d}{dr}
\end{array}
\right)\ ,
\end{equation*}
(we identify an element of $W^{1,2}(-T,0;\bR^d)$ with its continuous version restricted to $[-T,0)$ ) and the space
\begin{equation*}
  \aC^1:=A^{-1}\left(\br{0}\times E_0\right)=\left\{\xphi{y}{\phi}\in\aC\colon \phi\in C^1\left([-T,0);\bR^d\right), \, \exists \lim_{r \uparrow 0} \phi^\prime(r)\right\}\ \text{.}
\end{equation*}
The operator $A$ generates a strongly continuous semigroup $e^{tA}$ in $\sL^2$  which is explicitely given by the formula
\begin{equation}
  \label{eq:semigroup}
  e^{tA}\xphi{y}{\phi}=\left(\begin{matrix} y \\ \phi(\cdot +t)\ind_{[-T,-t)}+y\ind_{[-t,0)}\end{matrix}\right)\ 
\end{equation}
(see \cite{benosussan2007representation} or \cite{yosida1995functional} for details). 
It is evident that such a semigroup is well defined on $\sC$ and $\sD$, maps $\sD$ into itself, but it is not strongly continuous neither in $\sD$ nor in $\sC$.
 Nevertheless it is equibounded in $\sD$, it maps $\aC$ in itself and it is strongly continuous in $\aC$.\\

Consider now a strong solution $\xi=\xi^{\gamma,t}$ to equation \eqref{eq:SDEx} and set
\begin{equation*}
  X_s:=L^s\xi_{[0,s]}\ \text{.}
\end{equation*}
$X$ is a $\sD$-valued process that solves the SDE
\begin{equation}
  \tag{\ref{eq:SDEX}}
\begin{system}
  \ud X_s=AX_s \ud s+B(s,X_s)\ud s+\Sigma \ud W_s\ ,\quad s\in(t,T]\\
  X_{t}=L^t\gamma=:x\ \text{,}
\end{system}
\end{equation}
in mild sense, that is, it satisfies
\begin{equation}
  \label{eq:SDEXmild}
  X_s=e^{sA}x+\int_t^s \etsa B(r,X_r)\ud r+\intt\etsa\Sigma \ud W_s\ ,\quad s\in[t,T] \ ,
\end{equation}
where $\Sigma:\bR^{d_1}\to\sD$ is the operator given by
\begin{equation}
\label{Sigma}
  \Sigma w =\left(\begin{matrix} \sigma w\\ 0\end{matrix}\right)
\end{equation}
and
  \begin{equation}
    \label{eq:conv_def}
    \intt\etsa\Sigma\ud W_s=\intt\etsa\begin{pmatrix}\sigma \ud W_r\\0\end{pmatrix}=\begin{pmatrix}\intt\sigma\ud W_r\\ \intt\ind_{[-(s-r),0]}(\cdot)\sigma\ud W_r\end{pmatrix}=\begin{pmatrix}\sigma\left(W_s-W_{t}\right)\\ \sigma\left(W_{(s+\cdot)\vee t}-W_{t}\right)\end{pmatrix}\ .
  \end{equation}
Conversely, if $X$ solves \eqref{eq:SDEXmild}, its first component $X^1$ solves equation \eqref{eq:SDEx} and $X^1_s=M_sX_s$ for every $s\in[t,T]$. 
In the following we will study the forward equation both in $\sD$ and in $\sL^2$; this means that we will consider drift operators $B$ defined on $[0,T]\times\sL^2$ or on $[0,T]\times\sD$, depending on the occasion.
When needed, the solution to \eqref{eq:SDEX} will be denoted by $X^{t,x}$ to stress the dependence on initial data.

\begin{remark}
  \label{rem:conv}
If we denote by $\sZ^t$ the stochastic convolution
\begin{equation*}
  \sZ^{t}(s)=\intt\etsa\Sigma\ud W_r\ ,\quad s\geq t\ ,
\end{equation*}
then $s \mapsto \sZ^{t}(s)$ is a continuous process with values in $\aC$ for every $t \in [0,T]$ and  $\E\left[\norm{\sZ^{t}(s)}_{\aC}^p\right]\lesssim (s-t)^{\frac{p}{2}}$ for every $p\geq 2$. 
Thanks to the continuity of the embedding $\aC\subset\sL^2$, the same properties hold in $\sL^2$ as well. 
From the explicit form of the semigroup it can be easily seen that $X_s^{t,x}$ belongs to $\aC$ whenever $x\in\aC$, whereas it only belongs to $\sD$ if the path $x\in\sD$ is discontinuous at some point. 
We refer to \cite{flandoli2016infinite} for a detailed discussion.
\end{remark}

In \cite{flandoli2016infinite} it was shown that, under some regularity assumptions, the function
\begin{equation*}
  u(t,x)=\bE\left[\Phi\left(X^{t,x}_T\right)\right]
\end{equation*}
is a regular solution of the linear Kolmogorov backward equation
\begin{equation}\label{eq:kolm:diff}
\begin{system}
\frac{\partial u}{\partial t}(t,x) + Du(t,x)\left[Ax + B(t,x)\right]   + \frac{1}{2}\tr_{\bR^d}\left[\Sigma\Sigma^\ast D^2u(t,x)\right] = 0, \\
u(T,\cdot) = \Phi
\end{system}
\end{equation} 
where $\Phi:\sD\to\bR$ is a given terminal condition, $Du(t,x)\left[Ax + B(t,x)\right]$ is the duality pairing between $\sD^\prime$ and $\sD$ and the trace term is defined as
\begin{equation*}
  \tr_{\bR^d}\left[\Sigma\Sigma^\ast D^2v(t,x)\right] = \sum_{j=1}^d\Sigma\Sigma^\ast D^2v(t,x)\left(e_j,e_j\right)
\end{equation*}
for an arbitrary orthonormal system $\left\{e_j\right\}_{j=1}^d$ of $\bR^d$.\\
Here we are interested in the nonlinear version of (\ref{eq:kolm:diff}) given by
\begin{equation}
  \tag{\ref{eq:kolm:diff:nonlinear}}
\begin{system}
  \frac{\partial u}{\partial t}(t,x) + Du(t,x)\left[Ax + B(t,x) \right] + \frac{1}{2}\tr_{\bR^d}\left[\Sigma\Sigma^\ast D^2u(t,x)\right] = G(t,x,u(t,x),Du(t,x)\Sigma) \\
  u(T,\cdot) = \Phi \ ,
\end{system}
\end{equation} 
where $G:[0,T]\times \sD \times \bR\times \bR^{d_1}\to\bR$.
\begin{definition}
\label{def:sol_D}
Given $\Phi\in C\left(\sD,\mathbb{R}\right)  $, we say that $u:\left[  0,T\right]  \times \sD\rightarrow\mathbb{R}$ is a \emph{classical solution} of the Kolmogorov semilinear backward equation with terminal condition $\Phi$ if
\begin{equation*}
u\in C^{1;2}\left([0,T]\times \sD,\mathbb{R}\right)\ ,
\end{equation*}
 and satisfies identity (\ref{eq:kolm:diff:nonlinear}) for every $t\in\left[  0,T\right]  $ and $x\in \aC^1$.
\end{definition}
To find a classical solution to the semilinear Kolmogorov backward equation we introduce the following real-valued BSDE:
\begin{equation}\label{eq:BSDEint}
Y^{t,x}_s + \int_s^T Z^{t,x}_r\ud W_r = -\int_s^T G(r,X^{t,x}_r,Y^{t,x}_r,Z^{t,x}_r) \ud r + \Phi(X^{t,x}_T) \ , \qquad t\leq s \leq T \ ,
\end{equation}
where the notation $(\cdot)^{t,x}$ refers to the initial data $(t,x)$ of the forward equation. In a differential formulation, we are concerned with the forward-backward system of the form
\begin{equation}
\label{eq:FBSDE}
\begin{system}
\ud X^{t,x}_s = \left[ AX_s^{t,x} + B(s,X_s^{t,x}) \right]\ud s + \Sigma \ud W_s \\
\ud Y^{t,x}_s = G(s,X_s^{t,x},Y^{t,x}_s,Z^{t,x}_s) \ud s + Z^{t,x}_s \ud W_s \\
X^{t,x}_t = x \\
Y^{t,x}_T = \Phi(X^{t,x}_T)
\end{system}
\end{equation}
where $s \in [t,T] \subset [0,T]$.
Our goal is to show that the function
  \begin{equation*}
    u(t,x)=Y^{t,x}_t\ ,
  \end{equation*}
  is a classical solution of the semilinear Kolmogorov backward equation with terminal condition $\Phi$. Since the scheme we follow consists in first solving the PDE in $\sL^2$ and then passing to $\sD$, we will need to study the forward-backward system and the PDE in both these spaces. For this reason we will state some results and assumptions in a general separable Banach space $E$, specialyzing to the cases $E=\sL^2$, $\sL^p$, $\sD$, $\sC$, $\aC$ when needed.

\begin{remark}[Path-dependent case] 
When $G$ and $\Phi$ are infinite-dimensional lifting of path-dependent functions, i.e.
  \begin{equation*}
    G(s,x,y,z)=g_s\left(M_sx,y,z\right)\quad\text{and}\quad  \Phi(x)=\phi\left(M_Tx\right)
  \end{equation*}
for a family $\{g_s\}_{s\in[0,T]}$, $g_s:D([0,s];\bR^d)\times\bR\times\bR^{d_1}\to\bR$ and a map $\phi:D([0,T];\bR^d)\to\bR$ (cf. \eqref{eq:M_t} ), the PDE \eqref{eq:kolm:diff:nonlinear} can be interpreted as the Kolmogorov PDE associated to the path-dependent forward-backward system
\begin{equation*}
  \begin{system}
    \ud \xi_s=b_s\left(\xi_{[0,s]}\right)\ud s+\sigma\ud W_s\\
    \ud \psi_s =g_s\left(\xi_{[0,s]},\psi_s,\zeta_s\right)\ud s+\zeta_s\ud W_s\\
    \xi_{[0,t]}=\gamma_{[0,t]}\\
    \psi_T=\phi\left(\xi_{[0,T]}\right)\ .
  \end{system}
\end{equation*}
In this specific situation the PDE actually has the form
  \begin{equation}\label{PDE-path-dep}
    \frac{\partial u}{\partial t}(t,x)+Du(t,x)Ax=G(t,x,u(t,x),Du(t,x)\Sigma)- Du(t,x)B(t,x)-\frac{1}{2}\tr_{\bR^d}\left[\Sigma\Sigma^\ast D^2u^\ast\right] ,
  \end{equation}
where the r.h.s depends on $Du(t,x)\in\sD^\prime$ only through its action on the first components of elements in $\sD$.
Moreover, exploiting the so-called \emph{functional differential calculus} introduced in \cite{cont2013functional} one can formulate a PDE very similar to \eqref{PDE-path-dep} for which a wellposedness result can also be provided by our approach.
We refer again to \cite{flandoli2016infinite} and \cite{zanco2015phd} for a detailed discussion about the relations between the two settings and the role played by the operator $\frac{\partial}{\partial t}+D[\cdot]A$. 
\end{remark}
In the following, we essentially provide a solution theory for semilinear PDEs on the space of continuous functions under the assumption that the second order term concentrates on the final dimensional component of $\sD$ (thus ensuring the trace term be well defined) and without requiring that coefficients arise as liftings of path-dependent functions.
For these reasons,  in the whole presentation we will consider general coefficients, without sticking to the path-dependent formalism.

\subsection{Assumptions}
\label{subsec:assumptions}
Let $E$ be any of the spaces listed in \eqref{eq:spaces} and  let $m  \ge 0$. 
The following sets of assumptions are in force throughout the paper.

\begin{assumption}\label{ass:B}
The drift term $B$ belongs to $C^{1;2,\alpha}([0,T]\times E;E))$ for some $\alpha \in (0,1)$, with the H\"older norm of $D^2B$ bounded uniformly in $s \in [0,T]$.
Furthermore, there exists a constant $C \ge 0$ such that 
\begin{enumerate}[label=$(B.\Roman{*})$]
\item $\left\vert B(s,x)\right\vert \leq C\left(1+\left\vert x\right\vert\right)$,
\item $\left\vert B\left(s,x_1\right)-B\left(s,x_2\right)\right\vert \leq C\left\vert x_1-x_2\right\vert$,
\item $\left\vert D^2B(s,x)\right\vert \leq C\left(1+\left\vert x\right\vert^m\right)$,
\end{enumerate}
for every $x,x_1,x_2 \in E$, uniformly in $s \in [0,T]$.
\end{assumption}

For what concerns the coefficients of the BSDE \eqref{eq:BSDE}, hereinafter we use the notation $D_i G$ to denote the derivative with respect to the $i$-th (spatial) entry of the map $(x,y,z) \mapsto G(s,x,y,z)$, and $D_{i,j}^2 G$ for the second derivatives

\begin{assumption}\label{ass:G}
$G: [0,T] \times E \times \bR \times \bR^{d_1} \to \bR$ is such that  for every $s \in [0,T]$ the map  $G(s,\cdot) \in C^{2,\alpha}(E \times \R \times \R^{d_1}; \R)$.
Moreover there exists $C \ge 0$ such that :
\begin{enumerate}[label=$(G.\Roman{*})$]
\item\label{item:F1} $\abs{G(s,x,y,z)}\leq C(1+\abs{x}^{m}+\abs{y}+\abs{z})$;

\item $\abs{G(s,x,y_1,z_1)-G(s,x,y_2,z_2)}\leq C\left(\abs{y_1-y_2}+\abs{z_1-z_2}\right)$;
\item $\abs{D_1G(s,x,y,z)} + \abs{D^2_{1,1} G(s,x,y,z)} \leq C(1+\abs{x}^m)(1+\vert y\vert+\vert z\vert)$;
\item $\abs{D^2_{1,2} G(s,x,y,z)} + \abs{D^2_{1,3} G(s,x,y,z)}\leq C\left(1+\abs{x}^{m}\right)(1+\vert y\vert)$;
\item $\abs{D^2_{2,2} G(s,x,y,z)} + \abs{D^2_{2,3} G(s,x,y,z)} \leq C(1+\vert x \vert^m)$;
\item $\abs{D^2_{3,3} G(s,x,y,z)} \leq C$,
\end{enumerate}
for every $x,y,z, y_1,y_2,z_1,z_2 \in E$,  uniformly in $s \in [0,T]$.
\end{assumption}

\begin{assumption}\label{ass:Phi2}
The function $\Phi$ belongs to $C^{2,\alpha}(E,\bR)$ for some $\alpha \in (0,1)$ and 
\begin{equation*}
|\Phi(x)| + |D\Phi(x)| + |D^2\Phi(x)| \leq C(1 + |x|^m).
\end{equation*}
\end{assumption}

In Section \ref{sec:PDED}, when passing from the PDE in $\sL^2$ to the PDE in $\sD$ we will need to carefully approximate the coefficients.
In doing so, it is crucial that if $B$, $G$, $\Phi$ satisfy the above assumptions in $\sD$ then the same could hold for the approximations $B^n$, $G^n$, $\Phi^n$  in $\sL^2$, possibly with some uniformity with respect to $n$. It turns out that a convenient way to build such approximations is to consider a sequence of bounded linear operators $J^n$ from $\sL^2$ to $\sC$ with the following properties:

\begin{itemize}
\item $J^nx\to x$ in $\sC$ for every $x\in\sC$;
\item $\sup_n\left\VV J^n x\right\VV_\infty\leq C_J\VV x\VV_\infty$ for every $x\in\sD$ such that $M_T(x)$ has at most one jump and is continuous elsewhere, where $M_T$ is defined in \eqref{eq:M_t}.
\end{itemize}
Note that any such sequence converges to the identity uniformly on compact sets of $\sC$.\\
An example of  $\{J_n\}$ can be constructed as follows: given any $\epsilon\in\left(0,\frac{T}{2}\right)$ define a function $\tau_\epsilon:[-T,0]\to[-T,0]$ as
\begin{equation*}
  \tau_\epsilon(x)=\begin{cases}-T+\epsilon &\mbox{if } x\in[-T,-T+\epsilon]\\x&\mbox{if } x\in[-T+\epsilon,-\epsilon]\\-\epsilon&\mbox{if }x\in[-\epsilon,0]\ .
  \end{cases}
\end{equation*}
Then choose any function $\rho \in C^{\infty}(\bR;\bR)$ such that $\VV\rho\VV_1=1$, $0\leq\rho\leq 1$,  $\supp(\rho)\subseteq[-1,1]$ and define a sequence $\left\{\rho_n\right\}$ of mollifiers by $\rho_n(x):=n\rho(nx)$. Set, for any $\phi\in L^1(-T,0;\bR^d)$
\begin{equation}
\label{eq:Jn}
  \sJ^n\phi(x):=\int_{-T}^0\rho_n\big(\tau_{\frac{1}{n}}(x)-y\big)\phi(y)\ud y \ ;
\end{equation}
finally set
\begin{equation*}
  J^n\xphi{a}{\phi}=\left(\begin{matrix}a\\ \sJ^n\phi\end{matrix}\right)\ .
\end{equation*}
To ensure the applicability of the limiting procedure, we need one more assumption, that is satisfied by many examples as discussed in \cite{flandoli2016infinite} and \cite{zanco2015phd}.
\begin{definition}
\label{def:phi}
Let $F$ be a Banach space, $R\colon\sD\to F$ twice Fr\'echet differentiable and $\Gamma\subseteq\sD$. We say that $R$ has \emph{one-jump-continuous Fr\'echet differentials of first and second order on $\Gamma$} if there exists a sequence of linear continuous operators $J^n$ as above such that for every $y\in\Gamma$ and for almost every $a\in[-T,0]$ the following hold:
  \begin{equation*}
    DR(y)J^n\xphi{1}{\ind_{[a,0)}}\longrightarrow DR(y)\xphi{1}{\ind_{[a,0)}}\ \text{,}
  \end{equation*}
  \begin{equation*}
    D^2R(y)\Big(J^n\xphi{1}{\ind_{[a,0)}}-\xphi{1}{\ind_{[a,0)}},\xphi{1}{\ind_{[a,0)}}\Big)\longrightarrow 0 \ , \quad    D^2R(y)\Big(\xphi{1}{\ind_{[a,0)}},J^n\xphi{1}{\ind_{[a,0)}}-\xphi{1}{\ind_{[a,0)}}\Big)\longrightarrow 0 \ ,
  \end{equation*}
  \begin{equation*}
    D^2R(y)\Big(J^n\xphi{1}{\ind_{[a,0)}}-\xphi{1}{\ind_{[a,0)}},J^n\xphi{1}{\ind_{[a,0)}}-\xphi{1}{\ind_{[a,0)}}\Big)\longrightarrow 0\ \text{,}
  \end{equation*}
where we adopt the convention that $\xphi{1}{\ind_{[a,0)}}=\xphi{1}{0}$ when $a=0$.
\end{definition}
We will call \emph{smoothing sequence} for $R$ any sequence $\left\{J^n\right\}$ satisfying the above requirements.
By linearity, the above convergences hold true also if $\xphi{1}{\ind_{[a,0)}}$ is substituted with any $x\in\sD$ with the property that $M_T(x)$ has at most one jump and it is continuous elsewhere.
\begin{assumption}
\label{ass:dphi}
For every $s\in[0,T]$, $B(s,\cdot)$ and $\Phi$ have one-jump-continuous Fr\'echet differentials of first and second order on $\aC \subset \sD$ and the smoothing sequence of $B$ does not depend on $s$.
\end{assumption}
\begin{assumption}
  \label{ass:GJ}
  For every $s\in[0,T]$, $y\in\bR$, $z\in\bR^{d_1}$, $G(s,\cdot,y,z)$ has one-jump continuous Fr\'echet differential of first order and its smoothing sequence does not depend on $s$ nor on $y,z$.

\end{assumption}

\section{The forward-Backward system}
\label{sec:FBSDE}

This section is devoted to the forward-backward system \eqref{eq:FBSDE} (FBSDE in the following), that we write below in mild formulation for the reader's convenience:
\begin{equation}\label{FB_mild}
\begin{system}
X^{t,x}_s=e^{(s-t)A}x_{[0,t]}+\intt e^{(s-r)A} B(r,X^{t,x}_r)\ud r+\int_t^s e^{(s-r)A} \Sigma \ud W_r \\
\\
Y^{t,x}_s + \int_s^T Z^{t,x}_r\ud W_r = -\int_s^T G(r,X^{t,x}_r,Y^{t,x}_r,Z^{t,x}_r) \ud r + \Phi(X^{t,x}_T) \ ,\\
\end{system}
\end{equation}
where $t\leq s \leq T$.
Observe that the system is not fully coupled: the forward equation does not depend on the values of the pair $(Y,Z)$.
We firstly state some result for the process $X$, whose proof can be found in \citep[Thms.~2.2, 2.3, 2.4]{flandoli2016infinite}

\begin{proposition}\label{p:existence_X}
Under Assumption \ref{ass:B}, there exists a set $\Omega_0\subseteq\Omega$ of full probability such that:
\begin{enumerate}[label=$(\roman{*})$]
\item\emph{(existence)} for every initial data $(t,x)\in[0,T]\times E$ and every $\omega\in\Omega_0$ , equation (\ref{eq:SDEXmild}) admits a unique solution $(s,\omega)\mapsto X^{t,x}_s(\omega) \in E$ which 
is continuous in time if $E=\sL^2$, while it 
is only bounded in time if $E=\sD$; 
\item\emph{(regularity in space)} for every $\omega\in\Omega_0$, $t\in[0,T]$ and $s\in[t,T]$ the map $x\mapsto X^{t,x}_s(\omega)$ is in $C^{2,\alpha}$;
\item\emph{(regularity in time)} if $E=\sL^2$, for every $s\in[0,T]$, $x\in E$ and $\omega\in\Omega_0$ the map $t\mapsto X^{t,x}_s(\omega)$ ($t\leq s$) is continuous; if $E=\sD$ the same property holds whenever $x\in\aC$;
\item\emph{(Markovianity)} if $E=\sL^2$ the solution $X^{t,x}$ has the markov property.
\end{enumerate}
\end{proposition}

From now on we will denote by $\Omega_0 \subseteq \Omega$ the fixed set given by Proposition \ref{p:existence_X}.

\begin{theorem}
  \label{thm:X}
Assume that $B:[0,T]\times E\to E$ satisfies Assumption \ref{ass:B}. Fix a time $t\in[0,T]$ and a $\sF_t$-measurable $E$-valued random variable $\xi$, and let $X_\cdot^{t,\xi}$ be the unique $E$-valued solution to
  \begin{equation}
    \label{eq:Xlemma}
    X_s=e^{(s-t)A}\xi+\int_{t}^s \etsa B(r,X_r)\ud r+\int_{t}^s\etsa\Sigma\ud W_r\ .
  \end{equation}
For any $p\geq 1$, if $\xi$ has finite $p$-th moment then $X^{t,\xi}\in L^p\left(\Omega;C([t,T];E)\right)$ and
\begin{equation}
\label{eq:estX}
  \bE\sup_{s\in[t,T]}\left\vert X_s\right\vert_E^p\leq c_1\left(1+\bE\left\vert\xi\right\vert^p\right)
\end{equation}
When $\xi = x\in E$ is deterministic, for every $t\in[0,T]$ the map $x\mapsto X_\cdot^{t,x}$ is twice Fr\'echet differentiable as a map from $E$ to $L^p\left(\Omega;C([t,T];E)\right)$ with continuous differentials; the $L(E;E)$-valued process $D_xX_\cdot^{t,x}$ is the unique solution to
\begin{equation}
\label{eq:DX1}
  \Xi_s=e^{(s-t)A}+\int_{t}^se^{(s-r)A}DB(r,X_r^{t,x})\Xi_r\ud r\ ,
\end{equation}
while the $L(E,E;E)$-valued process $D^2_xX_\cdot^{t,x}$ is the unique solution to
\begin{equation}
  \label{eq:DX2}
  \Theta_s=\int_{t}^se^{(s-r)A}D^2B(r,X_r^{t,x})\left(D_xX_r^{t,x}, D_xX_r^{t,x}\right)\ud r+\int_{t}^se^{(s-r)A}DB(r,X_r^{t,x})\Theta_r\ud r\ .
\end{equation}
All the three SDEs above can be solved path-by-path, meaning that for any fixed $\omega\in\Omega_0$ there exist unique functions $s\mapsto X_s^{t,\xi}(\omega)$, $s\mapsto D_xX_s^{t,x}(\omega)$ and $s\mapsto D^2_xX_s^{t,x}(\omega)$ that satisfy (\ref{eq:Xlemma}), (\ref{eq:DX1}) and (\ref{eq:DX2}), respectively.
Moreover
\begin{equation}
  \label{eq:estDXnorm}
  \sup_{s\in[t,T]}\left\Vert D_xX_s^{t,x}\right\Vert_{L(E;E)}\leq c_2 \quad \text{ for a.e. } \omega \in \Omega\ 
\end{equation}
and in particular for any $E$-valued random variable $\eta\in L^p(\Omega;E)$
\begin{equation}
\label{eq:estDX}
\bE\sup_{s\in[t,T]}\left\vert D_xX_s^{t,x}\eta\right\vert^p\leq c_2\bE\left\vert\eta\right\vert^p\ .
\end{equation}
Furthermore
\begin{equation}
  \label{eq:estD2Xnorm}
  \sup_{s\in[t,T]}\left\Vert D^2_xX_s^{t,x}\right\Vert_{L(E,E;E)}\leq c_3\left(1+\left\vert x\right\vert^{m}\right) \quad \text{ for a.e. } \omega \in \Omega\ .
\end{equation}
The constants $c_1,c_2,c_3$ in the inequalities above depend only on $m$, $T$, $D^i B$ with $i = 0,1,2$, and on the constant $C$ in Assumption \ref{ass:B}.
\end{theorem}
\begin{proof}
Using that $\sZ^{t}_s$ is a $E$-valued martingale, by \cite[Thm.~3.9]{daprato2014stochastic} and Remark \ref{rem:conv}, we have that 
\begin{equation*}
  \bE\sup_{s\in[t,T]}\left\vert\sZ^{t}_s\right\vert_E^p\leq C \sup_{s\in[t,T]}\bE\left\vert \sZ^{t}_s\right\vert_E^p \leq CT^{\frac{p}{2}}\ .
\end{equation*}
Therefore, from the uniform estimates on $e^{tA}$ and Assumption \ref{ass:B} we get that for every $t\leq R\leq T$
\begin{equation*}
  \bE\sup_{s\in[t,R]}\left\vert X_s\right\vert^p\lesssim 1+\bE\left\vert\xi\right\vert^p + \int_{t}^R\bE\left(\sup_{s\in[t,r]}\left\vert X_s\right\vert^p\right)\ud r \ ,
\end{equation*}
from which (\ref{eq:estX}) follows thanks to Gronwall's lemma. 
Furthermore, the proof of the Fr\'echet differentiability of the map $x\mapsto X_s^{t,x}(\omega)$ given in \cite{flandoli2016infinite} can be easily extended to the required differentiability of $x\mapsto X^{t,x}$ in the space of $E$-valued processes. 
Well-posedness of \eqref{eq:DX1} and \eqref{eq:DX2} (and the fact that $D_xX^{t,x}$ ad $D^2_xX^{t,x}$ are the required solutions) has been already established in \cite{flandoli2016infinite}. Estimates (\ref{eq:estDXnorm}), (\ref{eq:estDX}) and (\ref{eq:estD2Xnorm}) are then easy consequences of Assumption \ref{ass:B}.
 \end{proof}

For what concerns the Backward SDE in \eqref{FB_mild}, the following wellposedness result has been given in \cite{fuhrman2002nolinear}.
 
\begin{proposition}\label{p:stimeBSDE}
Under Assumptions \ref{ass:B}, \ref{ass:G} and \ref{ass:Phi2}, for every $(t,x) \in [0,T] \times E$, the BSDE in \eqref{FB_mild} admits a unique solution $(Y,Z)\in \sK_p$, for every $p \in [2, +\infty)$. Moreover, the map $(t,x) \mapsto \left( Y^{t,x}_\cdot, Z^{t,x}_\cdot \right)$ belongs to $C([0,T] \times E; \sK_p)$ and
there exists $c\geq 0$ such that 
\begin{equation}\label{eq:est:bsde_Y_Z}
 \E \sup_{s \in [t,T]} |Y^{t,x}_s|^p  + \E\left(\int_t^T | Z^{t,x}_r |^2 \ud r \right)^{\nicefrac{p}{2}} \leq c\left( 1 + \left| x \right|^{pm} \right).
\end{equation}
\end{proposition}

\begin{remark}
The constant $c \geq 0$ appearing in \eqref{eq:est:bsde_Y_Z} can be chosen independently of $(t,x)$. The same applies to Propositions \ref{t.diff_gateaux_BSDE}, \ref{p:diff1} and \ref{p.diff_2_BSDE} below.
Alternatively, one could set $(Y_r, Z_r) = 0$ for every $r \in [0,t]$.
\end{remark}

\subsection{First-order differentiability of the BSDE}
\label{subsec:first}

Here we investigate the differentiability of the map $x \mapsto \left( D_xY^{t,x}, D_xZ^{t,x} \right)$.
G\^ateaux differentiability has been established in a Hilbert setting  in \cite{fuhrman2002nolinear} and then extended to a general Banach setting in \cite{masiero2008stochastic} and \cite{masiero2014hjb}. 
Our aim is to show that under Assumptions \ref{ass:G} and \ref{ass:Phi2}, also Fr\'echet differentiability takes place.  

Let us firstly lighten the notation introducing the shorthand
\begin{equation*}
D_iG_r(t,x) := D_iG(r,X^{t,x}_r,Y^{t,x}_r,Z^{t,x}_r)\ , \qquad i = 1,2,3 \ ,
\end{equation*}
and consider the backward equation satisfied by the pair $\left( U^{t,x}, V^{t,x} \right)$:
\begin{equation}\label{der1_U_BSDE}
\begin{split}
  U^{t,x}_s h +\int_s^T V^{t,x}_r h\ud W_r &= U^{t,x}_T h  -\int_s^T D_1G_r(t,x)D_xX^{t,x}_r h \ud r\\
  &-\int_s^T \left( D_2G_r(t,x) U^{t,x}_r h + D_3G_r(t,x) V^{t,x}_r h \right) \ud r \ ,\\
\end{split}
\end{equation}
where the terminal condition is given by $U^{t,x}_T h = D\Phi(X^{t,x}_T)D_xX^{t,x}_Th$.
It turns out that \eqref{der1_U_BSDE} admits a unique solution $(U^{t,x}h,V^{t,x}h)$  which is given by the directional derivatives $(D_xY^{t,x}h,D_xZ^{t,x}h)$, for every $h \in E$.
This is the content of the next Proposition, whose proof can be found in \cite[Prop.~4.8]{fuhrman2002nolinear}.

\begin{proposition}\label{t.diff_gateaux_BSDE}
Let Assumptions \ref{ass:B}, \ref{ass:G} and \ref{ass:Phi2} hold true. 
For every $h\in E$ equation \eqref{der1_U_BSDE} admits a unique solution $(U^{t,x} h, V^{t,x} h) =(D_xY^{t,x}h,D_xZ^{t,x}h)$.
Moreover, for every $p >1$ the map $(t,x) \mapsto \left(Y^{t,x}, Z^{t,x} \right)$ is G\'ateaux differentiable as a map from $[0,T] \times E$ to $\sK_p$ and for every $h\in E$ the directional derivatives $(D_xY^{t,x}h,D_xZ^{t,x}h)$ satisfy the BSDE \eqref{der1_U_BSDE}:
\begin{equation}\label{der1BSDE}
\begin{split}
  D_xY^{t,x}_s h&+\int_s^T D_xZ^{t,x}_r h\ud W_r =  D\Phi(X^{t,x}_T)D_xX^{t,x}_Th -\int_s^TD_1G_r(t,x)D_xX^{t,x}_r h\ud r\\
  &\phantom{=}-\int_s^T \left(D_2G_r(t,x)D_xY^{t,x}_rh  + D_3G_r(t,x)D_xZ^{t,x}_r h \right)\ud r.\\ 
\end{split}
\end{equation}

Finally, for every $(t,x) \in [0,T] \times E$, the following estimate holds true
\begin{equation}\label{eq:estBSDEh}
  \left[\E\sup_{s\in[t,T]}\abs{D_xY^{t,x}_sh}^p\right]^{\nicefrac{1}{p}}+\left[\E\left(\int_t^T\abs{D_xZ^{t,x}_r h}^2\ud r\right)^{\nicefrac{p}{2}}\right]^{\nicefrac{1}{p}}\leq C\abs{h}\left(1+\abs{x}^{m^2}\right)\ .
\end{equation}
\end{proposition}

We are now in position to study the Fr\'echet differentiability of the maps $t,x \mapsto \left( D_xY^{t,x}, D_xZ^{t,x} \right)$. 
\begin{proposition}\label{p:diff1}
Under Assumptions \ref{ass:B}, \ref{ass:G} and \ref{ass:Phi2}, the map $x \mapsto \left( Y^{t,x}_\cdot, Z^{t,x}_\cdot \right)$ (resp. $t \mapsto \left( Y^{t,x}_\cdot, Z^{t,x}_\cdot \right)$) is Fr\'echet differentiable as a map from $E$ (resp. $[0,T]$) to $\sK_p$. 
Moreover the following estimate holds true
\begin{equation}\label{eq:estBSDE}
  \left[\E\sup_{s\in[t,T]}\left\|D_xY^{t,x}_s\right\|^p\right]^{\nicefrac{1}{p}}+\left[\E\left(\int_t^T\left\|D_xZ^{t,x}_r \right\|^2\ud r\right)^{\nicefrac{p}{2}}\right]^{\nicefrac{1}{p}}\leq C\left(1+\abs{x}^{m^2}\right)\ .
\end{equation}

\end{proposition}

Before entering the details of the proof let us briefly comment on the crucial role played by estimate \eqref{estimate_linear}.
In taking the differences $\left( D_xY^{t,x} - D_xY^{t,y}, D_xZ^{t,x} - D_xZ^{t,y} \right)$ (hence comparing solutions whose forward process starts at different points), we inevitably end up with the term
\[ \int_s^T\left[D_3G_r(t,x)-D_3G_r(t,y)\right]D_xZ^{t,y}_rh\ud r, \]
leading to the product $\left(Z_r^{t,x} - Z_r^{t,y}\right)D_xZ^{t,y}_rh$ which does not belong to $ L^p(\Omega; L^2(0,T;\R))$. 
In this situation standard methods are not effective. 
Nonetheless, the minimal integrability requirement in Lemma \ref{l:linearBSDE} allows to treat with simple tools (see estimate \eqref{est_lambda_31}) the following term 
\[\bE\left(\int_s^T\left|Z_r^{t,x} - Z_r^{t,y}\right| \left|D_xZ^{t,y}_rh \right|\ud r\right)^p. \]

\begin{proof}[Proof of Proposition \ref{p:diff1}]
To shorten the proof we concentrate only on the Fr\'echet differentiability of the map $x \mapsto \left( Y_\cdot^{t,x}, Z_\cdot^{t,x} \right)$, for every $t \in [0,T]$. Differentiability in time follows by the very same technique (see e.g. \citep{masiero2014hjb} for what concerns differentiability in the G\^ateaux sense).

The strategy of the proof is as follows: 
by Theorem \ref{t.diff_gateaux_BSDE} we deduce that the pair $(Y^{t,x},Z^{t,x})$ is G\^ateaux differentiable with respect to $x$. 
Then we show the continuity of $x \mapsto (D_xY^{t,x},D_xZ^{t,x})$ as a map from $E$ to $L(E ; \sK_p)$, 
which easily yields the required Fr\'echet differentiability.
To do it, we write the equation for the differences  $D_xY^{t,x}h - D_xY^{t,y}h$, $D_xZ^{t,x}h - D_xZ^{t,y}h$ emphasizing its linear character. 
We  employ estimates \eqref{estimate_linear} and we show that the r.h.s. vanishes as $|x-y|_E \to 0$, uniformly in $h \in E$, $|h|_E \leq 1$.

\medskip

Given $x,y,h \in E$, let us write the equation for the differences $\big(D_xY_s^{t,x}h-  D_xY_s^{t,y}h \big)$, $\big(D_xZ_s^{t,x}h - D_xZ_s^{t,y}h\big)$:
\begin{equation*}
  \begin{aligned}
    [&D_xY^{t,x}_s - D_xY^{t,y}_s]h+\int_s^T\left[D_xZ^{t,x}_r-D_xZ^{t,y}_r\right]h\ud W_r 
    =\left[D\Phi\left(X^{t,x}_T\right)D_xX^{t,x}_T-D\Phi\left(X_T^{t,y}\right)D_xX_T^{t,y}\right]h \\
    &\phantom{=}-\int_s^T\left[D_1G_r(t,x)D_xX^{t,x}_r-D_1G_r(t,y)D_xX^{t,y}_r\right]h\ud r 
    - \int_s^T\left[D_2G_r(t,x)D_xY^{t,x}_r-D_2G_r(t,y)D_xY^{t,y}_r\right]h\ud r\\
    &\phantom{=}-\int_s^T\left[D_3G_r(t,x)D_xZ^{t,x}_r-D_3G_r(t,y)D_xZ^{t,y}_r\right]h\ud r\\
	&=\left[D\Phi\left(X^{t,x}_T\right)D_xX^{t,x}_T-D\Phi\left(X_T^{t,y}\right)D_xX_T^{t,y}\right]h 
	-\int_s^T\left[D_1G_r(t,x)D_xX^{t,x}_r-D_1G_r(t,y)D_xX^{t,y}_r\right]h\ud r   \\
    &\phantom{=}-\int_s^TD_2G_r(t,x)\left[ D_xY^{t,x}_r-D_xY^{t,y}_r\right]h\ud r 
    - \int_s^TD_3G_r(t,x)\left[ D_xZ^{t,x}_r-D_xZ^{t,y}_r\right]h\ud r\\
    &\phantom{=}-\int_s^T\left[D_2G_r(t,x)-D_2G_r(t,y)\right]D_xY^{t,y}_rh\ud r - \int_s^T\left[D_3G_r(t,x)-D_3G_r(t,y)\right]D_xZ^{t,y}_rh\ud r \ .\\
  \end{aligned}
\end{equation*}
If we define 
\begin{gather}
  \Delta Y_r=\left(D_xY^{t,x}_s - D_xY^{t,y}_s\right)h\ ,\quad  \Delta Z_r=\left(D_xZ^{t,x}_r-D_xZ^{t,y}_r\right)h\  \nonumber,\\
  \xi=\left[D\Phi\left(X^{t,x}_T\right)D_xX^{t,x}_T-D\Phi\left(X_T^{t,y}\right)D_xX_T^{t,y}\right]h \ , \nonumber \\
  \lambda(r)= -\lambda_1(r)-\lambda_2(r)-\lambda_3(r) =  - \left[D_1G_r(t,x)D_xX^{t,x}_r-D_1G_r(t,y)D_xX^{t,y}_r\right]h \label{eq:lambda}\\
  \phantom{++++}- \left[D_2G_r(t,x)-D_2G_r(t,y)\right]D_xY^{t,y}_rh 
  - \left[D_3G_r(t,x)-D_3G_r(t,y)\right]D_xZ^{t,y}_rh \ , \nonumber\\
V_s=\int_t^s\left\vert D_2G_r(t,x) \right\vert\ud r+\frac{1}{1\wedge(p-1)}\int_t^s\left\vert D_3G_r(t,x)\right\vert^2\ud r \ ,\nonumber
\end{gather}
the above equation reads
\begin{equation*}
 \Delta Y_s+\int_s^T\Delta Z_r\ud W_r=\xi - \int_s^T \left( D_2G_r(t,x)\Delta Y_r + D_3G_r(t,x)\Delta Z_r - \lambda(r)\ud r \right) \ud r \ .
\end{equation*}
where $D_2G_r(t,x)$ and $D_3G_r(t,x)$ are bounded processes with values in $\R$ and $\R^{d_1}$, respectively. 
Since $V$ is a bounded process as well, estimate \eqref{estimate_linear} in Lemma \ref{l:linearBSDE} guarantees that
\begin{equation*}
  \bE\sup_{s\in[t,T]}\left\vert e^{V_s}\Delta Y_s\right\vert^p+\bE\left(\int_t^Te^{2V_r}\left\vert\Delta Z_r\right\vert^2\ud r\right)^{\nicefrac{p}{2}}\lesssim \bE\left\vert \xi\right\vert^p+\bE\left(\int_t^T\left\vert\lambda(r)\right\vert\ud r\right)^p,
\end{equation*}
and the desired continuity follows as soon as
\begin{equation}\label{con_frechet1}
\sup_{\substack{h \in E, \\ |h|_E \leq 1}} \left[\bE\left\vert \xi\right\vert^p+\bE\left(\int_t^T\left\vert\lambda(r)\right\vert\ud r\right)^p \right] \longrightarrow 0, \qquad \text{ if } |x -y|_E \to 0.
\end{equation}
Let us start by showing the convergence for the first term in \eqref{con_frechet1}.
For $p>1$
\begin{equation*}
  \begin{split}
    \bE \left\vert \xi\right\vert^p
    &\lesssim \bE\left\vert D\Phi\left(X^{t,x}_T\right)D_xX^{t,x}_Th-D\Phi\left(X_t^{t,y}\right)D_xX^{t,y}_Th\right\vert^p \\
    &\lesssim \bE\left\vert D\Phi\left(X^{t,x}_T\right) \left( D_xX^{t,x}_Th-D_xX^{t,y}_Th \right)\right\vert^p\\
    &\phantom{=}+\bE\left\vert \left( D\Phi\left(X^{t,x}_T\right) - D\Phi\left(X^{t,y}_T\right) \right) D_xX^{t,y}_Th\right\vert^p\\
    &= \bE\left(\left|\xi_1\right|^p+\left|\xi_2\right|^p\right)\ .
  \end{split}
\end{equation*}
Using Assumption \ref{ass:Phi2} we have
\begin{equation*}
  \begin{split}
    \left| \xi_1\right|\leq\left\| D\Phi\left(X^{t,x}_T\right)\right\|\left\vert D_xX^{t,x}_Th-D_xX^{t,y}_Th\right\vert\lesssim \left(1+\sup_{s\in[t,T]}\left\vert X^{t,x}_s\right\vert^m\right)\left\vert D_xX^{t,x}_Th-D_xX^{t,y}_Th\right\vert \ ,
  \end{split}
\end{equation*}
and  thanks to estimates \eqref{eq:estX}, \eqref{eq:estDX} and the Fr\'echet differentiability of $x \mapsto X_s^{t,x}(\omega)$ for every $\omega \in \Omega_0$, $s \in [t,T]$ (see Proposition \ref{p:existence_X}) it holds
\begin{equation*}
\begin{split}
  \sup_{\substack{h \in E, \\ |h|_E \leq 1}} \bE\left|\xi_1\right|^p &\lesssim \sup_{\substack{h \in E, \\ |h|_E \leq 1}} \left[\bE\left(1+\sup_{s\in[t,T]}\left\vert X^{t,x}_s\right\vert^{2mp}\right)\right]^{\nicefrac{1}{2}}\left[\bE\left\| D_xX^{t,x}_T - D_xX^{t,y}_T \right\|_{L(E;E)}^{2p}\right]^{\nicefrac{1}{2}}|h|^p \\
  &\lesssim \left[\bE\left\| D_xX^{t,x}_T - D_xX^{t,y}_T \right\|_{L(E;E)}^{2p}\right]^{\frac{1}{2}} \longrightarrow 0 \ , \qquad \text{ if } |x -y|_E \to 0 \ .
\end{split}
\end{equation*}
A similar argument yields
\begin{equation*}
\begin{split}
 \sup_{\substack{h \in E, \\ |h|_E \leq 1}} \E \left|\xi_2\right|^p &\leq \sup_{\substack{h \in E, \\ |h|_E \leq 1}} \left[\E\left\| D\Phi(X_T^{t,x})-D\Phi(X^{t,y}_T)\right\|^{2p}\right]^{\nicefrac{1}{2}} \left[\left\| D_xX^{t,y}_T\right\|_{L(E;E)}^{2p} \right]^{\nicefrac{1}{2}} |h|^p \\
 &\lesssim \left[\E\left\| D\Phi(X_T^{t,x})-D\Phi(X^{t,y}_T)\right\|^{2p}\right]^{\nicefrac{1}{2}} \longrightarrow 0 \ , \qquad \text{ if } |x -y|_E \to 0 \ ,
\end{split}
\end{equation*}
where we employed Vitali convergence theorem.
More precisely, Fr\'echet differentiability of the map  $x \mapsto X_T^{t,x}(\omega)$ for every  $\omega \in \Omega_0$, along with the continuity of $D\Phi$, guarantees the convergence of $\| D\Phi(X_T^{t,x})-D\Phi(X^{t,y}_T) \| \to 0$ ; whereas Assumption \ref{ass:Phi2} combined with estimate \eqref{eq:estX} and the choice of a determintistic initial condition $x \in \E$, ensure the uniform integrability of $\left\| D\Phi(X_T^{t,x})-D\Phi(X^{t,y}_T)\right\|^{2p}$.    
\smallskip

\noindent Concerning the second term in \eqref{con_frechet1} we treat separately the three processes $\lambda_i$ in \eqref{eq:lambda}.
First we write
  \begin{equation*}
    \begin{split}
\left\vert\lambda_1(r)\right\vert&=\vert D_1G_r(t,x)D_xX^{t,x}_rh-D_1G_r(t,y)D_xX^{t,y}_rh\vert\\
&\lesssim \left\vert D_1G_r(t,x)\left(D_xX^{t,x}_rh-D_xX^{t,y}_rh\right)\right\vert +\left\vert \left[D_1G_r(t,x)-D_1G_r(t,y)\right]D_xX^{t,y}_rh\right\vert\\
      &=|\lambda_{11}(r)|+ |\lambda_{12}(r)| \ .
    \end{split}
  \end{equation*}
Assumption \ref{ass:G} ensures that 
\begin{equation*}
  \begin{split}
    |\lambda_{11}(r)|&\leq\left\| D_1G_r(t,x)\right\| \left| D_xX^{t,x}_rh-D_xX^{t,y}_rh\right|\\
    &\lesssim \left(1+\sup_{r\in[t,T]}\left\vert X^{t,x}_r \right\vert^m\right)\left(1+\sup_{r\in[t,T]}\left\vert Y^{t,x}_r\right\vert 
    + \left| Z^{t,x}_r \right| \right)\left\vert D_xX^{t,x}_rh-D_xX^{t,y}_rh\right\vert,
  \end{split}
\end{equation*}
and from estimates \eqref{eq:estX} and Proposition \ref{p:stimeBSDE} we get 
\begin{equation*}
  \begin{split}
   \bE\left(\int_t^T |\lambda_{11}(r)|\ud r \right)^p\lesssim \left(1+\left\vert x\right\vert ^{4pm}\right)|h|^p\left[\bE\left(\int_t^T\left\| D_xX^{t,x}_r-D_xX^{t,y}_r\right\|_{L(E;E)}^2\ud r\right)^p\right]^{\nicefrac{1}{2}}
  \end{split}
\end{equation*}
which converges to zero as $|x-y|_E \to 0$ uniformly with respect to $h$, $\| h\| \leq 1$, thanks to the Fr\'echet character of the map $x \mapsto X^{t,x}$  and the Lebesgue dominated convergence theorem (recall estimate \eqref{eq:estDXnorm}).

Then we have
  \begin{equation*}
    \begin{split}
      |\lambda_{12}(r)|&\leq\left\vert \left[ D_1 G\left(r,X^{t,x}_r,Y^{t,x}_r,Z^{t,x}_r\right)-D_1G\left(r,X^{t,x}_r,Y^{t,x}_r,Z^{t,y}_r\right) \right]D_xX^{t,y}_rh\right\vert\\
      &\phantom{\leq}+\left\vert \left[ D_1 G\left(r,X^{t,x}_r,Y^{t,x}_r,Z^{t,y}_r\right)-D_1G\left(r,X^{t,x}_r,Y^{t,y}_r,Z^{t,y}_r\right) \right]D_xX^{t,y}_rh\right\vert\\
      &\phantom{\leq}+\left\vert\left[ D_1 G\left(r,X^{t,x}_r,Y^{t,y}_r,Z^{t,y}_r\right)-D_1G\left(r,X^{t,y}_r,Y^{t,y}_r,Z^{t,y}_r\right) \right]D_xX^{t,y}_rh\right\vert\\
      &=|\lambda_{121}(r)|+|\lambda_{122}(r)|+|\lambda_{123}(r)|.
    \end{split}
  \end{equation*}
Assumption \ref{ass:G} yields
  \begin{equation*}
    \begin{split}
      |\lambda_{121}(r)|&=\left\vert\int_0^1 D_{1,3}^2G\left(r,X^{t,x}_r,Y^{t,x}_r,\alpha Z^{t,x}_r+(1-\alpha)Z^{t,y}_r\right)\Big(Z^{t,x}_r-Z^{t,y}_r,D_xX^{t,y}_rh\Big)\ud \alpha\right\vert\\
      &\lesssim\left(1+\sup_{r\in[t,T]}\left\vert X^{t,x}_r\right\vert^m\right)\left(1+\sup_{r\in[t,T]}\left\vert Y^{t,x}_r\right\vert\right)\left(\sup_{r\in[t,T]}\left\| D_xX^{t,y}_r\right\|\right)|h|\left\vert Z^{t,x}_r-Z^{t,y}_r\right\vert.
    \end{split}
  \end{equation*}
If we apply Holder inequality, Theorem \ref{thm:X} and Proposition \ref{p:stimeBSDE} we get that, as $|x-y|_E \to 0$,  
\begin{equation*}
  \begin{split}
   \sup_{\substack{h \in E, \\ |h|_E \leq 1}} \bE\left(\int_t^T|\lambda_{121}(r)|\ud r\right)^p\lesssim\left(1+\vert x\vert^{mp}\right)\left(1+\vert x\vert^{pm^2}\right)\left[\bE\left(\int_t^T\left\vert Z^{t,x}_r-Z^{t,y}_r \right\vert^2 \ud r\right)^{2p}\right]^{\frac{1}{4}}\longrightarrow 0 \ .
  \end{split}
\end{equation*} 
The same strategy also applies to the terms $\lambda_{122}(r)$ and $\lambda_{123}(r)$.
Therefore, uniformly in $h \in E$, $|h|_E \leq 1$:
\begin{equation*}
\bE\left(\int_t^T\beta_1(r)\ud r\right)^p \leq \bE\left(\int_t^T|\lambda_{11}(r)|\ud r\right)^p + \sum_{i=1}^3\bE\left(\int_t^T|\lambda_{12i}(r)|\ud r\right)^p \longrightarrow 0 \ ,  \qquad \text{ if } |x-y|_E \to 0 \ .
\end{equation*}

We proceed in a similar way for the term $\lambda_2$:
\begin{equation*}
  \begin{split}
    \vert\lambda_2(r)\vert&\leq \left\vert D_2G\left(r,X^{t,x}_r,Y^{t,x}_r,Z^{t,x}_r\right)-D_2G\left(r,X^{t,x}_r,Y^{t,x}_r,Z^{t,y}_r\right)\right\vert \left\vert D_xY^{t,y}_rh\right\vert\\
    &\phantom{aa}+\left\vert D_2G\left(r,X^{t,x}_r,Y^{t,x}_r,Z^{t,y}_r\right)-D_2G\left(r,X^{t,x}_r,Y^{t,y}_r,Z^{t,y}_r\right)\right\vert \left\vert D_xY^{t,y}_rh\right\vert\\
    &\phantom{aa}+\left\vert D_2G\left(r,X^{t,x}_r,Y^{t,y}_r,Z^{t,y}_r\right)-D_2G\left(r,X^{t,y}_r,Y^{t,y}_r,Z^{t,y}_r\right)\right\vert \left\vert D_xY^{t,y}_rh\right\vert\\
    &=|\lambda_{21}(r)|+|\lambda_{22}(r)|+|\lambda_{23}(r)|\ .
  \end{split}
\end{equation*}

Then
\begin{equation*}
  \begin{split}
    \bE\left(\int_t^T |\lambda_{21}(r)|\ud r\right)^p&=\bE\left(\int_t^T\left| D_xY^{t,y}_rh\right| \int_0^1\left\vert D^2_{2,3}G\left(r,X^{t,x}_r,Y^{t,x}_r,a Z^{t,x}_r+(1-a)Z^{t,y}_r\right)\right\vert\left\vert Z^{t,x}_r-Z^{t,y}_r\right\vert\ud a\ud r\right)^p\\
    &\lesssim \left[\bE\sup_{r \in [t,T]}\left\vert D_xY^{t,y}_rh\right\vert^4\right]^{\nicefrac{p}{4}}\left[1+\bE\sup_{r \in [t,T]}\left\vert Y^{t,x}_r\right\vert^{4m}\right]^{\nicefrac{p}{4}}\left[\bE\int_0^T\left\vert Z^{t,x}_r-Z^{t,y}_r\right\vert^2\ud r\right]^{\nicefrac{p}{2}}\\
  &\lesssim \left\vert h\right\vert^p\bE\left(\int_t^T\left\vert Z^{t,x}_r-Z^{t,y}_r\right\vert^2\ud r\right)^{\nicefrac{p}{2}}\longrightarrow 0 \ , \qquad \text{ if } |x-y|_E\to 0 \ ,
  \end{split}
\end{equation*}
uniformly with respect to $h \in E$, $|h|_E \leq 1$.
The terms $\lambda_{22}(r)$ and $\lambda_{23}(r)$ can be treated in the same manner, so that the required convergence holds for $\lambda_2(r)$ as well.\\
It remains to check the term $\lambda_3(r)$:
\begin{equation*}
  \begin{split}
    \vert\lambda_3(r)\vert&\leq \left\vert D_3G\left(r,X^{t,x}_r,Y^{t,x}_r,Z^{t,x}_r\right)-D_3G\left(r,X^{t,x}_r,Y^{t,x}_r,Z^{t,y}_r\right)\right\vert \left\vert D_xZ^{t,y}_rh\right\vert\\
    &\phantom{aa}+\left\vert D_3G\left(r,X^{t,x}_r,Y^{t,x}_r,Z^{t,y}_r\right)-D_3G\left(r,X^{t,x}_r,Y^{t,y}_r,Z^{t,y}_r\right)\right\vert \left\vert D_xZ^{t,y}_rh\right\vert\\
    &\phantom{aa}+\left\vert D_3G\left(r,X^{t,x}_r,Y^{t,y}_r,Z^{t,y}_r\right)-D_3G\left(r,X^{t,y}_r,Y^{t,y}_r,Z^{t,y}_r\right)\right\vert \left\vert D_xZ^{t,y}_rh\right\vert\\
    &=|\lambda_{31}(r)|+|\lambda_{32}(r)|+|\lambda_{33}(r)| \ .
  \end{split}
\end{equation*}
Exploiting again Assumption \ref{ass:G} we easily derive the required result for each term $\lambda_{3i}(r)$, $i = 1,2,3$.
We present here the estimate involving the increments $Z^{t,x}_r-Z^{t,y}_r$:  
\begin{equation}\label{est_lambda_31}
  \begin{split}
    \bE\left(\int_t^T|\lambda_{31}(r)|\ud r\right)^p&\leq\bE\left(\int_t^T\left\vert D_xZ^{t,y}_rh\right\vert\int_0^1\left\vert D_{3}^2G\left(r,X^{t,y}_r,Y_r^{t,x},a Z_r^{t,x}+(1- a)Z^{t,y}_r\right)\right\vert\left\vert Z^{t,x}_r-Z^{t,y}_r\right\vert\ud a\ud r\right)^p\\
    &\lesssim \left[\bE\left(\int_t^T\left\vert D_xZ^{t,y}_rh\right\vert^2\ud r\right)^p\right]^{\nicefrac{1}{2}}\left[\bE\left(\int_t^T\left\vert Z^{t,x}_r-Z^{t,y}_r\right\vert^2\ud r\right)^p\right]^{\nicefrac{1}{2}}\\
    &\lesssim \left\vert h\right\vert^p\left(1+\left\vert x\right\vert ^{m^2p}\right)\left[\bE\left(\int_t^T\left\vert Z^{t,x}_r-Z^{t,y}_r\right\vert^2\ud r\right)^p\right]^{\nicefrac{1}{2}}
  \end{split}
\end{equation}
and the above term converges to zero thanks to Proposition \ref{p:stimeBSDE}.

Summing up, we have that 
\begin{equation*}
\sup_{\substack{h \in E, \\ |h|_E \leq 1}} \bE\left(\int_t^T\left\vert\lambda(r)\right\vert e^{ V_r}\ud r\right)^p \lesssim \sum_{i=1}^3 \sup_{\substack{h \in E, \\ |h|_E \leq 1}}  \bE\left(\int_t^T\left\vert\lambda_i(r)\right\vert\ud r\right)^p \longrightarrow 0, \qquad \text{ if } |x-y|_E \to 0,
\end{equation*}
from which we get the required continuity.
For what concerns the estimate \eqref{eq:estBSDE} it simply follows from \eqref{eq:estBSDEh} by taking the supremum in $h \in E$, $|h|_E \leq 1$.
\end{proof}

Let us now give a representation result for the solution $Z^{t,x}$ of \eqref{eq:BSDEint}, in terms of the Fr\'echet differential of the map $x \mapsto Y^{t,x}$.
This will be crucial in the following, e.g. for the second-order Fr\'echet differentiability of the map $x \mapsto (Y^{t,x},Z^{t,x})$.

\begin{proposition}
\label{p:identification}
Let Assumptions \ref{ass:B}, \ref{ass:G} and \ref{ass:Phi2} be in force. 
Given the solution  $(Y^{t,x},Z^{t,x})$ of the BSDE in \eqref{FB_mild}, we define the map $u: [0,T] \times E \to \R$ as $u(t,x):= Y^{t,x}_t$.
Then for every $(t,x) \in [0,T] \times E$ it holds that
\begin{equation}
Y^{t,x}_s = u(s, X^{t,x}_s)  \ ,\quad s\in[t,T]
\end{equation}
and, by \eqref{eq:est:bsde_Y_Z}, there exists $c \geq 0$ such that 

\begin{equation}
\sup_{t\in[0,T]}|u(t,x)| \leq c(1 + |x|^m) \ .
\end{equation}
Moreover, by Proposition \ref{p:diff1} the map $x \mapsto u(\cdot, x)$ is differentiable and, denoting by $Du(t,x)$ its gradient with respect to the second variable, we have
  \begin{equation}
    \label{eq:stima_grad_u}
\sup_{t\in[0,T]}\left| Du(t,x) \right| \leq c \left( 1+ |x|^{m^2} \right) \ .
\end{equation}
Finally, for every $(t,x) \in [0,T] \times E$ we have the identification 
\begin{equation}\label{eq:identification}
Z^{t,x}_\cdot = DY^{t,X^{t,x}_\cdot}_\cdot\Sigma= Du (\cdot, X^{t,x}_\cdot) \Sigma \qquad \text{ in } L^p(\Omega; L^2(0,T)) \ ,
\end{equation}
where $DY^{t,X^{t,x}_\cdot}_\cdot= D_y Y_\cdot^{t,y} \vert_{y = X_\cdot^{t,x}}$.
In particular for every $s\in[t,T]$
\begin{equation*}
  Z_s^{t,x}=\lim_{r\downarrow s}D_xY^{s,X^{t,x}_s}_r\ .
\end{equation*}
Notice that \eqref{eq:identification} identifies a specific version $\tilde Z^{t,x} \in L^p(C([0,T]; \R^{d_1}))$ of $Z^{t,x}$. 
This identification will hold throughout the paper and clearly yields  
\begin{equation}\label{eq:stimaZ}
\E \left( \sup_{s \in [t,T]} \left| \tilde Z^{t,x}_s \right|^p \right) < +\infty \ .
\end{equation}
\end{proposition} 
\begin{proof}
For a proof of this result we refer to \cite[Cor.~6.29]{fabbri2017stochastic} for the Hilbert setting and to \cite{masiero2008stochastic}, \cite{Zhang2011optimal} for the extension to Banach spaces.  
\end{proof}

\begin{remark}\label{rem:BSDE:bdd}
Let us comment on the particular case in which $D\Phi$ and $D_iG$, $i =1,2,3$, in \eqref{der1BSDE} are uniformly bounded.
By a standard application of Girsanov theorem, see e.g. \cite[Section~6.7.1]{fabbri2017stochastic},  estimate \eqref{eq:estDXnorm} yields the boundedness of $D_xY^{t,x} h$ in the sense that there exists $K \ge 0$ such that for every $h \in E$
\begin{equation}
\label{bound_DY}
\left| D_x Y^{t,x}_s h \right| \le K|h|, \qquad \text{for a.e. } \omega \in \Omega, \ \forall t\le s\le T, \ \forall x \in E.
\end{equation}
Hence, in view of  Proposition \ref{p:identification},
\begin{equation}
\label{bound_Z}
\left| Z^{t,x}_s \right| \le K|\Sigma|, \qquad \text{for a.e. } \omega \in \Omega, \ \forall t\le s\le T, \ \forall x \in E,
\end{equation}
where the constant $K$ only depends on $\sup_x \|D\Phi(x)\|$, $\sup_{s,x,y,z} \left( |D_1G(s,x,y,z)| + |D_2G(s,x,y,z)| \right)$ but \emph{not} on $D_3G(s,x,y,z)$, 
This is crucial in the application to stochastic optimal control in Section \ref{sec:control}.    
\end{remark}

\subsection{Second-order differentiability of the BSDE}
\label{subsec:second}
From the previous section we know that Fr\'echet derivatives $(D_xY^{t,x},D_xZ^{t,x})$ are well-defined
and the pair  $(D_xY^{t,x}h,D_xZ^{t,x}h)$ solves equation \eqref{der1BSDE} for every $h \in E$.
By exploiting the form of the equation, here we firstly study the Fr\'echet differentiability of the directional derivatives
$x \mapsto (D_xY^{t,x}h,D_xZ^{t,x}h)$, for every $h \in E$ fixed.
Then, using the uniform character of all the estimates, we identify the pair $(D_x^2Y^{t,x},D_x^2Z^{t,x})$ as the second-order Fr\'echet differential of the map $x \mapsto (Y^{t,x},Z^{t,x})$. 
Similarly to the previous subsection we will use the shorthand
\begin{equation*}
  D^2_{i,j}G_r(t,x):=D^2_{i,j}G\left(r,X^{t,x}_r,Y^{t,x}_r,Z^{t,x}_r\right)\ .
\end{equation*}
For every $h,k \in E$, let us introduce the backward equation
\begin{equation}\label{BSDE_second}
\begin{split}
  F^{t,x}_s&(k,h)+\int_s^TH^{t,x}_r(k,h)\ud W(r)= F^{t,x}_T(k,h) + \int_s^T L_r^{t,x}(k,h) \ud r\\
  &-\int_s^T \left[ D_2G_r(t,x)F^{t,x}_r(k,h) + D_3G_r(t,x)H^{t,x}_r(k,h) \right]\ud r \ ,
\end{split}
\end{equation}
where we used the notation 
\[F^{t,x}_T(k,h) := D^2\Phi\left(X^{t,x}_T\right)\left(D_xX^{t,x}_Tk,D_xX^{t,x}_Th\right) + D\Phi(X^{t,x}_T)D^2_xX^{t,x}_T(k,h) \ ; \] 
\begin{equation*}
\begin{split}
L&_r^{t,x}(k,h) := - D^2_{1,1}G_r(t,x)\left(D_xX^{t,x}_rk,D_xX^{t,x}_rh\right) - D^2_{1,2}G(r,x)\left(D_xY^{t,x}_rk,D_xX^{t,x}_rh\right) \\
&- D^2_{1,3}G_r(t,x)\left(D_xZ^{t,x}_rk,D_xX^{t,x}_rh\right)- D_1G_r(t,x)D^2_xX^{t,x}_r(k,h) \\
&- D^2_{2,1}G_r(t,x)\left(D_xX^{t,x}_rk,D_xY^{t,x}_rh\right) - D^2_{2,2}G_r(t,x)\left(D_xY^{t,x}_rk,D_xY^{t,x}_rh\right) -D^2_{2,3}G_r(t,x)\left(D_xZ^{t,x}_rk,D_xY^{t,x}_rh\right) \\
&- D^2_{3,1}G_r(t,x)\left(D_xX^{t,x}_rk,D_xZ^{t,x}_rh\right) - D^2_{3,2}G_r(t,x)\left(D_xY^{t,x}_rk,D_xZ^{t,x}_rh\right) - D^2_{3,3}G_r(t,x)\left(D_xZ^{t,x}_rk,D_xZ^{t,x}_rh\right) \\
&\qquad \quad \;\,=: L_{1;r}^{t,x}(k,h) + L_{2;r}^{t,x}(k,h) + L_{3;r}^{t,x}(k,h)
\end{split}
\end{equation*}
and $L_{i;r}^{t,x}$ includes all the terms containing $D_{i}G$ or $D^2_{i,j}G$, for any $j$.
The main result of the section is  the following 

\begin{proposition}\label{p.diff_2_BSDE}
Let Assumptions \ref{ass:B}, \ref{ass:G} and \ref{ass:Phi2} hold true. 
For every $t \in [0,T]$, $h,k\in E$, equation \eqref{BSDE_second} admits a unique solution $\left(F^{t,x}_s(k,h), H^{t,x}_s (k,h)\right)$.
For every $p >1$ the map $x \mapsto \left(D_xY^{t,x}_{\cdot}h, D_xZ^{t,x}_\cdot h \right)$ (resp. $t \mapsto \left(D_xY^{t,x}_{\cdot}h, D_xZ^{t,x}_\cdot h \right)$) is G\^ateaux differentiable as a map from $E$ (resp. $[0,T]$) to $\sK_p$.
For every $(k,h)\in E$ the pair $\big(D^2_xY^{t,x}_s(k,h)$, $D^2_xZ^{t,x}_s(k,h)\big)$ satisfies the BSDE \eqref{BSDE_second}.
Moreover, for every $t \in [0,T]$ and $p >1$, the map $x \mapsto (Y^{t,x}_\cdot, Z^{t,x}_\cdot)$ is twice Fr\'echet differentiable as a map from $E$ to $\sK_p$ with second order Fr\'echet differential given by $(D^2_x Y^{t,x}, D^2_xZ^{t,x})$ and 
\begin{equation}\label{eq:est_D2YZ}
  \left[\E\sup_{s\in[t,T]}\left\| D^2_xY^{t,x}_s\right\|^p\right]^{\nicefrac{1}{p}}+\left[\E\left(\int_t^T\left\|D^2_xZ^{t,x}_s\right\|^2\ud r\right)^{\nicefrac{p}{2}}\right]^{\nicefrac{1}{p}}\leq c\left(1+\abs{x}^{l}\right)\ ,
\end{equation}
for some $c,l \ge 0$.
\end{proposition}

\begin{proof} 
For what concerns wellposedness of \eqref{BSDE_second}, let us check that $F_T^{t,x}(k,h)$ and $L_r^{t,x}(k,h)$ satisfy the integrability conditions given in Lemma \ref{linear_BSDE}. 
The application of H\"older inequality along with Assumption \ref{ass:Phi2} and Theorem \ref{thm:X} immediately give that $F_T^{t,x}(k,h) \in L^p_{\mathcal{F}_T}(\Omega;\R)$.
To prove that $L^{t,x}(k,h)$  belongs to $L^p(\Omega; L^1(0,T;\R))$, for every $p>1$, we profit from the growth of $G$ (see Assumption \ref{ass:G}) and the estimates on $D_xX_r, D_xY_r, D_xZ_r,$ given in Theorem \ref{thm:X} and Proposition \ref{t.diff_gateaux_BSDE}, respectively.
Let us give some details for the term $L_{1;r}^{t,x}(h,k)$:
\begin{equation*}
\begin{split}
\int_t^T&\left| D^2_{1,1}G_r(t,x)\left(D_xX^{t,x}_rk,D_xX^{t,x}_rh\right) \right|  \ud r\\
&\lesssim (1 + \sup_{s \in [t,T]}|X^{t,x}_s|^m) (1 +  \sup_{s \in [t,T]}|Y^{t,x}_r| )\sup_{s \in [t,T]}\left\|D_xX^{t,x}_s\right\|^2 |h| |k| \int_t^T  |Z^{t,x}_r| \ud r. \\
\end{split}
\end{equation*} 
Using H\"older inequality and the estimates recalled above we get boundedness in $L^p(\Omega; L^1(0,T;\R))$, as required. 
The other terms in $L^{t,x}(h,k)$ can be treated in a similar way.

To prove Fr\'echet differentiability,
fixing $h,k \in E$ and using the equations solved by
$D_xY^{t,x+k}_s h,  D_xY^{t,x}_s h$ and $F^{t,x}_s(k,h)$ it can be easily shown that
\begin{equation}\label{eq:upsilon}
\Upsilon^k_s+\int_s^T\Psi^k_r\ud W_r=\Upsilon^k_T- \int_s^T \left( D_2G_r(t,x)\Upsilon^k_r + D_3G_r(t,x)\Psi^k_r - M^k(r)\ud r \right) \ud r \ ,
\end{equation}
where 
\begin{equation}\label{notation_first_der}
\begin{gathered}
\Upsilon_r^k: = \frac{1}{|k|}\left[ D_xY^{t,x+k}_s h - D_xY^{t,x}_s h - F^{t,x}_s (h,k) \right], \quad \Psi_r^k: = \frac{1}{|k|} \left[ D_xZ^{t,x+k}_s h - D_xZ^{t,x}_s h - H^{t,x}_s (h,k)\right] \\
\Upsilon_T^k:= \frac{1}{|k|} \left[U_T^{t,x+k} h - U_T^{t,x} h - F^{t,x}_T(k,h) \right]\\
M^k:= M^k_1 + M^k_2 +M^k_3:= -\frac{1}{|k|}\left[ D_1G_r(t,x +k)D_xX^{t,x+k}_r h - D_1G_r(t,x)D_xX^{t,x}_r h  -  L_{1;r}^{t,x}(k,h)\right] \\
-\frac{1}{|k|}\left[ \left(D_2G_r(t,x+k) - D_2G_r(t,x) \right)D_xY^{t,x+k}_r h  - L_{2;r}^{t,x}(k,h) \right] \\ 
 -\frac{1}{|k|}\left[ 	\left(D_3G_r(t,x+k) - D_3G_r(t,x) \right)D_xZ^{t,x+k}_r h  - L_{3;r}^{t,x}(k,h) \right]. \\
\end{gathered}
\end{equation} 
Taking advantage from the linear character of \eqref{eq:upsilon}, thanks to   estimate \eqref{estimate_linear} and recalling that in this case $V$ is a bounded process, we have that 
\begin{equation*}
  \bE\sup_{s\in[t,T]}\left\vert \Upsilon^k_s\right\vert^p+\bE\left(\int_t^T\left\vert\Psi^k_r\right\vert^2\ud r\right)^{\nicefrac{p}{2}}\lesssim \bE\left\vert \Upsilon^k_T\right\vert^p+\bE\left(\int_t^T\left\vert M^k(r)\right\vert\ud r\right)^p.
\end{equation*}
The desired Fr\'echet differentiability follows as soon as
\begin{equation*}
\lim_{k \to 0} \sup_{|h|=1}\left[\bE\left\vert \Upsilon^k_T\right\vert^p+\bE\left(\int_t^T\left\vert M^k(r)\right\vert\ud r\right)^p \right] =0.
\end{equation*}
A detailed computation of all the terms is postponed in the Appendix. 
Here we only show how to deal with the most delicate one, which we denote by $M^k_{311}$ to be consistent with the notation of the appendix,
\begin{equation*}
M^k_{311}:= 
-\frac{1}{|k|}\int_0^1 \left[ D^2_{3,3} G(r,X_r^{t,x+k}, Y_r^{t,x+k}, \lambda Z_r^{t,x+k} + (1-\lambda)Z_r^{t,x}) - D^2_{3,3}G(r,X_r^{t,x+k}, Y_r^{t,x+k}, Z_r^{t,x})\right],
\end{equation*}
and where the application of Proposition \ref{p:identification} seems to be crucial.
Using the notation $u(t,x):= Y^{t,x}_t$, from $\alpha$-H\"older continuity of $D^2_3G_r(t,x)$ we get, on a set of full probability,
\begin{equation*}
\begin{split}
\int_t^T| & M^k_{311}(r)| \ud r\lesssim \frac{1}{|k|}
\int_t^T \left|Z_r^{t,x+k} - Z_r^{t,x} \right|^{\alpha} \left|Z_r^{t,x+k} - Z_r^{t,x} \right| \left\| D_xZ^{t,x+k}_r\right\| |h| \ud r \\
&\lesssim   \sup_{r \in [t,T]} \left\| Du(r,X_r^{t,x+k}) - Du(r,X_r^{t,x}) \right\|^{\alpha}  \left| \Sigma \right| |h| 
\left( \int_t^T \frac{\left| Z^{t,x+k}_r - Z^{t,x}_r\right|^2}{|k|^2} \ud r \right)^{\nicefrac{1}{2}}
\left( \int_t^T \left\| D_xZ^{t,x+k}_r\right\|^2 \ud r \right)^{\nicefrac{1}{2}}.\\  
\end{split}
\end{equation*}
To show that $\E \left(\int_t^T |M^k_{311}(r)| \ud r \right)^p \to 0$ we employ Vitali convergence theorem.
Taking advantage of the fact that the initial datum $x$ is deterministic,  
from \eqref{eq:estBSDE} there exists $l \ge 0$ such that for any $p' > p$
\begin{equation*}
\begin{split}
\E \left(\int_t^T |M^k_{311}(r)| \ud r \right)^{p'}\lesssim  \left(1 + |x+k|^l + |x|^l\right)\left| \Sigma \right| |h|
\left[ \E\left( \int_t^T \frac{ \left| Z^{t,x+k}_r - Z^{t,x}_r\right|^2}{|k|^2} \ud r \right)^{2p'} \right]^{\nicefrac{1}{4}} < +\infty \ , 
\end{split}
\end{equation*}
which is bounded thanks to Proposition \ref{p:diff1} and estimate \eqref{eq:estBSDE}. 
This guarantees the uniform integrability of the familiy $\E \left(\int_t^T |M^k_{311}(r)| \ud r \right)^p$, when $k$ is varying.
Hence, it remains to show that for a.e. $\omega \in \Omega$ and a.e $r \in [0,T]$
\begin{equation}\label{conv_Du_Sigma}
\left| Du(r,X_r^{t,x+k})\Sigma -  Du(r,X_r^{t,x})\Sigma \right| \longrightarrow 0 \ , \qquad \text{ if } |k| \to 0 \ .
\end{equation}
To do it, recall the general continuity result for $y \mapsto D_xY^{t,y}h$ as a map from $E$ to $L^2(\Omega; C([0,T]; \R))$ given in the proof of  Proposition \ref{p:diff1}.
When dealing with $u(t,x) = Y^{t,x}_t$, which is deterministic, this implies that for every $y_1,y_2,h \in E$ 

\begin{equation}\label{conv_Du_h}
|Du(t,y_1)h - Du(t,y_2)h| \longrightarrow 0 \ , \qquad \text{ as } |y_1 - y_2| \to 0 \ .
\end{equation}
Now, given a basis $e_1, \ldots, e_{d_1}$ in $\R^{d_1}$, \eqref{conv_Du_Sigma} is equivalent to the convergence
\[ \sup_{j \in \lbrace 1, \ldots, d_1\rbrace} \left| Du(r,X_r^{t,x+k})\Sigma e_j -  Du(r,X_r^{t,x})\Sigma e_j \right|\longrightarrow 0 \ , \qquad \text{ if } |k| \to 0 \ .  \]
where $\Sigma e_j \in E$ for every $j = 1,\ldots,{d_1}$. 
Hence combining the continuity of the map $x \mapsto X^{t,x}_s(\omega)$ for every $\omega \in \Omega_0$, $s \in [t,T]$ (see Theorem \ref{thm:X}) with the convergence result in \eqref{conv_Du_h} we easily get \eqref{conv_Du_Sigma}.
For a detailed proof of the convergence of  all the remaining terms we refer to the Appendix.

To conclude the proof observe that estimate \eqref{eq:est_D2YZ} is a direct consequence of Lemma \ref{l:linearBSDE} applied to (\ref{BSDE_second}). Indeed, for every $k,h \in E$, uniqueness of solutions to \eqref{BSDE_second} gives that  $F^{t,x}_s(k,h) = D^2_xY^{t,x}_s(k,h)$ and $H^{t,x}_s(k,h) = D^2_xZ^{t,x}_s(k,h)$.  
\end{proof}
\begin{remark}
The proof of (\ref{conv_Du_Sigma}) as it is performed here exploits the fact that noise is additive in the forward equation.
\end{remark}
  \begin{corollary}
    \label{cor:D2u}
Setting $u(t,x):= Y^{t,x}_t$ as in  Proposition \ref{p:identification}, the map $x \mapsto u(t,x)$ belongs to $C^2(E;\R)$, for every $t \in [0,T]$. 
 Moreover, for some $c \ge 0$, $l\in\bN$ it holds that
 \begin{equation*}
   \sup_{t\in[0,T]}\left\vert D^2u(t,x)\right\vert\leq c\left(1+\left\vert x\right\vert^l\right) \ .
 \end{equation*}
  \end{corollary}

\section{Solution to the Kolmogorov equation in $\sL^2$}
\label{sec:PDEL2}

In this section we deal with the infinite-dimensional semilinear PDE \eqref{eq:kolm:diff:nonlinear} in the space $\sL^2$. Therefore we assume all the coefficients to be defined on $\sL^2$ and that assumptions \ref{ass:B} and \ref{ass:G} are satisfied with $E=\sL^2$. This is a quite strong requirement that is seldom satisfied by examples; however it represents only the first step towards establishing the theory for $E=\aC$, where the same assumptions are much more reasonable and indeed verified by a large class of examples.

\begin{theorem}
\label{thm:L2}
Let Assumptions \ref{ass:B}, \ref{ass:G} and \ref{ass:Phi2} hold with $E=\sL^2$. Then the function $u(t,x):=Y^{t,x}_t$
is a classical solution, in the sense of Definition \ref{def:sol_D}, to the semilinear Kolmogorov equation
(\ref{eq:kolm:diff:nonlinear}). 
\end{theorem}

\begin{proof}
Thanks to the regularity results given in Section \ref{sec:FBSDE} we know that the map $(t,x)\mapsto u(t,x)$ belongs to $C^{1;2}([0,T] \times \mathcal{L}^2)$.
Hence, it is enough to prove that $u(t,x):= Y^{t,x}_t$ is a solution of the semilinear Kolmogorov equation in integral form:
\begin{multline}
  \label{eq:kolm:int:nonlinear}
  u(t,x)-\Phi(x)+\int_t^TG(s,x,u(s,x),Du(s,x)\Sigma)\ud s\\=\int_t^T\left[Du(s,x)\left[Ax+B(s,x)\right]+\frac{1}{2}\sum_{j=1}^d \Sigma \Sigma^*D^2u(s,x)\left(e_j,e_j\right)\right]\ud s\ .
\end{multline}
The standard way of proving such a result goes through an application of It\^o formula to the increments of $u(t,X)$ along a partition; eventually taking expectations, summing along the partition and letting the size of the mesh going to $0$ yields the result. The difficulty here lies in the fact that at every time $t$, $X(t)$ lies almost surely not in the domain of the operator $A$. There are different ways to circumvent this difficulty; we detail here one of the possibilities.\\
Consider two time instants $0\leq t_0\leq t_1\leq T$ and a point $x\in\sL^2$. We want to analyse the increment
\begin{align}
  \nonumber u(t_0,x)-u\left(t_1,e^{(t_1-t_0)A}x\right)&=\E u(t_0,x)-u\left(t_1,e^{(t_1-t_0)A}x\right)\\
\label{eq:u_incr1}  &=\E Y^{t_0,x}_{t_0}-\E Y^{t_0,x}_{t_1}+\E Y^{t_0,x}_{t_1}-u\left(t_1,e^{(t_1-t_0)A}x\right)\ .
\end{align}
Thanks to the Markov property of $X^{t_0,x}$ it is not difficult to show (see \citep{fuhrman2002nolinear}) that almost surely
\begin{equation*}
  Y^{t_0,x}_t=Y^{t_1,X^{t_0,x}_{t_1}}_t\ ,\quad Z^{t_0,x}_t=Z^{t_1,X^{t_0,x}_{t_1}}_t
\end{equation*}
for every $t\in[t_1,T]$, hence
\begin{equation*}
  \E Y^{t_0,x}_{t_1}=\E Y^{t_1,X^{t_0,x}_{t_1}}_{t_1}=\E u\left(t_1,X^{t_0,x}_{t_1}\right)
\end{equation*}
and (\ref{eq:u_incr1}) yields
\begin{equation}
\label{eq:u_incr2}
u(t_0,x)-u\left(t_1,e^{(t_1-t_0)A}x\right)=\E\left[ Y^{t_0,x}_{t_0}-Y^{t_0,x}_{t_1}\right]+\E\left[ u\left(t_1,X^{t_0,x}_{t_1}\right)-u\left(t_1,e^{(t_1-t_0)A}x\right)\right]\  .
\end{equation}
Since $Y$ satisfies the BSDE (\ref{eq:BSDEint}), the first expectation on the r.h.s. can be written as
\begin{equation*}
  \E\left[-\int_{t_0}^{t_1}G\left(r,X^{t_0,x}_r,Y^{t_0,x}_r,Z^{t_0,x}_r\right)\ud r\right]\ .
\end{equation*}
Now fix $t\in[0,T]$ and consider a sequence $\br{\pi^n}$ of partitions of $[t,T]$ such that each of the $\pi^n$'s is given by $k_n+1$ points $t=t^n_1\leq t^n_2\leq\dots\leq t^n_{k_n+1}=T$ and such that $\abs{\pi^n}\to 0$ as $n\to\infty$. For each fixed $n$ and every $i=1,\dots,k_n+1$ we consider (\ref{eq:u_incr2}) with $t_0=t^n_i$ and $t_1=t^n_{i+1}$ and sum over the index $i$ obtaining
\begin{equation*}
  u(t,x)-\Phi(x)=-\sum_{i=1}^{k_n+1}\E\left[\int_{t_i^n}^{t^n_{i+1}}G\left(r,X^{t^n_i,x}_r,Y^{t^n_i,x}_r,Z^{t_i^n,x}_r\right)\ud r\right] + I_n\ .
\end{equation*}
The term
\begin{equation*}
  I_n=\sum_{i=1}^{k_n+1}\E\left[ u\left(t_{i+1},X^{t_i,x}_{t_{i+1}}\right)-u\left(t_{i+1},e^{(t_{i+1}-t_i)A}x\right)\right]
\end{equation*}
can be treated as in the proof of Theorem 4.1 of \citep{flandoli2016infinite}, yelding as $n\to\infty$ the linear part of the PDE (i.e. the r.h.s. of (\ref{eq:kolm:int:nonlinear})). The only difference is that in \cite{flandoli2016infinite} $\Phi$ is assumed to be bounded, but the generalization to the polynomial growth (cf. Assumption \ref{ass:Phi2}) is immediate.\\
Concerning the remaining term, we need to prove that  
\begin{equation*}
\E \int_{t}^{T}\sum_{i=1}^{k_n+1} G\left(r,X^{t^n_i,x}_r,Y^{t^n_i,x}_r,Z^{t_i^n,x}_r\right) \ind_{[t^n_i,t^n_{i+1})}(r)\ud r \xrightarrow{n\to\infty} \int_t^T G\left(r,x,u(r,x),Du(r,x)\Sigma\right) \ud r \ .
\end{equation*} 
Let us write
\begin{equation*}
\begin{aligned}
  &\E \int_{t}^{T}\sum_{i=1}^{k_n+1}\left[G\left(r,X^{t^n_i,x}_r,Y^{t^n_i,x}_r,Z^{t_i^n,x}_r\right) - G\left(r,x,u(r,x),Du(r,x)\Sigma\right)\ud r\right]\ind_{[t^n_i,t^n_{i+1})}(r)\ud r \\
 &= \E \int_{t}^{T}\sum_{i=1}^{k_n+1}\left[G\left(r,X^{t^n_i,x}_r,Y^{t^n_i,x}_r,Z^{t_i^n,x}_r\right) - G\left(r,X^{t^n_i,x}_r,Y^{t^n_i,x}_r,Du(r,x)\Sigma\right) \right]\ind_{[t^n_i,t^n_{i+1})}(r)\ud r \\
 &\;+ \E \int_{t}^{T}\sum_{i=1}^{k_n+1}\left[G\left(r,X^{t^n_i,x}_r,Y^{t^n_i,x}_r,Du(r,x)\Sigma\right)  - G\left(r,X^{t^n_i,x}_r,u(r,x),Du(r,x)\Sigma\right)\right]\ind_{[t^n_i,t^n_{i+1})}(r)\ud r \\
 &\;+ \E \int_{t}^{T}\sum_{i=1}^{k_n+1}\left[G\left(r,X^{t^n_i,x}_r,u(r,x),Du(r,x)\Sigma\right) - G\left(r,x,u(r,x),Du(r,x)\Sigma\right)\right]\ind_{[t^n_i,t^n_{i+1})}(r)\ud r \\
\end{aligned}
\end{equation*}
so that, by the Lipschitz character of $G$ and Proposition \ref{p:identification}
\begin{equation}\label{conv:G}
\begin{aligned}
&\lesssim \E \int_{t}^{T}\sum_{i=1}^{k_n+1}\left[ \left| Du(r,X^{t_1^n,x}_r)\Sigma - Du(r,x)\Sigma\right| + \left| u(r,X^{t_i^n,x}_r) - u(r,x)\right|  \right]\ind_{[t^n_i,t^n_{i+1})}(r)\ud r \\
 &\;+ \E \int_{t}^{T}\sum_{i=1}^{k_n+1}\left[G\left(r,X^{t^n_i,x}_r,u(r,x),Du(r,x)\Sigma\right) - G\left(r,x,u(r,x),Du(r,x)\Sigma\right)\right]\ind_{[t^n_i,t^n_{i+1})}(r)\ud r \ .
\end{aligned}
\end{equation}
The last term in \eqref{conv:G} can be treated as follows. 
For every $r \in [t, T]$ fixed, there exists a unique sequence of intervals $\{[ t_{i(r,n)}^n, t_{i(r,n)+1}^n)\}_{n\in \N}$ such that  $r \in [ t_{i(r,n)}^n, t_{i(r,n)+1}^n)$ for every $n \in \N$. 
Moreover, for every $x \in \sL^2$, $r \in [t,T]$ and $\omega \in \Omega_0$, Proposition \ref{p:existence_X} guarantees the continuity of the map $\tau \mapsto X^{\tau,x}_r(\omega)$, so that
\begin{equation}\label{conv:Xti}
|X_r^{t^n_{i(r,n)},x}(\omega) - x|_{\sL^2} \xrightarrow{n\to \infty} 0 \ .
\end{equation} 
 From the regularity of $G$ (see Assumption \ref{ass:G}) it easily follows that for every $x \in \sL^2$, $r \in [t,T]$ and $\omega \in \Omega_0$
\begin{equation*}
\left|G\left(r,X^{t^n_{i(r,n)},x}_r(\omega),u(r,x),Du(r,x)\Sigma\right) - G\left(r,x,u(r,x),Du(r,x)\Sigma\right)\right| \xrightarrow{n\to \infty} 0
\end{equation*} 
Thanks to Assumption \ref{ass:G} and estimate \eqref{eq:estX}, the application of the Vitali theorem gives the required convergence.

Concerning the first term in \eqref{conv:G}, we employ for every $r \in [t,T]$ the continuity of the (deterministic) maps
$y \mapsto u(r,y)$ and $y \mapsto Du(r,y)\Sigma$, with $y \in \sL^2$,
given by Proposition \ref{p:identification} along with Propositions \ref{p:stimeBSDE} and \ref{p:diff1}, respectively.
These, in combination with \eqref{conv:Xti}, give the convergence of the integrand, for a.e. $\omega \in \Omega$, for every $r \in [t,T]$. 
Recalling estimates \eqref{eq:est:bsde_Y_Z} and \eqref{eq:estBSDE} and applying again the Vitali convergence theorem we finally get the result. 
\end{proof}

\begin{remark}
 With the very same proof we can easily show existence of solutions also in the space
    \begin{equation*}
      \sL^p:=\bR^d\times L^p\left(-T,0;\bR^d\right)\ , \quad  p>2\ .
    \end{equation*}
This allows to treat coefficients depending on the path in an integral way, which are not smooth in $\sL^2$ but satisfy our assumptions in $\sL^{2+\epsilon}$, for any $\epsilon>0$.
For this particular choice, it is then possible to establish wellposedness neither  requiring Assumption \ref{ass:GJ} nor introducing the approximation procedure explained in the next section.
\end{remark}

\section{Solution of the Kolmogorov equation in $\sD$}
\label{sec:PDED}
Here we prove our main result, following the strategy described in the introduction.\\
For every initial condition $x\in\aC^1$ and every initial time $t\in[0,T]$ we can find a solution $(X^{t,x},Y^{t,x},Z^{t,x})\in L^p\left(\Omega;C\left([t,T];\sD\right)\right)\times L^p\left(\Omega;C\left([t,T];\bR\right)\right)\times L^p\left(\Omega;L^2\left(0,T;\bR^{d_1}\right)\right)$ of the forward-backward system
\begin{equation}
\label{eq:FBSDE_0}
\begin{cases}
\ud X_s = \left[ AX_s + B(s,X_s) \right]\ud s + \Sigma \ud W_s & \text{ in }[t,T]\\
\ud Y_s = G(s,X_s,Y_s,Z_s) \ud s + Z_s \ud W_s & \text{ in }[t,T]\\
X_t = x \\
Y_T = \Phi(X_T).
\end{cases}
\end{equation} 
Then we can define the function $u\colon [0,T]\times \aC^1\to\bR$ as
\begin{equation}
  \label{eq:uuu}
  u(t,x)\colon=Y^{t,x}_t,
\end{equation}
and show that it is a classical solution of the Kolmogorov equation. This is the content of the next theorem.
\begin{theorem}\label{t:main}
Let $B$, $G$ and $\Phi$ satisfy respectively Assumtpions \ref{ass:B}, \ref{ass:G} and \ref{ass:Phi2} with $E=\sD$, as well as Assumptions \ref{ass:dphi} and \ref{ass:GJ}. Assume moreover that $B$ maps $\sC$ into itself. The function $u$ defined by (\ref{eq:uuu})
is a classical solution of the Kolmogorov semilinear equation with terminal condition $\Phi$, i.e. $u\in C^{1;2}\left([0,T]\times \sD,\mathbb{R}\right)$ and 
\begin{equation}
  \tag{\ref{eq:kolm:diff:nonlinear}}
\begin{system}
  \frac{\partial u}{\partial t}(t,x) + Du(t,x)\left[Ax + B(t,x)\right] + \frac{1}{2}\tr_{\bR^d}\left[\Sigma\Sigma^\ast D^2v(t,x)\right] = G(t,x,u(t,x),Du(t,x)\Sigma),\\
  u(T,\cdot) = \Phi,
\end{system}
\end{equation}
 for every $t \in [0,T]$ and $x \in \aC^1$.
\end{theorem}
\begin{remark}
  \label{rmk:Bgeneral}
The requirement that $B$ maps $\sC$ into itself is automatically satisfied if $B$ is the lifting of a path-dependent function as described in Subection \ref{subsec:FBSDE}.
The regularity $x \in \aC^1$ is necessary to give sense to the term $Ax$ and it is a standard requirement in the framework of classical solutions.
\end{remark}
 \begin{proof}
     We give here a complete scheme of the proof and postpone most of the technicalities to Lemmas \ref{step1.a}-\ref{step4} below.  
 We will assume for simplicity that $B$, $\Phi$ and $G$ have the same smoothing sequences, but the proof applies with almost no modifications also when the smoothing sequences are different.\\
 To lighten the notation, we will write $X^n_s=X^{n;t,x}_s(\omega)$, $X_s=X^{t,x}_s(\omega)$  (and similarly for $Y^n_s, Z^n_s$).
 We will also take $m\geq 1$ in Assumptions \ref{ass:B} and \ref{ass:G}; this guarantees that the exponents $mp$ in all the estimates below are larger than $1$. 
 The general case $m\geq 0$ follows from a further application of H\"older's inequality.\\
 Firstly observe that, by Proposition \ref{p:identification} and Corollary \ref{cor:D2u}, $u$ has the required regularity and 
 \begin{equation*}
  Du(t,x)=D_xY^{t,x}_t\quad\text{ and}\quad  D^2u(t,x)=D^2_xY^{t,x}_t\ .
 \end{equation*}
 Then, given $B$, $G$ and $\Phi$ we define for every $n\in\bN$
 \begin{gather*}
   B^n(t,x):=B(t,J^nx)\\
   G^n(t,x,y,z):=G(t,J^nx,y,z)\\
   \Phi^n(x):=\Phi(J^nx)\ ;
 \end{gather*}
 it is immediate to check that also $B^n$, $G^n$, $\Phi^n$ satisfy Assumptions \ref{ass:B}-\ref{ass:GJ} on $\sL^2$ with constants that \emph{do not depend} on $n$. Moreover
 \begin{equation}
   \label{eq:derivate}
   \begin{gathered}
     D\Phi^n(x)\bar x=D\Phi(J^nx)J^n\bar x \ , \\ D^2\Phi^n(x)(\bar x_1,\bar x_2)=D^2\Phi(J^nx)(J^n\bar x_1,J^n\bar x_2) \ , \\
     D_1G^n(r,x,y,z)\bar x=D_1G(r,J^nx,y,z)J^n\bar x \ , \\ D_{1,1}^2G^n(r,x,y,z)(\bar x_1,\bar x_2)=D_{1,1}^2G(r,J^nx,y,z)(J^n\bar x_1,J^n\bar x_2) \ , \\
     D_{1,2}G^n(r,x,y,z)(\bar x,\bar y)=D_{1,2}G(r,J^nx,y,z)(J^n\bar x,\bar y) \ , \\
     D_{2}G^n(r,x,y,z)\bar y=D_2G^n(r,J^nx,y,z)\bar y \ , \\
      D_{2,2}^2G^n(r,x,y,z)(\bar y_1,\bar y_2)=D_{2,2}^2G(r,J^nx,y,z)(\bar y_1,\bar y_2) \ ,
   \end{gathered}
 \end{equation}
 with similar identities for the derivatives of $G^n$ with respect to the variable $z$ and for the derivatives of $B^n$. We actually have that $B^n$ maps $\sL^2$ into $\sC\subset\sD$ and, thanks to the properties of $\{J^n\}_{n \in \N}$ (see Section \ref{subsec:assumptions}) and the Lipschitz character of $B$, for every $t\in[0,T]$ and $x\in\sC$ it holds that
   \begin{equation}
     \label{eq:convB}
   B^n(t,x)\xrightarrow{n\to\infty} B(t,x)\text{ in }\sC
 \end{equation}
and for every $(t,x,y,z)\in[0,T]\times\sC\times\bR\times\bR^{d_1}$
 \begin{equation}
   \label{convPhiG}
   \Phi^n(x)\xrightarrow{n\to\infty}\Phi(x)\ ,\quad G^n(t,x,y,z)\xrightarrow{n\to\infty} G(t,x,y,z) \quad \text{ in }\bR\ .
 \end{equation}
 For every $x\in\sL^2$, $t\in[0,T]$ and for each $n\in\bN$ we can solve the forward-backward system in $ \sL^2$ 
 \begin{equation}
 \label{eq:FBSDE_n}
 \begin{cases}
 dX^n_s = \left[ AX^n_s + B^n(s,X^n_s) \right]\ud s + \Sigma \ud W_s & \text{ in }[t,T]\\
 dY^n_s = G^n(s,X^n_s,Y^n_s,Z^n_s) \ud s + Z^n_s \ud W_s & \text{ in }[t,T]\\
 X^n_t = x \\
 Y^n_T = \Phi^n(X^n_T)
 \end{cases}
 \end{equation} 
thus obtaining a sequence of solutions $(X^{n;t,x},Y^{n;t,x},Z^{n;t,x})_n$. Note that all the estimates in Theorem \ref{thm:X}, Proposition \ref{p:stimeBSDE} and Propositions \ref{t.diff_gateaux_BSDE}-\ref{p.diff_2_BSDE} hold \emph{uniformly} in $n$ due to the equiboundedness of the $J^n$'s; this is a crucial feature for the proof.\\
Thanks to Theorem \ref{thm:L2}, the deterministic function
 \begin{equation*}
   u^n(t,x):=Y^{n;t,x}_t
 \end{equation*}
is twice Fr\'echet differentiable with
 \begin{equation}\label{id_Dun}
   Du^n(t,x)=D_xY^{n;t,x}_t\ ,\quad D^2u^n(t,x)=D^2_xY^{n;t,x}_t\ ,
 \end{equation}
 and it solves the backward PDE\textsubscript{n} in $ \sL^2$
 \begin{equation}\label{eq:kolm:diff:nonlinear_n}
 \begin{cases}
 \frac{\partial u^n}{\partial t}(t,x) + Du^n(t,x)\left[Ax + B^n(t,x)\right] + \frac{1}{2}\tr_{\bR^d}\left[\Sigma\Sigma^\ast D^2u^n(t,x)\right] = G^n(t,x,u^n(t,x),Du^n(t,x)\Sigma) \ , \\
 u^n(T,x) = \Phi^n(x) \ .
 \end{cases}
 \end{equation}
 By choosing $x\in\aC^1 \subset \aC$ \emph{also} in the system (\ref{eq:FBSDE_n}), we know that for every $n\in\bN$ and $s\in[t,T]$ the random variable $X^{n;t,x}_s$ belongs to $\aC$ (cf. Proposition \ref{p:existence_X}) and in particular it is differentiable as a random variable with values in $\sL^2$. 
 To conclude the proof, it remains to show that $u^n(t,x)$ converges to $u(t,x)$ for every $t\in[0,T]$ and that each term in the PDE\textsubscript{$n$} converges to the corresponding term in the PDE (\ref{eq:kolm:diff:nonlinear}) as $n\to\infty$.\\
 Convergence of $u^n(t,x)$ to $u(t,x)$, for every $(t,x)\in[0,T]\times\aC$, is a consequence of the (more general) convergence $Y^n \to Y$ in $L^p\left(\Omega;C\left([t,T];\bR\right)\right)$ given in Lemma \ref{step2.a} below.\\
Regarding the first derivative of $u$, Lemma \ref{step3} guarantees that for any $h\in\sC$ 
 \begin{equation*}
   DY^n_\cdot h\to DY_\cdot h \qquad  \text {in } L^p(\Omega;C([t,T];\bR)) \ ;
 \end{equation*}
 therefore $Du^n(t,x) h\to Du(t,x)h$  for every $h \in \sC$.
 Writing
 \begin{equation*}
   Du^n(t,x)B^n(t,x)- Du(t,x)B(t,x)= Du^n(t,x)\left[B^n(t,x)-B(t,x)\right]+ \left[Du^n(t,x)-Du(t,x)\right]B(t,x)\ ,
 \end{equation*}
 the convergence of the second term on the r.h.s. is straightforward. 
 Using estimate (\ref{eq:stima_grad_u}) for $u^n$ (which is indeed uniform in $n$), the first term goes to zero by (\ref{eq:convB}).
 Since $Ax \in \sC$, this implies that the linear first order term $Du^n(t,x)\left[Ax+B^n(t,x)\right]$ in PDE\textsubscript{n} converges to the corresponding term in the limit PDE.
 \smallskip
 
 For what concerns the second order term in PDE\textsubscript{n}, in Lemma \ref{step4} we exploit the identification result obtained  in Proposition \ref{p:identification} to show that for any $h,k\in\sC$ it holds
 \begin{equation*}
 \begin{split}
 D^2_xY^n(k,h) &\to D^2_xY(k,h) \qquad \text{ in } L^p\left(\Omega;C\left([t,T];\bR\right)\right) \ , \\
 D^2_xZ^n(k,h) &\to D^2_xZ(k,h) \qquad \text{ in } L^p(\Omega;L^2([t,T];\bR^{d_1})) \ ,
 \end{split}
 \end{equation*}
 which is sufficient thanks to \eqref{id_Dun}.
 Finally, since $Y_t$, $Y^n_t$, $D_xY_t$ and $D_xY^n_t$ are all deterministic, from continuity of $G$ and Lemmas \ref{step2.a}, \ref{step3} it follows that 
 \begin{equation*}
   G^n\left(t,x,u^n(t,x),D_xY^n_t\Sigma\right)\to G\left(t,x,u(t,x),Du(t,x)\Sigma\right)\ ,
   \end{equation*}
   and this concludes the proof.
 \end{proof}
In Lemmas \ref{step1.a}$-$\ref{step4} below we will always let the assumptions of Theorem \ref{t:main} to hold, without explicitly write it in every statement.
The only difference concerns the less stringent requirement $x \in \aC$  (instead of $x \in \aC^1$) which turns out to be sufficient for all the convergences.

We will use the notation $a\lesssim b$ meaning $a\leq Cb$ for some positive constant $C$ only when the hidden constant $C$ does not depend on $n$ nor on the time variables. 
All the convergences has to be intended as $n$ goes to $+\infty$.
\begin{lemma}\label{step1.a}
Let $t \in [0,T]$ and  $x \in \aC$. Then $X^{n;t,x}\to X^{t,x}$ and $J^nX^{n;t,x}\to X^{t,x}$ $\bP$-a.s. in $C\left([0,T];\sC\right)$ and
 also in $L^p\left(\Omega;C\left([0,T];\sC\right)\right)$.
\end{lemma}
\begin{proof}Recall that $\Omega_0\subset\Omega$ is the subset of full probability where $X_\cdot^{(t,x)}$  has continuous trajectories.
Given $x\in\aC$ and $\omega\in\Omega_0$ the continuity of the map $[t,T]\ni s\mapsto X_s\in\aC$ guarantees the compactness in $\aC$ of the set $\left\{X_s(\omega)\right\}_s$. Therefore $J^nX_s\to X_s$ uniformly in $s$ almost surely, i.e.
\begin{equation*}
  \sup_{s\in[t,T]}\left\vert J^nX_s-X_s\right\vert\xrightarrow{n\to\infty}0\ , \qquad \bP\text{-a.s.} \ .
\end{equation*}
Thanks to Assumption \ref{ass:B} and to the properties of the semigroup $e^{tA}$ we have almost surely
\begin{align*}
  \left\vert X^n_\tau-X_\tau\right\vert&=\left\vert\int_t^\tau e^{(\tau-r)A}\left[B^n\left(r,X^n_r\right)-B\left(r,X_r\right)\right]\ud r\right\vert\\
  &\lesssim\int_t^\tau\left[\left\vert B\left(r,J^nX^n_r\right)-B\left(r,J^nX_r\right)\right\vert+\left\vert B\left(r,J^nX_r\right)-B\left(r,X_r\right)\right\vert\right]\ud r\\
  &\lesssim\int_t^\tau\left[\left\vert X^n_r-X_r\right\vert+\left\vert J^nX_r-X_r\right\vert\right]\ud r\ ;
\end{align*}
therefore
\begin{equation*}
  \sup_{\tau\in[t,s]}\left\vert X^n_\tau-X_\tau\right\vert\lesssim\int_t^s\sup_{\tau\in[t,r]}\left\vert X^n_\tau-X_\tau\right\vert\ud r+\int_t^s\sup_{\tau\in[t,r]}\left\vert J^nX_\tau-X_\tau\right\vert\ud r
\end{equation*}
and by Gronwall's lemma 
\begin{equation}
  \label{eq:duepallini}
  \sup_{\tau\in[t,s]}\left\vert X^n_\tau-X_\tau\right\vert\xrightarrow{n\to\infty}0 \ , \qquad \bP\text{-a.s.} \ .
\end{equation}
 Since for every $s$ a.s.
\begin{equation*}
  \begin{split}
    \left\vert J^nX^n_s-X_s\right\vert_{\aC}&\leq\left\vert J^nX^n_s-J^nX_s\right\vert_{\aC}+\left\vert J^nX_s-X_s\right\vert_{\aC}\\
    &\leq \left\vert J^n\right\vert_{L(\sL^2;\sC)}\ \left\vert X^n_s-X_s\right\vert_{\aC}+\left\vert J^nX_s-X_s\right\vert_{\aC}\xrightarrow{n\to\infty}0\ ,
  \end{split}
\end{equation*}
the equiboundedness of the $J^n$'s implies that
\begin{equation*}
  \sup_{\tau\in[t,s]}\left\vert J^nX^n_s-X_s\right\vert_{\aC}\xrightarrow{n\to\infty}0 \ .
\end{equation*}
The second claim follows by estimate \eqref{eq:estX} (here the initial datum $x$ is deterministic). Indeed
\begin{equation*}
\sup_{s\in[t,T]}\left\vert X^n_s\right\vert+\sup_{s\in[t,T]}\left\vert X_s\right\vert \leq \gamma_T, \qquad \bP\text{-a.s.}
\end{equation*}
where $\gamma_T$ is a random variable with $\bE\gamma_T^p<\infty$, for every $p\geq 1$.
\end{proof}

\begin{lemma}\label{step1.b}
Let $t \in [0,T]$ and  $x \in \aC$. For every $h\in\sC$,  $D_xX^{n;t,x}h\to D_xX^{t,x} h$ $\bP$-a.s. in $C\left([0,T];\sC\right)$ and in $L^p\left(\Omega;C\left([0,T];\sC\right)\right)$.
Moreover, for every $h,k\in\sC$, $D^2_xX^{n;t,x}(k,h)\to D_x^2X^{t,x}(k,h)$ in $L^p\left(\Omega,C\left([0,T];\sC\right)\right)$.
\end{lemma}

\begin{proof}
First note that in general we cannot expect $J^nD_xX_sh$ to converge to $D_xX_sh$ when $h\notin\aC$; this is due to the action of the semigroup $e^{tA}$ on $h$ (see equation (\ref{eq:DX1})).
We prove here only the first part of the statement, for what concerns second order derivatives we refer to the appendix. 

Thanks to the equiboundedness of the $J^n$'s and (\ref{eq:estDXnorm}), we can find a constant $c=c(B,\Sigma,T)$ such that 
\begin{equation}
  \label{eq:estDXunif}
  \left\vert D_xX_s(t,x)h\right\vert\vee\sup_{n\in\bN}\left\vert D_xX^n_s(t,x)h\right\vert\leq c\left\vert h\right\vert\, \quad \forall h\in\sC \, ,\ \forall s\in[t,T]\ .
\end{equation}
By properties of $B$, $e^{tA}$ and $J^n$ we also have, for $h\in\sC$,
\begin{align*}
&\left\vert D_xX^n_\tau h-D_xX_\tau h\right\vert^p\lesssim\int_t^\tau\left\vert DB^n\left(r,X^n_r\right)D_xX^n_rh-DB\left(r,X_r\right)D_xX_rh\right\vert^p\ud r\\ 
  &\lesssim\int_t^\tau\left\vert DB\left(r,J^nX^n_r\right) \left(J^nD_xX^n_rh - J^nD_xX_rh\right)\right\vert^p\ud r +\int_t^\tau\left\vert \left(DB\left(r,J^nX^n_r\right) - DB\left(r,X_r\right) \right)J^nD_xX_rh\right\vert^p\ud r\\
  &\phantom{\lesssim}+\int_t^\tau\left\vert DB\left(r,X_r\right)\left(J^nD_xX_rh-D_xX_rh\right)\right\vert^p\ud r\\
  &\lesssim\int_t^\tau\left\vert D_xX^n_rh-D_xX_rh\right\vert^p\ud r +\int_t^\tau\left\vert h\right\vert^p\left\vert J^nX^n_r-X_r\right\vert^p\int_0^1\left\vert D^2B\left(r,a J^nX^n_r+(1-a)X_r\right)\right\vert^p\ud a\ud r\\
  &\phantom{\lesssim}+\int_t^\tau\left\vert DB\left(r,X_r\right)\left(J^nD_xX_rh-D_xX_rh\right)\right\vert^p\ud r\\
  &\lesssim\int_t^\tau\left\vert D_xX^n_rh-D_xX_rh\right\vert^p\ud r +\int_t^\tau\left\vert h\right\vert^p\left\vert J^nX^n_r-X_r\right\vert^p\int_0^1\left(1+\left\vert a J^nX^n_r+(1-a)X_r\right\vert^{mp}\right)\ud a\ud r\\
  &\phantom{\lesssim}+\int_t^\tau\left\vert DB\left(r,X_r\right) \left(J^nD_xX_rh-D_xX_rh\right)\right\vert^p\ud r\ .
\end{align*}
Therefore
\begin{align*}
  \bE\sup_{\tau\in[t,T]}&\left\vert D_xX^n_\tau h-D_xX_\tau h\right\vert^p\lesssim\int_t^T\bE\sup_{\tau\in[t,r]}\left\vert D_xX^n_\tau h-D_xX_\tau h\right\vert^p\ud r\\
  &\phantom{\lesssim}+\bE\int_t^T\left\vert h\right\vert^p\left\vert J^nX^n_r-X_r\right\vert^p\int_0^1\left(1+\left\vert a J^nX^n_r+(1-a)X_r\right\vert^{mp}\right)\ud a\ud r\\
  &\phantom{\lesssim}+\bE\int_t^T\left\vert DB\left(r,X_r\right)\left(J^nD_xX_rh-DX_rh\right)\right\vert^p\ud r\ .
\end{align*}
Now the second term is bounded by
\begin{equation*}
 \E \left[ \left\vert h\right\vert^p\left(1+\sup_{r\in[t,T]}\left\vert X_r^n\right\vert^{mp}+\sup_{r\in[t,T]}\left\vert X_r\right\vert^{mp}\right)\sup_{r\in[t,T]}\left\vert J^nX^n_r-X_r\right\vert^p \right],
\end{equation*}
which goes to zero thanks to H\"older inequality, estimates \eqref{eq:estX} and  Lemma \ref{step1.a}. 
Exploiting estimate \eqref{eq:estDX} and Assumption \ref{ass:dphi} the same holds for the third term. 
From the Gronwall's lemma we get the convergence in $L^p\left(\Omega;C\left([0,T];\sC\right)\right)$. 
Finally, by the very same technique, the a.s. convergence in $C\left([0,T];\sC\right)$ follows directly exploiting the a.s. convergence of $J^nX^{n;t,x}$ given in Lemma \ref{step1.a}.
\end{proof}

\begin{lemma}\label{step2.a}
Let $t \in [0,T]$ and  $x \in \aC$. 
Then $Y^{n;t,x}\to Y^{t,x}$ in $L^p\left(\Omega;C\left([t,T];\bR\right)\right)$
and  $Z^{n;t,x}\to Z^{t,x}$ in $L^p\left(\Omega;L^2\left(t,T;\bR^{d_1}\right)\right)$.
\end{lemma}

\begin{proof}
We first show that, for every $s\in[t,T]$, $Y_s^{n;t,x}\to Y_s^{t,x}$ in $L^p\left(\Omega;\bR\right)$.\\
Given $p \ge 2$, for every $s\in[t,T]$, $Y^n_s-Y_s$ and $Z^n_s-Z_s$ satisfy the identity
\begin{equation*}
  Y^n_s-Y_s+\int_s^T\left[ Z^n_r-Z_r\right]\ud W_r=\Phi^n\left(X^n_T\right)-\Phi\left(X_T\right)+\int_s^T\hat G^n_r\ud r\ ,
\end{equation*}
where $\hat G_r^n$ is the process
\begin{equation*}
  \hat G^n_r=G^n\left(r,X^n_r,Y^n_r,Z^n_r\right)-G\left(r,X_r,Y_r,Z_r\right)\ .
\end{equation*}
By It\^o formula and taking expectation we get
\begin{equation}
  \label{eq:Ydiff}
  \begin{aligned}
    \bE&\left\vert Y^n_s-Y_s\right\vert^p +\frac{p(p-1)}{2}\bE\int_s^T\left\vert Y^n_r-Y_r\right\vert^{p-2}\left\vert Z^n_r-Z_r\right\vert^2\ud r \\
    &\quad \leq \bE\left\vert\Phi^n\left(X^n_T\right)-\Phi(X_T)\right\vert^p+p\bE\int_s^T\left\vert Y^n_r-Y_r\right\vert^{p-1}\ \vert \hat G^n_r\vert\ud r\ .
\end{aligned}  
  \end{equation}
  Since $G^n, G$ satisfy Assumption \ref{ass:G}, for the last integral in (\ref{eq:Ydiff}) we have
\begin{equation*}
  \begin{aligned}
  &\bE\int_s^T\left\vert Y^n_r-Y_r\right\vert^{p-1}\ \vert \hat G^n_r\vert\ud r \leq 
    \left( \frac{C^2}{2} + C + \frac{p-1}{p} \right) \int_s^T\left\vert Y^n_r-Y_r\right\vert^p\ud r\\
    &\quad +\frac{1}{2}\int_s^T \left\vert Y^n_r-Y_r\right\vert^{p-2}\left\vert Z^n_r-Z_r\right\vert^2\ud r +\frac{1}{p}\int_s^T\left\vert G\left(r,J^nX^n_r,Y_r,Z_r\right)-G\left(r,X_r,Y_r,Z_r\right)\right\vert^p\ud r
  \end{aligned}
\end{equation*}
where $C$ is the constant provided by assumption \ref{ass:G}. 
Therefore
\begin{equation}  \label{eq:Ydifftemp}
\begin{aligned}
  \bE\left\vert Y^n_s-Y_s\right\vert^p&\lesssim \int_s^T\bE\left\vert Y^n_r-Y_r\right\vert^p\ud r + \bE\left\vert \Phi^n\left(X^n_T\right)-\Phi\left(X_T\right)\right\vert^p\\ 
  &+\bE\int_s^T\left\vert G\left(r,J^nX^n_r,Y_r,Z_r\right)-G\left(r,X_r,Y_r,Z_r\right)\right\vert^p\ud r
\end{aligned}
\end{equation}
and since $\int_s^T\left\vert G\left(r,J^nX^n_r,Y_r,Z_r\right)-G\left(r,X_r,Y_r,Z_r\right)\right\vert^p\ud r$ is decreasing in $s$, by Gronwall's lemma
\begin{equation}
  \label{eq:YnY1}
  \bE\left\vert Y^n_s-Y_s\right\vert^p\lesssim\bigg[\bE\left\vert \Phi^n\left(X^n_T\right)-\Phi\left(X_T\right)\right\vert^p+\bE\int_t^T\left\vert G\left(r,J^nX^n_r,Y_r,Z_r\right)-G\left(r,X_r,Y_r,Z_r\right)\right\vert^p\ud r\bigg]
\end{equation}
The first term on the r.h.s. of (\ref{eq:YnY1}) can be easily shown to converge to $0$ thanks to the properties of $\Phi$, the uniform bound on $J^n$ and the convergence proved in Lemma \ref{step1.a}.
For the second term on the r.h.s. of (\ref{eq:YnY1}), recall that $G$ is continuous and by  Lemma \ref{step1.a} $J^nX^n_s(\omega)$ converges to $X_s(\omega)$ for every $s \in [t,T]$ and a.e. $\omega \in \Omega$;
 then Assumption \ref{ass:G}, estimates \eqref{eq:estX} and Propositions \ref{p:stimeBSDE} and \ref{p:identification} yield
\begin{equation*}
  \bE\int_t^T\left\vert G\left(r,J^nX^n_r,Y_r,Z_r\right)-G\left(r,X_r,Y_r,Z_r\right)\right\vert^p\ud r\xrightarrow{n\to\infty}0 \ ,
\end{equation*}
thanks to Vitali convergence theorem.
Note that, by Tonelli's theorem and the dominated convergence theorem,
  $\bE\int_s^T\left\vert Y^n_r-Y_r\right\vert^p\ud r \to 0$
for every $s\in[t,T]$; so that 
\begin{equation*}
  \bE\left(\int_t^T\left\vert Y^n_r-Y_r\right\vert^2\ud r\right)^{\nicefrac{p}{2}}\xrightarrow{n\to\infty}0 \ .
\end{equation*}
For what concerns $Z^{n;t,x}$, starting by \eqref{eq:Ydiff} with $p=2$ it is easily seen that $\bE\int_t^T\left\vert Z^n_r-Z_r\right\vert^2\ud r \to 0$; moreover, 
\begin{equation*}
  \bE\left(\int_s^T\left\vert Z^n_r-Z_r\right\vert^2\ud r\right)^{\frac{p}{2}}\leq\left[\bE\int_s^T\left\vert Z^n_r-Z_r\right\vert^2\ud r\right]^{\frac{1}{2}}\left[\bE\left(\int_s^T\left\vert Z^n_r-Z_r\right\vert^2\ud r\right)^{p-1}\right]^{\frac{1}{2}} \ ,
\end{equation*}
so that the result holds for arbitrary $p\geq 2$ since the last term can be estimated uniformly in $n \in \N$ thanks to Proposition \ref{p:stimeBSDE}.

Let finally show the refined convergence $Y^{n;t,x}\to Y^{t,x}$ in $L^p\left(\Omega;C\left([t,T];\bR\right)\right)$.
It is easy to show that  
\begin{equation*}
\bE\left(\int_t^T\left\vert \hat G^n_r\right\vert^2\ud r \right)^{\nicefrac{p}{2}} \lesssim 1+\left\vert x\right\vert^{mp} \ ;
\end{equation*}
hence we can apply estimate (\ref{eq:BSDE_stima_classica}) in Proposition \ref{prop:FT} to obtain
\begin{equation}
  \label{eq:est_Ghat}
  \bE\sup_{s\in[t,T]}\left\vert Y^n_r-Y_r\right\vert^p+\bE\left(\int_t^T\left\vert Z^n_r-Z_r\right\vert^2\ud r\right)^{\nicefrac{p}{2}}\lesssim \bE\left(\int_t^T\left\vert\hat G_r^n\right\vert^2\ud r\right)^{\nicefrac{p}{2}}+\bE\left\vert\Phi^n\left(X^n_T\right)-\Phi\left(X_T\right)\right\vert^p\ .
\end{equation}
From Assumption \ref{ass:G} it holds 
\begin{multline*}
  \bE\left(\int_t^T\left\vert\hat G_r^n\right\vert^2\ud r\right)^{\nicefrac{p}{2}}\lesssim\bE\left(\int_t^T\left\vert Y^n_r-Y_r\right\vert^2\ud r\right)^{\nicefrac{p}{2}}+\bE\left(\int_t^T\left\vert Z^n_r-Z_r\right\vert^2\ud r\right)^{\nicefrac{p}{2}}\\+\bE\left(\int_t^T\left\vert G\left(r,J^nX^n_r,Y_r,Z_r\right)-G\left(r,X_r,Y_r,Z_r\right)\right\vert^2\ud r\right)^{\nicefrac{p}{2}}\ ,
\end{multline*}
therefore the r.h.s. of (\ref{eq:est_Ghat}) converges to $0$ as $n \to + \infty$ thanks to Lemma \ref{step2.a} and the uniform convergence of the $X^n$.
\end{proof}

\begin{lemma}\label{step3}
Let $t \in [0,T]$ and  $x \in \aC$. 
For any $h\in\sC$, $DY_\cdot^{n;t,x}h\to DY_\cdot^{t,x}h$ in $L^p(\Omega;C([t,T];\bR))$ and $DZ_\cdot^{n;t,x}h\to DZ_\cdot^{t,x}h$ in $L^p(\Omega;L^2([t,T];\bR^{d_1}))$.\\
\end{lemma}
\begin{proof}
Similarly to the proof of Proposition \ref{p:diff1}, for any $h\in\sC\subset\sL^2$, we consider the equations satisfied by 
$\Delta Y^n_r=\left(D_xY^n_r-D_xY_r\right)h$, $\Delta Z^n_r=\left(D_xZ^n_r-D_xZ_r\right)h$:
\begin{equation*}
  \Delta Y^n_s+\int_s^T\Delta Z^n_r\ud W_r=\eta^n+\int_s^T\alpha^n_r\Delta Y^n_r\ud r +\int_s^T\beta^n(r)\ud r +\int_s^T\gamma^n_r\Delta Z_r\ud r
\end{equation*}
where
\begin{gather*}
  \eta^n:=\left[D_x\Phi^n\left(X^n_T\right)D_xX^n_T-D_x\Phi\left(X_t\right)D_xX_T\right]h\ ,\\
  \alpha^n_r:=-D_2G^n\left(r,X^n_r,Y^n_r,Z^n_r\right)\ ,\\
  \beta^n(r):=-\left[D_1G^n\left(r,X^n_r,Y^n_r,Z^n_r\right)D_xX^n_r-D_1G\left(r,X_r,Y_r,Z_r\right)D_xX_r\right]h\\
  -\left[D_2G^n\left(r,X^n_r,Y^n_r,Z^n_r\right)-D_2G\left(r,X_r,Y_r,Z_r\right)\right]D_xY_rh\phantom{=,}\\
  -\left[D_3G^n\left(r,X^n_r,Y^n_r,Z^n_r\right)-D_3G\left(r,X_r,Y_r,Z_r\right)\right]D_xZ_rh\phantom{+}\\
  =:-\beta^n_1(r)-\beta^n_2(r)-\beta^n_3(r)\ ,\quad\gamma^n_r:=-D_3G^n\left(r,X^n_r,Y^n_r,Z^n_r\right)\ .\phantom{++++,}\\
\end{gather*}
and 
\begin{equation}
\label{eq:Vn}
  V^n_s=\int_t^s\left\vert\alpha^n_r\right\vert\ud r+\frac{1}{1\wedge(p-1)}\int_t^s\left\vert\gamma^n_r\right\vert^2\ud r\ ,
\end{equation}
By Lemma \ref{l:linearBSDE} we have, for every $n \in \N$, the estimate
\begin{equation*}
  \bE\sup_{t\in[t,T]}\left\vert e^{V^n_t}\Delta Y^n_t\right\vert^p+\bE\left(\int_t^Te^{2V^n_r}\left\vert\Delta Z^n_r\right\vert^2\ud r\right)^{\frac{p}{2}}\lesssim \bE\left\vert e^{V^n_t}\eta^n\right\vert^p+\bE\left(\int_t^Te^{V^n_r}\left\vert\beta^n(r)\right\vert\ud r\right)^p\ ;
\end{equation*}
to get the desired convergence we need to show that the r.h.s. of the above inequality goes to $0$ as $n \to +\infty$.
By the uniform boundedness of $V^n_s$, $n \in \N$, $s \in [t,T]$, it holds, for $p \geq 2$
\begin{equation*}
  \begin{aligned}
    \bE\left\vert \eta^n\right\vert^p &\lesssim \bE\left\vert D\Phi^n\left(X^n_T\right)D_xX^n_Th-D\Phi\left(X_t\right)D_xX_Th\right\vert^p \\
    &\lesssim \bE\left\vert D\Phi^n\left(X^n_T\right) \left( D_xX^n_Th - D_xX_Th \right)\right\vert^p+\bE\left\vert \left(D\Phi^n\left(X^n_T\right) - D\Phi\left(X_T\right)\right)D_xX_Th\right\vert^p\\
    &=\bE\left(\left[\mathrm{B}_1^n\right]^p+\left[\mathrm{B}_2^n\right]^p\right)\ .
  \end{aligned}
\end{equation*}
Recalling (\ref{eq:derivate}) we have, by the equiboundedness of the $J^n$'s,
\begin{equation*}
  \begin{split}
    \mathrm{B}^n_1\leq\left\| D\Phi\left(J^nX^n_T\right)\right\|\left\vert J^nD_xX^n_Th-J^nD_xX_Th\right\vert\lesssim \left(1+\sup_{s\in[t,T]}\left\vert X^n_s\right\vert^m\right)\left\vert D_xX^n_Th-D_xX_Th\right\vert
  \end{split}
\end{equation*}
so that, by Lemma \ref{step1.b}
\begin{equation*}
  \bE\left(\left[\mathrm{B}_1^n\right]^p\right)\lesssim\left[\bE\left(1+\sup_{s\in[t,T]}\left\vert X^n_s\right\vert^{2mp}\right)\right]^{\frac{1}{2}}\left[\bE\left\vert D_xX^n_Th-D_xX_Th\right\vert^{2p}\right]^{\frac{1}{2}}\xrightarrow{n \to \infty} 0 \ .
\end{equation*}
Concerning $\mathrm{B}^n_2$ we have 
\begin{equation*}
\begin{split}
  \E \left( \mathrm{B}^n_2 \right)^p &\leq \E  \left\| D\Phi(J^nX_T^n)-D\Phi(X_T)\right\|^p\left\vert J^nD_xX_Th\right\vert^p+ \E \left\vert D\Phi(X_T)\left(J^nD_xX_Th-D_xX_Th\right)\right\vert^p \\
  &\lesssim \left(1+\vert x\vert ^{mp}\right)\vert h \vert^p
\end{split}  
  \end{equation*}
Hence by Lemma \ref{step1.b}, continuity of $\Phi$, equiboundedness of $J^n$ and Assumption \ref{ass:dphi}, $\E \left( \mathrm{B}^n_2 \right)^p$ goes to zero, implying that $\bE\left\vert\eta^n\right\vert^p\to 0$, as $n \to +\infty$.\\
 To prove that $\bE\left(\int_o^T\left\vert\beta^n(r)\right\vert\ud r\right)^p\to 0$ first note that
\begin{equation*}
  \bE\left(\int_0^T\left\vert\beta^n(r)\right\vert e^{\lambda V^n_r}\ud r\right)^p\lesssim\sum_{i=1}^3\bE\left(\int_0^T\left\vert\beta^n_i(r)\right\vert\ud r\right)^p.
\end{equation*}
We detail the computations for the term $\beta_1^n(r)$, the remaining terms being very similar.
  \begin{equation*}
    \begin{split}
\left\vert\beta^n_1(r)\right\vert&=\vert D_1G\left(r,J^nX^n_r,Y^n_r,Z^n_r\right)J^nD_xX^n_rh-D_1G\left(r,X_r,Y_r,Z_r\right)D_xX_rh\vert\\
&\lesssim \left\vert D_1G\left(r,J^nX^n_r,Y^n_r,Z^n_r\right)J^n\left(D_xX^n_rh-D_xX_rh\right)\right\vert\\
&\phantom{\lesssim}+\left\vert \left[D_1G\left(r,J^nX^n_r,Y^n_r,Z^n_r\right)J^n-D_1G\left(r,X_r,Y_r,Z_r\right)\right]D_xX_rh\right\vert\\
      &=\mathrm{C}_1^n(r)+\mathrm{C}_2^n(r) \ .
    \end{split}
  \end{equation*}
Using Assumption \ref{ass:G} we get
\begin{equation*}
  \begin{split}
    \mathrm{C}_1^n(r)&\leq\left\vert D_1G\left(r,J^nX^n_r,Y^n_r,Z^n_r\right)\right\vert\left\vert J^n\right\vert\left\vert D_xX^n_rh-D_xX_rh\right\vert\\
    &\lesssim \left(1+\sup_{r\in[s,T]}\left\vert X^n_r\right\vert^m\right)\left(1+\sup_{r\in[s,T]}\left\vert Y^n_r\right\vert\right)\left\vert D_xX^n_rh-D_xX_rh\right\vert \ ,
  \end{split}
\end{equation*}
therefore, reasoning as before, 
\begin{equation*}
  \begin{split}
   \bE\int_s^T\left[\mathrm{C}_1^n(r)\right]^p\ud r\lesssim \left(1+\left\vert x\right\vert ^{2pm}\right)\left[\bE\left(\int_s^T\left\vert D_xX^n_rh-D_xX_rh\right\vert\ud r\right)^{3p}\right]^{\frac{1}{3}}
  \end{split}
\end{equation*}
and the r.h.s. goes to $0$ as $n \to +\infty$ by Lemma \ref{step1.b}.
Then  we have
  \begin{equation*}
    \begin{split}
      \mathrm{C}_2^n(r)&\leq\left\vert D_1 G\left(r,J^nX^n_r,Y^n_r,Z^n_r\right)J^nD_xX_rh-D_1G\left(r,J^nX^n_r,Y^n_r,Z_r\right)J^nD_xX_rh\right\vert\\
      &\phantom{\leq}+\left\vert D_1G\left(r,J^nX^n_r,Y^n_r,Z_r\right)J^nD_xX_rh-D_1G\left(r,J^nX^n_r,Y_r,Z_r\right)J^nD_xX_rh\right\vert\\
      &\phantom{\leq}+\left\vert D_1G\left(r,J^nX^n_r,Y_r,Z_r\right)J^nD_xX_rh-D_1G\left(r,X_r,Y_r,Z_r\right)D_xX_rh\right\vert\\
      &=\mathrm{C}_{21}^n(r)+\mathrm{C}_{22}^n(r)+\mathrm{C}_{23}^n(r)\ ,
    \end{split}
  \end{equation*}
thanks to Assumption \ref{ass:G} we can further bound
  \begin{equation*}
    \begin{split}
      \mathrm{C}_{21}^n(r)&=\left\vert\int_0^1 D_{1,3}^2G\left(r,J^nX^n_r,Y^n_r,a Z^n_r+(1-a)Z_r\right)\Big(Z^n_r-Z_r,J^nD_xX_rh\Big)\ud a\right\vert\\
      &\lesssim\left(1+\sup_{r\in[s,T]}\left\vert X^n_r\right\vert^m\right)\left(1+\sup_{r\in[s.T]}\left\vert Y^n_r\right\vert\right)\left(\sup_{r\in[s,T]}\left\vert D_xX_rh\right\vert\right)\left\vert Z^n_r-Z_r\right\vert
    \end{split}
  \end{equation*}
  and
\begin{equation*}
  \begin{split}
    \bE\left(\int_s^T\mathrm{C}_{21}^n(r)\ud r\right)^p&\lesssim\left[1+\bE\sup_{r\in[0,T]}\left\vert X^n_r\right\vert^{4mp}\right]^{\frac{1}{4}}\left[1+\bE\sup_{r\in[0,T]}\left\vert Y^n_r\right\vert^{4p}\right]^{\frac{1}{4}}\\
&\phantom{\lesssim a}\times\left[\bE\sup_{r\in[0,T]}\left\vert D_xX_rh\right\vert^{4p}\right]^{\frac{1}{4}}\left[\bE\left(\int_s^T\left\vert Z^n_r-Z_r\right\vert^2\ud r\right)^{2p}\right]^{\frac{1}{4}}\\
&\lesssim\left(1+\vert x\vert^{mp}\right)\vert h\vert^p\left(1+\vert x\vert^{pm^2}\right)\left[\bE\left(\int_s^T\left\vert Z^n_r-Z_r\right\vert^2\ud r\right)^{2p}\right]^{\frac{1}{4}}\xrightarrow{n\to\infty}0 
  \end{split}
\end{equation*}  
thanks to Lemma \ref{step2.a}.
The same holds for $\mathrm{C}_{22}^n(r)$,    
while
\begin{equation*}
  \begin{aligned}
      \mathrm{C}_{23}^n(r) &\leq \left\vert \left[D_1G\left(r,J^nX^n_r,Y_r,Z_r\right) - D_1G\left(r,X_r,Y_r,Z_r\right)\right]J^nD_xX_rh\right\vert  \\
      &+ \left\vert D_1G\left(r,X_r,Y_r,Z_r\right)\left(J^nD_xX_rh-D_xX_rh\right)\right\vert
    \end{aligned}
    \end{equation*}
 which  goes to $0$ as $n\to +\infty$ thanks to continuity of $D_1G$, equiboundedness of $J^n$, Lemma \ref{step1.b} and Assumption \ref{ass:GJ}.
 By the uniform bound (\ref{eq:estDXunif}) and Vitali convergence theorem also $\bE\left(\int_s^T\mathrm{C}_{23}^n(r)\ud r\right)^p$ goes to $0$ as $n \to +\infty$. 
This immediately yields that 
\begin{equation*}
  \bE\left(\int_s^T\mathrm{C}_2^n(r)\ud r\right)^p\leq\sum_{i=1}^3\bE\left(\int_s^T\mathrm{C}_{2i}^n(r)\ud r\right)^p \xrightarrow{n\to\infty}0 \ ,
\end{equation*}
which concludes the proof.

\end{proof}

The next lemma concerns the convergence of second order derivatives and it is the most delicate one.

\begin{lemma}\label{step4} 
Let $t \in [0,T]$ and  $x \in \aC$. 
For any $h,k\in\sC$, $D^2_xY^{n;t,x}(k,h)\to D^2_xY^{t,x}(k,h)$ in $L^p\left(\Omega;C\left([t,T];\bR\right)\right)$ and $D^2_xZ^{n;t,x}(k,h)\to D^2_xZ^{t,x}(k,h)$ in $L^p(\Omega;L^2([t,T];\bR^{d_1}))$.
\end{lemma}

\begin{proof}
Here we just focus on the most difficult (and interesting) term to deal with, the other terms are discussed in the Appendix.
Recalling the equation satisfied by $D^2_xY_s(k,h)$ (analogously by $D^2_xY_s^n(k,h)$) and denoting
\begin{equation*}
  \Delta^2Y^n_s(k,h)=D^2_xY^n_s(k,h)-D^2_xY_s(k,h)\ ,\quad \Delta^2Z^n_s(k,h)=D^2_xZ^n_s(k,h)-D^2_xZ_s(k,h) \ ,
\end{equation*}
it holds that
\begin{equation*}
  \Delta^2Y^n_s(k,h)+\int_s^T\Delta^2Z^n_r(k,h)\ud W_r=\bar\eta^n+\int_s^T\bar\alpha^n_r\Delta^2Y^n_r(k,h)\ud r+\int_s^T\bar\beta^n_r\ud r+\int_s^T\bar\gamma^n_r\Delta^2Z^n_r(k,h)\ud r \ , 
\end{equation*}
where $\bar \eta^n$, $\bar \alpha^n$, $\bar \gamma^n$ and $\bar \beta^n$ are suitable coefficients, whose definition is given in the Appendix. 
Exploiting the linear character of the equation (see also  the proof of Proposition \ref{p.diff_2_BSDE})
we need to check that for some $p\ge 2$
\begin{equation*}
  \bE\left\vert \bar\eta^n\right\vert^p+\bE\left(\int_t^T\left\vert\bar\beta^n_r\right\vert\ud r\right)^p\xrightarrow{n\to\infty}0\ .
\end{equation*}
Let us show how to deal with one of the term involved in the definition of $\bar \beta^n$, namely, for $s \in [t, T]$,
\begin{equation*}
\ccd{12} : =\int_s^T\left[D^2_{3,3}G^n\left(r,X^n_r,Y^n_r,Z^n_r\right)\left(D_xZ^n_rk,D_xZ^n_rh\right)-D^2_{3,3}G\left(r,X_r,Y_r,Z_r\right)\left(D_xZ_rk,D_xZ_rh\right)\right]\ud r \ ,
\end{equation*}
where, for sake of consistency, we used the same notation as in the Appendix. 
Then
\begin{align*}
  \bE\left\vert\ccd{12}\right\vert^p&\leq\bE\left(\int_t^T\left\vert \left[D^2_{3,3}G_r(n)-D^2_{3,3}G_r(\cdot)\right]\left(D_xZ^n_rk,D_xZ^n_rh\right)\right\vert\ud r\right)^p\\
                                    &\phantom{\lesssim}+\bE\left(\int_t^T\left\vert D^2_{3,3}G_r(\cdot)\left(D_xZ^n_rk,D_xZ^n_rh-D_xZ_rh\right)\right\vert\ud r\right)^p\\
                                    &\phantom{\lesssim}+\bE\left(\int_t^T\left\vert D^2_{3,3}G_r(\cdot)\left(D_xZ^n_rk-D_xZ_rk,D_xZ_rh\right)\right\vert\ud r\right)^p\\
    &=\bE\left[\overline{\mathrm{F}}^n_{331}\right]^p+\bE\left[\overline{\mathrm{F}}^n_{332}\right]^p+\bE\left[\overline{\mathrm{F}}^n_{333}\right]^p\ .
\end{align*}
By Assumtpion \ref{ass:G}
\begin{align*}
\bE\left[\overline{\mathrm{F}}^n_{331}\right]^p&\lesssim\bE\left(\int_t^T\left(\left\vert J^nX^n_r-X_r\right\vert^\alpha+\left\vert Y^n_r-Y_r\right\vert^\alpha+\left\vert Z^n_r-Z_r\right\vert^\alpha\right)\left\vert D_xZ^n_rk\right\vert\left\vert D_xZ^n_rh\right\vert\ud r\right)^p\\
                                                &\lesssim \bE\left[\left(\sup_{r\in[t,T]}\left\vert J^nX^n_r-X_r\right\vert^{\alpha p}+\sup_{r\in[t,T]}\left\vert Y^n_r-Y_r\right\vert^{\alpha p}\right)\right.\\                                
  &\phantom{\sup_{r\in[0,T]}\Vert_{\aC^\prime}}\left.\times\left(\int_t^T\left\vert D_xZ^n_rk\right\vert^2\ud r\right)^{\nicefrac{p}{2}}\left(\int_t^T\left\vert D_xZ^n_rh\right\vert^2\ud r\right)^{\nicefrac{p}{2}}\right]\\
&\phantom{\lesssim}+\bE\left(\int_t^T\left\vert Z^n_r-Z_r\right\vert^\alpha\left\vert D_xZ^n_rk\right\vert\left\vert D_xZ^n_rh\right\vert\ud r\right)^p\\
                                               &\lesssim\left(\left[\bE\sup_{r\in[t,T]}\left\vert J^nX^n_r-X_r\right\vert^{\bar \nu \alpha p}\right]^{\nicefrac{1}{\bar \nu}}+\left[\bE\sup_{r\in[t,T]}\left\vert Y^n_r-Y_r\right\vert^{\bar \nu\alpha p}\right]^{\nicefrac{1}{\bar \nu}}\right)\\
&\phantom{aaaaaaaaaaaaaaaa\lesssim}
                                                                                                                                                                                                                                                                        \times\left[\bE\left(\int_t^T\left\vert D_xZ^n_rk\right\vert^2\ud r\right)^{\nicefrac{\nu p}{2}}\right]^{\nicefrac{1}{\nu}}\left[\bE\left(\int_t^T\left\vert D_xZ^n_rh\right\vert^2\ud r\right)^{\nicefrac{\nu p}{2}}\right]^{\nicefrac{1}{\nu}}\\
                                               &\phantom{\lesssim}+\bE\left(\int_t^T\left\vert Z^n_r-Z_r\right\vert^\alpha\left\vert D_xZ^n_rk\right\vert\left\vert D_xZ^n_rh\right\vert\ud r\right)^p\ .\\
\end{align*}
The first two terms go to $0$ as $n\rightarrow\infty$ by Lemmas \ref{step1.a} and \ref{step2.a}, respectively.
It remains to estimate
\begin{align*}
\bE&\left(\int_t^T\left\vert Z^n_r-Z_r\right\vert^\alpha\left\vert D_xZ^n_rk\right\vert\left\vert D_xZ^n_rh\right\vert  \ud r\right)^p\lesssim
\bE\left(\int_t^T\left\vert Z^n_r-Z_r\right\vert^\alpha\left\vert D_xZ_rk\right\vert\left\vert D_xZ_rh\right\vert\ud r\right)^p\\
&\qquad\qquad+\bE\left(\int_t^T\left\vert Z^n_r-Z_r\right\vert^\alpha\left\vert D_xZ^n_rk- D_xZ_rk\right\vert\left\vert D_xZ^n_rh- D_xZ_rh\right\vert\ud r\right)^p\\
&\qquad\qquad+\bE\left(\int_t^T\left\vert Z^n_r-Z_r\right\vert^\alpha\left\vert D_xZ_rk\right\vert\left\vert D_xZ^n_rh- D_xZ_rh\right\vert\ud r\right)^p\\
&\qquad\qquad+\bE\left(\int_t^T\left\vert Z^n_r-Z_r\right\vert^\alpha\left\vert D_xZ^n_rk- D_xZ_rk\right\vert\left\vert D_xZ_rh\right\vert\ud r\right)^p
\end{align*}
Thanks to (\ref{eq:stimaZ}), we get that
\begin{align*}
&\bE\left(\int_t^T\left\vert Z^n_r-Z_r\right\vert^\alpha\left\vert D_xZ^n_rk- D_xZ_rk\right\vert\left\vert D_xZ^n_rh- D_xZ_rh\right\vert\ud r\right)^p\\
&\leq\bE\sup_{r\in[t,T]}\left\vert Z^n_r-Z_r\right\vert^\alpha\left(\int_t^T\left\vert D_xZ^n_rk- D_xZ_rk\right\vert\left\vert D_xZ^n_rh- D_xZ_rh\right\vert\ud r\right)^p\\
&\lesssim\left(\bE\int_t^T \left\vert D_xZ^n_rk- D_xZ_rk\right\vert^2\ud r\right)^{\frac{1}{2p}}
\left(\bE\int_t^T\left\vert D_xZ^n_rh- D_xZ_rh\right\vert^2\ud r\right)^{\frac{1}{2p}}\xrightarrow{n\to\infty} 0 \ ,
\end{align*}
where the last convergence follows by Lemma \ref{step3}. 
Similarly, 
\begin{align*}
&\bE\left(\int_t^T\left\vert Z^n_r-Z_r\right\vert^\alpha\left\vert D_xZ_rk\right\vert\left\vert D_xZ^n_rh- D_xZ_rh\right\vert\ud r\right)^p\\
&\leq\bE\sup_{r\in[t,T]}\left\vert Z^n_r-Z_r\right\vert^\alpha\left(\int_t^T\left\vert  D_xZ_rk\right\vert\left\vert D_xZ^n_rh- D_xZ_rh\right\vert\ud r\right)^p\\
&\lesssim\left(\bE\int_t^T \left\vert  D_xZ_rk\right\vert^2\ud r\right)^{\frac{1}{2p}}
\left(\bE\int_t^T\left\vert D_xZ^n_rh- D_xZ_rh\right\vert^2\ud r\right)^{\frac{1}{2p}}\xrightarrow{n\to\infty} 0 \ ,
\end{align*} and we can proceed in the same way for the term
\begin{equation*}
\bE\left(\int_t^T\left\vert Z^n_r-Z_r\right\vert^\alpha\left\vert D_xZ^n_rk- D_xZ_rk\right\vert\left\vert D_xZ_rh\right\vert\ud r\right)^p.
\end{equation*}
It remains to study the convergence of
\begin{equation}
  \label{eq:pezzo}
  \bE\left(\int_t^T\left\vert Z^n_r-Z_r\right\vert^\alpha\left\vert D_xZ_rk\right\vert\left\vert D_xZ_rh\right\vert\ud r\right)^p
\end{equation}
We first notice that for every $\bar{q}\geq 1$
\begin{equation}
  \label{eq:EsupDZ}
  \bE\sup_{r\in[t,T]}\left\vert D_xZ_rh\right\vert^{\bar{q}} <+\infty\ .
\end{equation}
Indeed, by (\ref{eq:identification})
\begin{align*}
  D_xZ_r&=D_x\left[Du\left(t,X_r^{t,x}\right)\Sigma\right]\\
  &=D^2u\left(t,X^{t,x}_r\right)D_xX^{t,x}_r\Sigma\ ,
\end{align*}
and thanks to Corollary \ref{cor:D2u} and \eqref{eq:estDXnorm} we immediately get \eqref{eq:EsupDZ}.
Therefore, to show that \eqref{eq:pezzo} converges to $0$ as $n \to +\infty$, by a standard application of H\"older inequality with respect to $\omega$,  it is enough to prove that
\begin{equation}
  \label{eq:Zn-z}
  \bE\int_t^T\left\vert Z_r^n-Z_r\right\vert^q\ud r\longrightarrow 0 \ ,
\end{equation}
for some  $q\geq 1$ (it actually holds for every $q \ge 1$).
If we write
\begin{align*}
  \bE\int_t^T\left\vert Z^n_r-Z_r\right\vert^{q}\ud r&=\bE\int_t^T\left\vert Du^n\left(r,X_r^n\right)\Sigma-Du\left(r,X_r\right)\Sigma\right\vert^q\ud r\\
  &\leq\bE\int_t^T\left\vert Du^n(r,X_r^{n})\Sigma-Du^n(r,X_r)\Sigma \right\vert^{q}\ud r+\bE\int_0^T\left\vert Du^n(r,X_r)\Sigma-Du(r,X_r^{n})\Sigma\right\vert^{q}\ud r\ ;
\end{align*}
for what concerns the first term, by Corollary~\ref{cor:D2u} we have 
\begin{multline*}
\bE\sup_{r\in[t,T]}\left\vert Du^n(r,X_r^{n})\Sigma -Du^n(r,X_r)\Sigma \right\vert^{q} \\ \lesssim \left[1+\bE\sup_{r\in[t,T]}\left\vert X^{n}_r\right\vert^{2lq}+\bE\sup_{r\in[t,T]}\left\vert X_r\right\vert^{2lq}\right]^{\frac{1}{2}}\left[\bE\sup_{r\in[t,T]}\left\vert X^n_r-X_r\right\vert^{2q}\right]^{\frac{1}{2}}
\end{multline*}
which converges to $0$ as $n$ goes to $+\infty$ by \eqref{eq:estX} and Lemma \ref{step1.a}.\\
As for the second term, the convergence provided by Lemma \ref{step3} yields
\begin{equation*}
  Du^n\left(s, X^{t,x}_s(\omega)\right) \Sigma \xrightarrow{n\to+\infty}{Du\left(s,X^{t,x}_s(\omega)\right)}\Sigma \quad \text{ for a.e. } (s, \omega) \in [t,T] \times \Omega, \  \forall x\in \aC \ .
\end{equation*}
Indeed, for every $y\in \aC$, $D_xY^{s,y}_s$ is deterministic and  \eqref{eq:identification} along with Lemma \ref{step3} imply that for a.e. $s \in [t,T]$
\begin{equation*}
  Du^n(s,y) \Sigma \to Du(s,y)\Sigma \ ;
\end{equation*}
the required convergence follows by the evaluation $y = X^{t,x}_s(\omega) \in \aC$, for a.e. $\omega \in \Omega$.
A final application of Lebesgue dominated convergence theorem, together with \eqref{eq:stima_grad_u} applied to $u^n$ and $u$ (we remark here that the constant in that estimate does not depend on $n$) and the bound \eqref{eq:estDXnorm} yields 
\begin{equation*}
  \bE\int_t^T\left\vert Du^n(r,X_r)\Sigma-Du(r,X_r^{n})\Sigma\right\vert^{q}\ud r\to 0\ ,
\end{equation*}
hence the required convergence in \eqref{eq:Zn-z}. 
\end{proof}

\section{Application to stochastic optimal control}
\label{sec:control}
Here we apply the results obtained in the previous sections to semilinear Kolmogorov equations arising as Hamilton-Jacobi-Bellman (HJB) equations associated to some control problems.\\
We describe the evolution of the state with the forward controlled dynamics in $\sD$
\begin{equation}
\label{eq:SDEX:control}
\begin{system}
  \ud X^u_s=AX^u_s\ud s+B\left(s,X^u_s\right)\ud s+\Sigma u_s \ud s+\Sigma \ud W_s\ ,\quad s\in[t,T]\\
  X^u_t=x\ \text{,}
\end{system}
\end{equation}
where $\Sigma:\bR^{d_1}\to\bR^{d}\times\{0\}\subset\sD$, $B:\sD\to\sD$ is such that $B(\sC)\subseteq\sC$, $u: \Omega \times [0,T] \to \bR^{d_1}$ is the control action and $A$, $W$  are as in Subsection \ref{subsec:FBSDE}. 
As before, the solution to (\ref{eq:SDEX:control}) has to be intended in mild sense and will be denoted also by $X^{u;t,x}$ to emphasize the dependence on the initial data.

Besides equation (\ref{eq:SDEX:control}) we define the cost functional $\mathscr{J}: [0,T] \times \sD \times \R^{d_1} \to \R$
\begin{equation}
  \label{cost:abstract}
\mathscr{J}\left(  t,x,u\right) :=\E\int_{t}^{T}\left[
 L\left(s,X^{u;t,x}_s\right) +Q\left(u_s\right)\right]\ud s+\E\Upsilon\left(X^{u;t,x}_T\right)
\end{equation}
for real-valued functions $L, Q, \Upsilon$, defined respectively on $[0,T]\times\sD$, $\bR^{d_1}$ and $\sD$,  and the  class of admissible controls
\begin{equation*}
\sA := \left\{ u: \Omega \times [0,T] \to \bR^{d_1} \ , \left(  \mathcal{F}_{s}\right)  _{s}\text{-predictable} :  \left\| u\right\|_{L^{\infty}\left((0,T)\times\Omega\right)}<+\infty  \right\} \ .
\end{equation*}
The control problem consists in minimizing the functional $\rJ$ over the admissible controls $u\in\sA$.

\begin{remark}
For example, our control problem arises as infinite-dimensional lifting of a finite-dimensional path-dependent control problem. 
Indeed, let $\xi^u$ be a solution to the path-dependent state equation
\begin{equation}
\label{eq:SDEx:control}
\begin{system} 
\ud\xi^u(s) = b_s(\xi^u_{[0,s]})\ud s +\sigma u_s \ud s+ \sigma \ud W_s\ ,\quad s\in[t,T]\ ,\\
\xi_{[0,t]} = \gamma,
\end{system}
\end{equation}
with $b_s:D\left([0,s]:\bR^d\right)\to\bR$, for every $s \in [t, T]$, $\sigma:\bR^{d_1}\to\bR^d$ and $\gamma\in D\left([0,t];\bR^d\right)$ fixed.
To obtain the forward SDE \eqref{eq:SDEX:control} we just set $x:=L^t\gamma$ (with $L^t$ as in \eqref{eq:Lt}) and define $B$, $\Sigma$ as in  \eqref{eq:Bprod}, \eqref{Sigma}, respectively (see Subsection \ref{subsec:FBSDE} for more details). 
Furthermore, given a path-dependent functional
\begin{equation*}
  l=\{l_s\}_{s\in[t,T]} \ ,\quad l_s:D\left([0,s];\bR^d\right)\to\bR
\end{equation*}
and functions
\begin{equation*}
  q:\bR^{d_1}\to\bR\ ,\quad \phi:D\left([0,T];\bR^d\right)\to\bR \ ,
\end{equation*}
we can define the path-dependent cost functional
\begin{equation}
  \label{cost}
j\left(  t,\gamma,u\right)  :=\mathbb{E}\int_{t}^{T}\left[l_s\left(\xi^{u}_{[0,s]}\right) +q\left(u_s\right)\right]\ud s+\mathbb{E}\phi\left(x^{u}_{[0,T]}\right). 
\end{equation}
Introducing the liftings
\begin{gather*}
  L:[0,T]\times\sD\to\bR\ ,\quad L\left(s,x\right)=l_s\left(M_sx\right)\ ,\\
  Q=q\ ,\quad\Upsilon:\sD\to\bR\ ,\quad \Upsilon(x)=\phi\left(M_Tx\right) \ ,
\end{gather*}
we exactly recover (\ref{cost:abstract}) with the property $\rJ(t,x,u)=j(t,M_tx,u)$, for every $(t,x,u)\in[0,T]\times\sD\times \sA$. 
For the sake of generality, in the following we deal with abstract problems in $\sD$, without exploiting the path-dependent structure behind it.
\end{remark}

The following assumptions on the optimal control problem will be in force throughout.
\begin{assumption}
  \label{ip costo}
There exists $\alpha \in (0,1)$ and constants $a > 0$, $b,c, R, C \ge 0$ such that 
\begin{enumerate}[label=(J.\Roman{*})]
\item $Q:\bR^{d_1}\rightarrow\bR$ is continuous and
\begin{equation}\label{g-coercivo}
 \vert Q(u)\vert\geq a \vert u\vert^2-b, \quad \forall u\in \R^{d_1}\ , \quad \vert Q(u)\vert\leq c \vert u\vert^2, \text{  for   } \vert u\vert \geq R \ ;
\end{equation}
\item $L: [0,T] \times \sD  \to \R$ is continuous, $L(s,\cdot)$ belongs to $C^{2,\alpha}(\sD,\mathbb{R})$ for every  $s\in[0,T]$ and 
\begin{equation*}\label{ip-psi-contr}
\abs{ L(s,x)} \le C(1 + |x|) \ , \quad  \abs{D L(s,x)} + \abs{D^2 L(s,x)} \leq C \ , \quad \forall s \in [0,T] \ ;
\end{equation*}
\item $\Upsilon: \sD  \to \R$ belongs to $C^{2,\alpha}(\sD,\mathbb{R})$ and 
\begin{equation*}
\abs{ \Upsilon(x)} + \abs{D \Upsilon(x)} + \abs{D^2 \Upsilon(x)} \leq C \ ;
\end{equation*}
\end{enumerate}
\end{assumption}

The Hamiltonian of the problem is defined as 
\begin{equation*}
 \sH\left(z\right)  :=\inf_{u\in \R^{d_1}}\left\{Q\left(u\right)+zu\right\}\quad \forall z\in \R^{d_1} ,
\end{equation*}
and we denote with $\Gamma(z)$ the set of minimizers
\begin{equation*}
\Gamma (z)=\left\lbrace u: \sH(z)=Q(u)+zu\right\rbrace\ .
\end{equation*}
\begin{assumption}
\label{ass:H}
The Hamiltonian $\sH: \R^{d_1} \to \R$ belongs to $C^{2,\alpha}(\R^{d_1})$.
\end{assumption}
Note that $\sH(z) = - Q^\ast(-z)$, $Q^\ast$ being the Fenchel conjugate of $Q$, and from Assumption \ref{ip costo} it is easily seen that $\sH$ has quadratic growth. 
Moreover, a sufficient condition for Assumption \ref{ass:H} to hold is to require 
$Q \in C^3(\R^{d_1})$ strictly convex (along with the superlinearity given by Assumption \ref{ip costo}), see e.g. \cite[Prop.~2.6.3]{fathi2008weak} for a general result.	

\smallskip

Denoting with $\Psi:[0,T]\times \sD\times \R\rightarrow \R$ the map $\Psi(t,x,z):= L(t,x)+\sH(z)$, let us now introduce the BSDE 
\begin{equation}
\label{eq:BSDE-second-contr}
\begin{system}
\ud Y^{t,x}_s = \Psi\left(s,X^{t,x}_s,Z^{t,x}_s\right) \ud s + Z^{t,x}_s \ud W_s \ , \qquad s \in [t,T] \ ,\\
Y^{t,x}_T = \Upsilon\left(X^{t,x}_T\right) \ ;
\end{system}
\end{equation}
where $X^{t,x}_\cdot$ solves the forward state equation \eqref{eq:SDEX:control} for every $t,x \in [0,T] \times \sD$ with $u = 0$.

\begin{proposition}
  \label{prop:2derFHFrechet-contr} Let Assumptions \ref{ass:B}, \ref{ip costo} and \ref{ass:H} hold; then for every $(t,x) \in [0,T] \times \sD$ the BSDE \eqref{eq:BSDE-second-contr}
admits a unique solution $(Y^{t,x},Z^{t,x}) \in \sK_p$, for every $p >1$. 
Moreover, the map
\begin{equation*}
x\mapsto Y^{t,x},\; \sD\rightarrow L^p\left(\Omega; C\left([0,T];\R\right)\right)
\end{equation*}
is twice differentiable and there exists $K \ge 0$ such that 
\begin{equation*}
|Z^{t,x}_s| \leq K |\Sigma|\ , \quad \forall s \in [t,T], \, \bP\text{-a.s. }\ .
\end{equation*}
In addition, if $|D^2B(s,x)| \leq C$ for every $(s,x) \in [0,T] \times \sD$, there exists $c\ge 0$ such that
\begin{equation}
\label{bound:der2}
 \E\sup_{s\in[t,T]}\| D^2_{x}Y^{t,x}_s\| \leq c \ .
\end{equation}
 \end{proposition}
\begin{proof}
Let us start by setting $\sH_M=\sH(\rho_M(\cdot))$, where $\rho_M(z)$ is a smooth function such that $\rho(z)=z$ if $|z|<M$ and  $\rho(z)=0$ if $|z|>M+1$. 
The truncated Hamiltonian $\sH_M \in C^{2,\alpha}(\R^{d_1})$ has bounded derivatives, $\Psi_M(t,x,z):=L(t,x)+\sH_M(z)$ complies with Assumption \ref{ass:G} and thanks to Proposition \ref{p:stimeBSDE} the BSDE
\begin{equation} \label{eq:BSDE-second-contr-M}  \left\{\begin{array}{l}
 \ud Y^{M;t,x}_s=\Psi_M\left(s, X^{t,x}_s,Z^{M;t,x}_s\right)\ud s+Z^{M;t,x}_s\ud W_s \ ,
 \quad s\in [t,T] \ ,
  \\\dis
  Y^{M;t,x}_T=\Upsilon(X_T^{t,x}) \ ,
\end{array}\right.
\end{equation}
admits a unique solution $(Y^{M;t,x},Z^{M;t,x})$ in $\mathcal{K}_p$. 
In view of Assumptions \ref{ip costo} and \ref{ass:H}, the application of Propositions 
\ref{p:diff1} and \ref{p.diff_2_BSDE} guarantees that the map
\begin{equation*}
x\mapsto Y^{M;t,x},\; \sD \rightarrow L^p\left(\Omega; C\left([0,T];\R\right)\right)
\end{equation*}
is twice differentiable.
Furthermore, thanks to Remark \ref{rem:BSDE:bdd} it follows that 
\begin{equation*}
|Z^{M;t,x}_s| \leq K |\Sigma|\ , \quad \forall s \in [t,T], \, \bP\text{-a.s.}\ ,
\end{equation*}
for some constant $K \ge 0$ which is \emph{independent} on $M$.
Thus, choosing $M > K|\Sigma|$, it follows that $\sH_M(Z^{t,x}_s) = \sH(Z^{t,x}_s)$ for every $s \in [t,T]$ and the pair 
$\left(Y^{M;t,x},Z^{M;t,x}\right)$ solves also equation \eqref{eq:BSDE-second-contr}.
By the uniqueness of solutions of \eqref{eq:BSDE-second-contr} with bounded second component it easily follows that $\left(Y^{M;t,x},Z^{M;t,x}\right) \equiv \left(Y^{t,x},Z^{t,x}\right)$ whenever $M > K|\Sigma|$, hence the boundedness of $Z^{t,x}$.

Concerning estimate \eqref{bound:der2}, fixing $\bar M > K|\Sigma|$ big enough we firstly observe that $D^i_{x}Y^{\bar M;t,x} = D^i_{x}Y^{t,x}$, for $i = 1,2$, so that we can directly work with $\sH$ instead of $\sH_{\bar M}$. 
Then, for every $h \in \sD$, the application of estimate \eqref{estimate_linear} to the linear BSDE solved by the pair $\left(DY^{t,x}h,DZ^{t,x}h\right)$, immediately gives $D_xZ^{t,x} \in L^p(\Omega; L^2(0,T;\R^{d_1}))$.
Moreover, for every $(k,h) \in \sD$ the equation for the second derivatives reads as 
\begin{equation}\label{der2contr}
\begin{aligned}
D_x^2 &Y^{t,x}_s(k,h) + \int_s^T D_x^2Z^{t,x}_r(k,h) \ud W(r) = \Xi_T(k,h)  - \int_s^T \big[ D^2L(r,X^{u;t,x})\left(D_xX^{t,x}_rk,D_xX^{t,x}_rh\right) \\
&\quad + DL(r,X^{u;t,x})D^2_xX^{t,x}_r(k,h)  + D^2 \sH(Z^{t,x}_r)\left(D_xZ^{t,x}_rk,D_xZ^{t,x}_rh\right)\big] \ud r  \\
& - \int_s^T \left[ D\sH(Z^{t,x}_r)D_x^2Z^{t,x}_r(k,h)   \right] \ud r  \ ,
\end{aligned}
\end{equation}
where $\Xi_T(k,h):= D^2\Upsilon\left(X^{t,x}_T\right)\left(D_xX^{t,x}_Tk,D_xX^{t,x}_Th\right) + D\Upsilon(X^{t,x}_T)D^2_xX^{t,x}_T(k,h)$ is uniformly bounded thanks to Assumption \ref{ip costo} and  estimate \eqref{eq:estDXnorm}.\\
The boundedness of $DL, D^2L$ and the uniform bound on $D^2_xX^{t,x}$ given by \eqref{eq:estD2Xnorm} (with $m = 0$) finally guarantee the validity of \eqref{bound:der2} by the application of estimate \eqref{estimate_linear} to equation \eqref{der2contr}. 
This concludes the proof.
\end{proof}

For every $(t,x) \in [0,T] \times \sD$, let us now introduce the value function $v$ associated to the cost functional $\rJ$: 
\begin{equation*}
 v\left(  t,x\right)  =\inf_{u\in\mathcal{A}}\rJ\left(  t,x,u\right) \ ,
 \end{equation*}
and consider the associated HJB equation 
\begin{equation}\label{eq:hjb}
  \begin{system}
  \frac{\partial v}{\partial t}(t,x) + Dv(t,x)\left[Ax + B(t,x)\right] + \frac{1}{2}\tr_{\bR^d} \left[\Sigma\Sigma^\ast D^2v(t,x)\right] = \Psi(t,x,Dv(t,x)\Sigma) \ ,\\
  v(T,\cdot) = \Upsilon \ .
\end{system}
\end{equation} 
This is a semilinear Kolmogorov equation with the same structure as equation \eqref{eq:kolm:diff:nonlinear}, for which we already obtained a wellposedness result.

\begin{proposition}\label{p:HJB}
Let $B$ satisfies Assumptions \ref{ass:B}, \ref{ass:dphi} with $E = \sD$ and suppose that $B$ maps $\sC$ into itself.
Let also Assumptions \ref{ip costo}, \ref{ass:H} hold and suppose that $\Upsilon$ has one-jump-continuous Fr\'echet differential of first and second order on $\aC \subset \sD$.
Then the function $v(t,x):=Y_t^{t,x}$ (where $Y^{t,x}$ solves \eqref{eq:BSDE-second-contr}) is a classical solution to the HJB equation \eqref{eq:hjb} in the sense of Definition \ref{def:sol_D}, for every $t \in [0,T]$ and $x \in \aC^1$.
\end{proposition}

\begin{proof}
Recall that $\Psi_M:= L(t,x)+\sH_M(z)$ satisfies Assumtpion \ref{ass:G}. 
Hence, for $M > K |\Sigma|$, Theorem \ref{t:main} and Proposition \ref{prop:2derFHFrechet-contr} guarantee that  $v(t,x):=Y_t^{M;t,x} = Y_t^{t,x}$ is a classical solution of the HJB equation
\begin{equation*}
  \begin{system}
  \frac{\partial v}{\partial t}(t,x) + Dv(t,x)\left[Ax + B(t,x)\right] + \frac{1}{2}\tr_{\bR^d} \left[\Sigma\Sigma^\ast D^2v(t,x)\right] = \Psi_M(t,x,Dv(t,x)\Sigma)  \ , \\
  v(T,\cdot) = \Upsilon \ .
\end{system}
\end{equation*} 
Since $|Dv(t,x)\Sigma| \le K |\Sigma|$, $\Psi_M \equiv \Psi$ and $v$ is also a  classical solution of \eqref{eq:hjb} in the sense of Definition \ref{def:sol_D}.

\end{proof}


In order to derive the so-called {\em fundamental relation} for the value function, it is useful to introduce a family of auxiliary problems.
 For every $\Lambda\in\bR$, we define 
\begin{equation*}
\mathcal{A}^\Lambda := \left\{ u: \Omega \times [0,T] \to \bR^{d_1} \ , \left(  \mathcal{F}_{s}\right)  _{s}\text{-predictable} :  \left\| u\right\|_{L^{\infty}\left((0,T)\times\Omega\right)} \leq \Lambda  \right\} \ ,
\end{equation*}
with the corresponding value function  and Hamiltonian 
\begin{equation*}
 v^\Lambda\left(  t,x\right)  =\inf_{u \in\mathcal{A}^\Lambda}\rJ\left(  t,x,u\right),  \qquad \quad  \sH^\Lambda\left(z\right)  =\inf_{u\in \R^{d_1}, |u|\leq  \Lambda}\left\{  Q\left(u\right)+zu\right\}\quad \forall z\in \R^{d_1} \ .
\end{equation*}

\begin{remark}
Note that $\sH^\lambda$ is Lipschitz but,  
 in general, does not belong to $C^{2,\alpha}(\R^{d_1})$.
A counterexample is given by $ Q(u)=\frac{1}{2}\vert u\vert^2$, for which $ - \sH(z)= \frac{1}{2}\vert z\vert^2$ and
  \begin{equation}
 \label{hamiltonian-controes}
 - \sH^\Lambda(z)=\left \lbrace 
 \begin{array}{ll}
  \frac{1}{2}\vert z\vert^2&\vert z\vert \leq \Lambda \ ,\\
  \Lambda\vert z\vert-\frac{\Lambda^2}{2}&\vert z\vert >\Lambda \ ;
 \end{array}
 \right.
\end{equation}
which is not  $C^2$-regular  on $|z| = \Lambda$.
\end{remark}

\begin{proposition}\label{prop rel fond} 
Under the same assumptions of Proposition \ref{p:HJB}, for every $t\in\left[0,T\right]$, $x\in \aC^1$ and for every $u\in\sA$ it holds
\begin{equation}\label{backward rel fond}
v\left( t,x\right) =\rJ\left(t,x,u\right)+\mathbb{E}\int_{t}^{T}\left[ \sH\left( Dv\left(X_s^{u}\right)\Sigma\right)-Dv\left(X_s^{u}\right)\Sigma u_s -Q( u_s)\right] \ud s \ ,
\end{equation}
so that $v\left( t,x\right) \le \rJ\left( t,x,u \right) $ and equality holds if and only if $u_s\in \Gamma \left(D v (s,X^u_s)\Sigma\right)$ for a.e. $s\in[t,T]$, $\omega \in \Omega$.\\
Furthermore, if $|D^2B(s,x)| \leq C$ for every $(s,x) \in [0,T] \times \sD$, and $\Gamma_{0}: \R^{d_1}\to \bR^{d_1}$ is a measurable selection of $\Gamma$ with $\Gamma_0$ Lipschitz continuous, then the closed-loop equation
\begin{equation}
\begin{system}
  \ud X^u_s=AX^u_s\ud s+B\left(s,X^u_s\right)\ud s+\Sigma\Gamma_{0}\left(Dv(s,X^u_{s})\Sigma\right) \ud s+\Sigma \ud W_s\ ,\quad s\in[t,T] \ ,\\
  X^u_t=x\ \text{,}
\end{system}
\label{closed loop eq}
\end{equation}
admits a unique solution denoted by $X^*$, and the pair $\left(X^*_\cdot, \Gamma_{0}\left(D v(\cdot,X_{\cdot}^{*})\Sigma\right)\right)$
is optimal.
\end{proposition}
\begin{proof}
For $M> K|\Sigma|$ there exists $\Lambda \ge 0$ such that $\sH_M(z)=\sH^\Lambda(z)$ for every $|z|\leq M$. 
Moreover, there exists an increasing function $\tilde\rho: \mathbb{R}^+\rightarrow \mathbb{R}^+$ with
$\lim_{\Lambda\rightarrow +\infty }\tilde\rho(\Lambda)=+\infty$ such that  $\sH^\Lambda(z)=\sH(z)$ for every $|z|<\tilde\rho(\Lambda)$.
Thus, for $\Lambda, M \ge 0$ with $\tilde\rho(\Lambda) > M $ it holds 
\begin{equation}\label{rel:H}
  \sH_M(z) = \sH^\Lambda(z)=\sH(z) \ , \qquad \forall |z| \le M \ . 
\end{equation}
Thanks to Proposition \ref{prop:2derFHFrechet-contr}, $\left(Y^{t,x},Z^{t,x}\right)$ solves equation \eqref{eq:BSDE-second-contr} either with $\Psi$, $\Psi_M$ or $\Psi^\Lambda:=L+\sH^\Lambda$.
Then, $v(t,x):= Y^{t,x}_t$ is a classical solution of \eqref{eq:hjb}, and by  \cite[Thm.~7.2]{fuhrman2002nolinear} (see also \cite[Section 6]{fuhrman2010stochastic}) it can be easily proved that equality \eqref{backward rel fond} holds for every $u\in \mathcal{A}^\Lambda$.
The extension to any $u\in \mathcal{A}$ follows again by \eqref{rel:H}.

For what concerns the closed loop equation \eqref{closed loop eq}, existence of a (unique) solution follows by the Lipschitz continuity of the selection $\Gamma_0$ along with the estimate \eqref{bound:der2} applied to $D^2v$ (which in turn ensures the Lipschitz continuity of $Dv$).
\end{proof}


\section{Appendix}

We collect here a detailed version of the proofs of Proposition \ref{p.diff_2_BSDE}, Lemma \ref{step1.b} and Lemma \ref{step4}.

\begin{proof}[Proof of  Proposition \ref{p.diff_2_BSDE}]
We detail here the  Fr\'echet differentiability of the map $x \mapsto \left( D_xY^{t,x} h, D_xZ^{t,x} h \right)$. 
We fix $h,k \in E$ and use the equations solved by
$D_xY^{t,x+k}_s h,  D_xY^{t,x}_s h$ and $F^{t,x}_s(h,k)$ to write:
\begin{align*}
&\left[ D_xY^{t,x+k}_s h - D_xY^{t,x}_s h - F^{t,x}_s (h,k) \right] + \int_s^T \left[ D_xZ^{t,x+k}_s h - D_xZ^{t,x}_s h - H^{t,x}_s(h,k) \right] \ud W(r) \\
&=\left[ U_T^{t,x+k} h - U_T^{t,x} h - F^{t,x}_T(k,h)\right]  \\
&\quad -\int_s^T \left[ D_1G_r(t,x +k)D_xX^{t,x+k}_r h - D_1G_r(t,x)D_xX^{t,x}_r h  -  L_{1;r}^{t,x}(k,h)\right] \ud r\\
&\quad -\int_s^T \left[D_2G_r(t,x+k)D_xY^{t,x+k}_rh  - D_2G_r(t,x)D_xY^{t,x}_r h -  L_{2;r}^{t,x}(k,h) - D_2G_r(t,x)F^{t,x}_r(k,h)\right]\ud r\\
&\quad -\int_s^T \left[D_3G_r(t,x+k)D_xZ^{t,x+k}_rh  - D_3G_r(t,x)D_xZ^{t,x}_r h  -  L_{3;r}^{t,x}(k,h) - D_3G_r(t,x)H^{t,x}_r(k,h)\right]\ud r\\ 
& =\left[ U_T^{t,x+k} h - U_T^{t,x} h - F^{t,x}_T(k,h)\right]  \\
&\quad -\int_s^T \left[ D_1G_r(t,x +k)D_xX^{t,x+k}_r h - D_1G_r(t,x)D_xX^{t,x}_r h  -  L_{1;r}^{t,x}(k,h)\right] \ud r\\
&\quad -\int_s^T \left[ \left(D_2G_r(t,x+k) - D_2G_r(t,x) \right)D_xY^{t,x+k}_r h  - L_{2;r}^{t,x}(k,h) \right]\ud r \\
&\quad -\int_s^T D_2G_r(t,x)  \left[D_xY^{t,x+k}_r h - D_xY^{t,x}_r h - D_2G_r(t,x)F^{t,x}_r(k,h)\right]\ud r\\
&\quad -\int_s^T \left[ 	\left(D_2G_r(t,x+k) - D_3G_r(t,x) \right)D_xZ^{t,x+k}_r h  - L_{3;r}^{t,x}(k,h) \right]\ud r \\
&\quad -\int_s^T D_3G_r(t,x)  \left[D_xZ^{t,x+k}_r h - D_xZ^{t,x}_r h  - D_3G_r(t,x)H^{t,x}_r(k,h)\right]\ud r.\\
\end{align*}

If we divide both left and right hand side by $|k|$  and we shorten the notation as in \eqref{notation_first_der} we end up with
\begin{equation*}
\Upsilon^k_s+\int_s^T\Psi^k_r\ud W_r=\Upsilon^k_T- \int_s^T \left( D_2G_r(t,x)\Upsilon^k_r + D_3G_r(t,x)\Psi^k_r - M^k(r)\ud r \right) \ud r \ .
\end{equation*}
Exploiting estimate \eqref{estimate_linear} in Lemma \ref{l:linearBSDE} we have to show that 
\begin{equation*}
\lim_{k \to 0} \sup_{|h|=1}\left[\bE\left\vert \Upsilon^k_T\right\vert^p+\bE\left(\int_t^T\left\vert M^k(r)\right\vert\ud r\right)^p \right] =0 \ .
\end{equation*}
Firstly, dealing with the final datum, let us set for convenience $\bar{\Upsilon}^k_T: =|k| \Upsilon^k_T$; then almost surely
\begin{equation}
\begin{split}
\bar{\Upsilon}^k_T&= U_T^{t,x+k} h - U_T^{t,x} h - F^{t,x}_T(k,h) \\
&= D\Phi(X^{t,x+k}_T)D_xX^{t,x+k}_Th - D\Phi(X^{t,x}_T)D_xX^{t,x}_Th \\
&\quad - D^2\Phi\left(X^{t,x}_T\right)\left(D_xX^{t,x}_Tk,D_xX^{t,x}_Th\right) - D\Phi(X^{t,x}_T)D^2_xX^{t,x}_T(k,h)\\
&= \left[ D\Phi(X^{t,x+k}_T) - D\Phi(X^{t,x}_T)\right] D_xX^{t,x+k}_Th  - D^2\Phi\left(X^{t,x}_T\right)\left(D_xX^{t,x}_Tk,D_xX^{t,x}_Th\right)\\
&\quad +D\Phi(X^{t,x}_T)\left[ D_xX^{t,x+k}_Th - D_xX^{t,x}_Th  - D^2_xX^{t,x}_T(k,h)\right] \\
&=: \bar\Upsilon^{k}_{T,1} + \bar\Upsilon^{k}_{T,2}  \ .
\end{split}
\end{equation}
For what concerns $\Upsilon^{k}_{T,1}$ we use Assumption \ref{ass:Phi2} to write
\begin{equation}
\begin{split}
\bar\Upsilon^{k}_{T,1} &= \int_0^1 \left[ D^2\Phi (\lambda X_T^{t,x+k} + (1-\lambda)X_T^{t,x}) - D^2\Phi (X_T^{t,x}) \right] \left( X_T^{t,x+k} - X_T^{t,x}, D_xX_T^{t,x+k} h \right) \ud \lambda \\
&\quad + D^2\Phi (X_T^{t,x})\left( X_T^{t,x+k} - X_T^{t,x} - D_xX^{t,x}_T k, D_xX_T^{t,x+k} h \right) \\
&\quad + D^2\Phi (X_T^{t,x})\left(D_xX^{t,x}_T k, D_xX_T^{t,x+k} h - D_xX_T^{t,x} h \right) \\
&=: \bar\Upsilon^{k}_{T,11} + \bar\Upsilon^{k}_{T,12} + \bar\Upsilon^{k}_{T,13} \ .
\end{split}
\end{equation}
Using $\alpha$-H\"older continuity of $D^2\Phi(\cdot)$ (see Assumption \ref{ass:Phi2}), $\Upsilon^{k}_{T,11}$ can be estimated as follows:
\begin{equation}\label{est_ups11}
\begin{split}
\E \left|\Upsilon^{k}_{T,11}\right|^p &\lesssim \left(\E \left|  X_T^{t,x+k} - X_T^{t,x} \right|^{4p\alpha} \right)^{\nicefrac{1}{4}} 
 \E \left( \left| D_xX_T^{t,x+k} \right|^{4p} \right)^{\nicefrac{1}{4}} |h|^p\\
&\qquad \times \left[ \E \left(\frac{\left|  X_T^{t,x+k} - X_T^{t,x}\right|}{|k|}\right)^{2p} \right]^{\nicefrac{1}{2}} \longrightarrow 0 \ , \qquad\text{ if } |k| \to 0 \ ,
\end{split}
\end{equation}
thanks to the Fr\'echet differentiability of $x \mapsto X_T^{t,x}$ and estimate \eqref{eq:estDX} (see Theorem \ref{thm:X}).
A similar argument can be used for $\Upsilon^{k}_{T,12}$ and $\Upsilon^{k}_{T,13}$:
\begin{equation}
\begin{split}
\E \left|\Upsilon^{k}_{T,12}\right|^p &\lesssim \left( 1+ \E|X_T^{t,x}|^{4p} \right)^{\nicefrac{1}{4}} 
 \left[\E \left( \frac{\left|X_T^{t,x+k} - X_T^{t,x}  - D_xX_T^{t,x} h\right|}{|k|} \right)^{4p} \right]^{\nicefrac{1}{4}} \\
&\qquad \times |h|^p\left[ \E\left| D_x X_T^{t,x+k} \right|^{2p} \right]^{\nicefrac{1}{2}}  \longrightarrow 0 \ , \qquad\text{ if } |k| \to 0 \ ,
\end{split}
\end{equation} 
\begin{equation}
\begin{split}
\E \left|\Upsilon^{k}_{T,13}\right|^p &\lesssim \left( 1+ \E|X_T^{t,x}|^{4p} \right)^{\nicefrac{1}{4}} 
 \E \left( \left|D_xX_T^{t,x+k} - D_xX_T^{t,x} \right|^{4p} \right)^{\nicefrac{1}{4}} |h|^p\\
&\qquad \times \left[ \E\left\| D_x X_T^{t,x} \right\|^{2p} \right]^{\nicefrac{1}{2}} \longrightarrow 0 \ , \qquad\text{ if } |k| \to 0 \ ,
\end{split}
\end{equation} 
where we used Assumption \ref{ass:Phi2}, estimate \eqref{eq:estX}, continuity of the map $x \mapsto D_xX^{t,x}_T$ and the bound \eqref{eq:estDX}. \\
Finally, $\Upsilon^{k}_{T,2}$ can be shown to go to zero, as $|k| \to 0$, thanks to the second Fr\'echet differentiability of $x \mapsto X^{t,x}_T$ along with the bound  \eqref{eq:estX}. 
Indeed
\begin{equation*}
\begin{split}
\sup_{|h|=1}\E \left|\Upsilon^{k}_{T,2}\right|^p &\lesssim \left( 1+ \E|X_T^{t,x}|^{2p} \right)^{\nicefrac{1}{2}} 
 \left[\E \left( \frac{\left| D_xX^{t,x+k}_T - D_xX^{t,x}_T  - D^2_xX^{t,x}_Tk\right|}{|k|} \right)^{2p} \right]^{\nicefrac{1}{2}} \longrightarrow 0 \ , \quad\text{ if } |k| \to 0 \ .
\end{split}
\end{equation*} 
In the rest of the proof we only concentrate on $M^k_3$, being the more intricate term (it involves the derivatives $D_xZ$ for which the topology of the estimates is weaker).
The other terms can be treated in a similar, and simpler, way.\\
To lighten the notation we introduce the shorthand $\sM^k_3:= -|k| M^k_i(r)$.
\begin{equation}
\begin{split}
\sM_3^k= &\left[D_3G(r,X_r^{t,x+k}, Y_r^{t,x+k}, Z_r^{t,x+k}) - D_3G(r,X_r^{t,x+k}, Y_r^{t,x+k}, Z_r^{t,x})\right] D_xZ^{t,x+k}_rh \\
&\qquad + D^2_{3,3}G_r(t,x)\left(D_xZ^{t,x}_rk,D_xZ^{t,x}_rh\right) \\
&+ \left[ D_3G(r,X_r^{t,x+k}, Y_r^{t,x+k}, Z_r^{t,x}) - D_3G(r,X_r^{t,x+k}, Y_r^{t,x}, Z_r^{t,x}) \right] D_xZ^{t,x+k}_rh\\
&\qquad + D^2_{3,2}G_r(t,x)\left(D_xY^{t,x}_rk,D_xZ^{t,x}_rh\right)  \\
&+ \left[ D_3G(r,X_r^{t,x+k}, Y_r^{t,x}, Z_r^{t,x}) - D_3G(r,X_r^{t,x}, Y_r^{t,x}, Z_r^{t,x}) \right] D_xZ^{t,x+k}_rh\\
&\qquad + D^2_{1,3}G_r(t,x)\left(D_xX^{t,x}_rk,D_xZ^{t,x}_rh\right)  \\
=:&\, \sM^k_{31} + \sM^k_{32} + \sM^k_{33} \ .
\end{split}
\end{equation}
Thanks to the Fr\'echet differentiability of $D_3G$, the first term can be rewritten as
\begin{equation*}
\begin{split}
\sM_{31}^k&=  \int_0^1 \left[ D^2_{3,3} G(r,X_r^{t,x+k}, Y_r^{t,x+k}, \lambda Z_r^{t,x+k} + (1-\lambda)Z_r^{t,x}) - D^2_{3,3}G_r(t,x)\right] \\
&\qquad \qquad (Z_r^{t,x+k} - Z_r^{t,x}, D_xZ^{t,x+k}_rh) \ud \lambda + D^2_{3,3}G_r(t,x) \left(Z_r^{t,x+k} - Z_r^{t,x}, D_xZ^{t,x+k}_rh\right) \\
&-  D^2_{3,3}G_r(t,x) \left(D_xZ^{t,x}_r k, D_xZ^{t,x+k}_rh \right)  +  D^2_{3,3}G_r(t,x) \left(D_xZ^{t,x}_r k, D_xZ^{t,x+k}_rh - D_xZ^{t,x}_rh \right)
\end{split}
\end{equation*}
from which
\begin{equation*}
\begin{split}
\sM_{31}^k=& \int_0^1 \left[ D^2_{3,3} G(r,X_r^{t,x+k}, Y_r^{t,x+k}, \lambda Z_r^{t,x+k} + (1-\lambda)Z_r^{t,x}) - D^2_{3,3}G(r,X_r^{t,x+k}, Y_r^{t,x+k}, Z_r^{t,x})\right] \\ 
&\qquad \qquad (Z_r^{t,x+k} - Z_r^{t,x}, D_xZ^{t,x+k}_rh) \ud \lambda  \\
&+ \left[ D^2_{3,3} G(r,X_r^{t,x+k}, Y_r^{t,x+k}, Z_r^{t,x} ) - D^2_{3,3}G(r,X_r^{t,x+k}, Y_r^{t,x}, Z_r^{t,x})\right](Z_r^{t,x+k} - Z_r^{t,x}, D_xZ^{t,x+k}_rh) \\
&+ \left[ D^2_{3,3} G(r,X_r^{t,x+k}, Y_r^{t,x}, Z_r^{t,x} ) - D^2_{3,3}G_r(t,x)\right](Z_r^{t,x+k} - Z_r^{t,x}, D_xZ^{t,x+k}_rh) \\
&+  D^2_{3,3}G_r(t,x) \left(Z_r^{t,x+k} - Z_r^{t,x} - D_xZ^{t,x}_rk, D_xZ^{t,x+k}_rh\right)\\
&+  D^2_{3,3}G_r(t,x) \left(D_xZ^{t,x}_r k, D_xZ^{t,x+k}_rh - D_xZ^{t,x}_rh \right)  \\
=:& \sM^k_{311} + \sM^k_{312} + \sM^k_{313} +\sM^k_{314} + \sM^k_{315} \ .
\end{split}
\end{equation*}
The term $\sM^k_{311}$ has been already discussed in the main proof. 
For what concerns $\sM^k_{314}$ and $\sM^k_{315}$ we have
\begin{equation}
\begin{split}
\E &\left(\int_t^T |M^k_{314}(r)| \ud r \right)^p \lesssim \E \left( \int_t^T \left| Z_r^{t,x+k} - Z_r^{t,x} - D_xZ^{t,x}_rk \right|  \left|D_xZ^{t,x+k}_rh\right| \ud r \right)^p \\ 
&\lesssim \left[ \E\left( \int_t^T \frac{ \left| Z_r^{t,x+k} - Z_r^{t,x} - D_xZ^{t,x}_rk \right|^2}{|k|^2} \ud r \right)^{p} \right]^{\nicefrac{1}{2}} \left[\E \left( \int_t^T \left\| D_xZ^{t,x+k}_r\right\|^2 \ud r \right)^{p}\right]^{\nicefrac{1}{2}}|h|   \; \longrightarrow 0 \ , \qquad\text{ if } |k| \to 0 \ ,
\end{split}
\end{equation}
uniformly in $h$, $|h|=1$, thanks to the Fr\'echet differentiability of $x \mapsto Z^{t,x}$ and estimate \eqref{eq:estBSDE}. Finally
\begin{equation}
\begin{split}
\E &\left(\int_t^T |M^k_{315}(r)| \ud r \right)^p \lesssim \left[ \E\left( \int_t^T \left| D_xZ^{t,x+k}_rh - D_xZ^{t,x}_rh \right|^2 \ud r \right)^{p} \right]^{\nicefrac{1}{2}} \\
&\times  \left[\E \left( \int_t^T \left\| D_xZ^{t,x+k}_r\right\|^2 \ud r \right)^{p}\right]^{\nicefrac{1}{2}}   \; \longrightarrow 0 \ , \qquad\text{ if } |k| \to 0 \ ,
\end{split}
\end{equation}
thanks to the continuity of the map $x \mapsto D_xZ^{t,x}h$ from $E \to \sK_p$ given in Proposition \ref{p:diff1} and the estimate \eqref{eq:estBSDE}.

Coming back to $\sM^k_3$, the terms $\sM^k_{32}$ and $\sM^k_{33}$ can be treated in the same way as $\sM^k_{31}$ (in this cases there is no need for the identification theorem), taking advantage of the estimates \eqref{eq:est:bsde_Y_Z}, \eqref{eq:estBSDE} for $Y^{t,x},DY^{t,x}$ in $L^p(C([0,T];\R))$ and $L^p(C([0,T];E'))$, respectively. 
For what concerns the terms $\sM^k_1$ and $\sM^k_2$ the argument of the proof is exactly the same, due to the symmetry of the construction.

Summing up the above computations we finally have
\begin{equation*}
\sup_{|h|=1} \left[ \bE\left\vert \Upsilon^k_T\right\vert^p + 
\E \left(\int_t^T |M^k(r)| \ud r \right)^p \right] \longrightarrow 0 \ , \qquad\text{ if } |k| \to 0 \ ,
\end{equation*}
which is the required result.\\

\end{proof}  

\begin{proof}[Prof of Lemma \ref{step1.b}]
Let us prove the second part of the statement, concerning second order derivatives.
For $0\leq t\leq\tau\leq s\leq T$ we have
\begin{align*}
  \left\vert D_x^2X^n_\tau\left(k,h\right)\right.&\left.-D_x^2X_\tau(k,h)\right\vert\\
  &\lesssim\int_t^\tau\left\vert D^2B^n\left(r,X_r^n\right)\left(D_xX^n_rk,D_xX^n_rh\right)-D^2B\left(r,X_r\right)\left(D_xX_rk,D_xX_rh\right)\right\vert\ud r\\
                                                 &\phantom{\lesssim}+\int_t^\tau\left\vert DB^n\left(r,X^n_r\right)D_x^2X_r^n(k,h)-DB\left(r,X_r\right)D_x^2X_r(k,h)\right\vert\ud r\\
  &\phantom{\lesssim}=\int_t^\tau\left(\mathrm{A}^n_1(r)+\mathrm{A}^n_2(r)\right)\ud r\ 
\end{align*}
so that
\begin{align*}
\bE\sup_{\tau\in[t,s]}\left\vert D_x^2X^n_r(k,h)-D^2_xX_r(k,h)\right\vert^p\lesssim  \int_t^s\left(\bE\left[\mathrm{A}^n_1(r)\right]^p+\bE\left[\mathrm{A}^n_2(r)\right]^p\right)\ud r\ .
\end{align*}
Choose now $\epsilon>\max\left\{1-\frac{1}{\alpha p},0\right\}$ and set
\begin{equation*}
  \bar{\nu}=\frac{1}{\alpha p}+\epsilon\ ,\quad \nu=2\frac{1+\epsilon\alpha p}{1+\alpha p(\epsilon-1)}\ ;
\end{equation*}
$\bar{\nu}$ and $\nu$ will be used as exponents and are chosen to ensure that all the inequalities exploited below are correct for any choice of $p>1$ and $\alpha\in(0,1)$.
From the bound
\begin{align*}
  \mathrm{A}^n_1(r)&\leq\left\vert D^2B\left(r,J^nX^n_r\right)\left(J^nD_xX^n_rk,J^nD_xX^n_rh-J^nD_xX_rh\right)\right\vert\\
                   &\phantom{\lesssim}+\left\vert D^2B\left(r,J^nX^n_r\right)\left(J^nD_xX^n_rk-J^nD_xX_rk,J^nD_xX_rh\right)\right\vert\\
                   &\phantom{\lesssim}+\left\vert \left[D^2B\left(r,J^nX^n_r\right)-D^2B\left(r,X_r\right)\right]\left(J^nD_xX_rk,J^nD_xX_rh\right)\right\vert\\
                   &\phantom{\lesssim}+\left\vert D^2B\left(r,X_r\right)\left(J^nD_xX_rk-D_xX_rk,J^nD_xX_rh-D_xX_rh\right)\right\vert\\
                   &\phantom{\lesssim}+\left\vert D^2B\left(r,X_r\right)\left(D_xX_rk,J^nD_xX_rh-D_xX_rh\right)\right\vert\\
                   &\phantom{\lesssim}+\left\vert D^2B\left(r,X_r\right)\left(J^nD_xX_rk-D_xX_rk,D_xX_rh\right)\right\vert\\
                   &\lesssim\left(1+\sup_{r\in[t,T]}\left\vert X^n_r\right\vert^{m}\right)\sup_{r\in[t,T]}\left\vert D_xX_r^nk\right\vert\sup_{r\in[t,T]}\left\vert D_xX^n_rh-D_xX_rh\right\vert\\
                   &\phantom{\lesssim}+\left(1+\sup_{r\in[t,T]}\left\vert X^n_r\right\vert^{m}\right)\sup_{r\in[t,T]}\left\vert D_xX_rh\right\vert\sup_{r\in[t,T]}\left\vert D_xX^n_rk-D_xX_rk\right\vert\\
                   &\phantom{\lesssim}+\sup_{r\in[t,T]}\left\vert D_xX_rk\right\vert\sup_{r\in[t,T]}\left\vert D_xX_rh\right\vert\sup_{r\in[t,T]}\left\vert J^nX^n_r-X_r\right\vert^{\alpha}\\
                   &\phantom{\lesssim}+\left\vert D^2B\left(r,X_r\right)\left(J^nD_xX_rk-D_xX_rk,J^nD_xX_rh-D_xX_rh\right)\right\vert\\
                   &\phantom{\lesssim}+\left\vert D^2B\left(r,X_r\right)\left(D_xX_rk,J^nD_xX_rh-D_xX_rh\right)\right\vert\\
                   &\phantom{\lesssim}+\left\vert D^2B\left(r,X_r\right)\left(J^nD_xX_rk-D_xX_rk,D_xX_rh\right)\right\vert\ 
\end{align*}
we get
\begin{align*}
  \bE\left[\mathrm{A}^n_1(r)\right]^p&\lesssim \left[\bE\left(1+\sup_{r\in[t,T]}\left\vert X^n_r\right\vert^{3mp}\right)\right]^{\nicefrac{1}{3}}\left[\bE\sup_{r\in[t,T]}\left\Vert D_xX_r^n\right\Vert^{3p}\right]^{\nicefrac{1}{3}}\\
                                     &\phantom{\lesssim}\times\left(\left\vert k\right\vert\left[\bE\sup_{r\in[t,T]}\left\vert D_xX^n_rh-D_xX_rh\right\vert^{3p}\right]^{\nicefrac{1}{3}}+\left\vert h\right\vert\left[\bE\sup_{r\in[t,T]}\left\vert D_xX^n_rk-D_xX_rk\right\vert^{3p}\right]^{\nicefrac{1}{3}}\right)\\
                                     &\phantom{\lesssim}+\left\vert h\right\vert^p\left\vert k\right\vert^p\left[\bE\sup_{r\in[t,T]}\left\Vert D_xX_r\right\Vert^{\nu p}\right]^{\nicefrac{2}{\nu}}\left[\bE\sup_{r\in[t,T]}\left\vert J^nX^n_r-X_r\right\vert^{\bar{\nu}\alpha p}\right]^{\nicefrac{1}{\bar{\nu}}}\\
                   &\phantom{\lesssim}+\bE\left\vert D^2B\left(r,X_r\right)\left(J^nD_xX_rk-D_xX_rk,J^nD_xX_rh-D_xX_rh\right)\right\vert^p\\
                   &\phantom{\lesssim}+\bE\left\vert D^2B\left(r,X_r\right)\left(D_xX_rk,J^nD_xX_rh-D_xX_rh\right)\right\vert^p\\
                   &\phantom{\lesssim}+\bE\left\vert D^2B\left(r,X_r\right)\left(J^nD_xX_rk-D_xX_rk,D_xX_rh\right)\right\vert^p \ .
\end{align*}
Thanks to Lemmas \ref{step1.a}-\ref{step1.b}, \eqref{eq:estX} and \eqref{eq:estDX}, the first three terms goes to zero.
Concerning  the remaining three terms, we employ Assumption \ref{ass:dphi}. 
More precisely, we have
\begin{align*}
  \left[\left(1+\bE\sup_{r\in[t,T]}\left\vert X_r\right\vert^{3mp}\right)\bE\sup_{r\in[t,T]}\left\vert D_xX_rk\right\vert^{3p}\bE\sup_{r\in[t,T]}\left\vert D_xX_rh\right\vert^{3p}\right]^{\nicefrac{1}{3}},
\end{align*}
and recalling again \eqref{eq:estX} and \eqref{eq:estDX} we get a uniform bound in $r \in [0,T]$. 
We can then pass to the limit exploiting Vitali convergence theorem.  
Similarly
\begin{align*}
  \mathrm{A}^n_2(r)&\leq\left\vert \left[DB\left(r,J^nX^n_r\right)-DB\left(r,X_r\right)\right]J^nD_x^2X^n_r(k,h)\right\vert\\
                   &\phantom{\lesssim}+\left\vert DB\left(r,X_r\right)J^n\left[D_x^2X^n_r(k,h)-D_x^2X_r(k,h)\right]\right\vert\\
                   &\phantom{\lesssim}+\left\vert DB\left(r,X_r\right)\left[J^nD^2_xX_r(k,h)-D^2_xX_r(k,h)\right]\right\vert\\
                   &\lesssim\sup_{r\in[t,T]}\left\vert D^2_xX^n_r(k,h)\right\vert\left(1+\sup_{r\in[t,T]}\left\vert X_r^n\right\vert^m+\sup_{r\in[t,T]}\left\vert X_r\right\vert^m\right)\sup_{r\in[t,T]}\left\vert J^nX^n_r-X_r\right\vert\\
                   &\phantom{\lesssim}+\left\vert D^2_xX^n_r(k,h)-D^2_xX_r(k,h)\right\vert\\
  &\phantom{\lesssim}+\left\vert DB\left(r,X_r\right)\left[J^nD^2_xX_r(k,h)-D^2_xX_r(k,h)\right]\right\vert\ ,
\end{align*}
where we used the $C^{1}$-regularity of $x\mapsto DB(\cdot, x)$ along with the growth of $D^2B$ given in  Assumption \ref{ass:B}. 
Taking the expectation of the $p$-th power, the first and last term converge to $0$ for a.e. $r \in [0,T]$ as above. Therefore
\begin{equation*}
  \bE\sup_{\tau\in[t,s]}\left\vert D_x^2X^n_r(k,h)-D^2_xX_r(k,h)\right\vert^p\lesssim N^n(s) +\int_t^s\bE\sup_{\tau\in[t,r]}\left\vert D^2_xX^n_\tau(k,h)-D^2_xX_\tau(k,h)\right\vert^p \ud r
\end{equation*}
where $N^n(s)$ contains all the other terms and  $N^n(s) \to 0$ for every $s\in[t,T]$. 
The application of the Gronwall lemma concludes the proof.\\
\end{proof}

\begin{proof}[Proof of Lemma \ref{step4}]
Recalling the equation satisfied by $D^2_xY_s(k,h)$ (analogously by $D^2_xY_s^n(k,h)$) we set
\begin{equation*}
  \Delta^2Y^n_s(k,h)=D^2_xY^n_s(k,h)-D^2_xY_s(k,h)\ ,\quad \Delta^2Z^n_s(k,h)=D^2_xZ^n_s(k,h)-D^2_xZ_s(k,h)
\end{equation*}
to obtain
\begin{align}
\nonumber  &\Delta^2Y^n_s(k,h)+\int_s^T\Delta^2Z^n_r(k,h)\ud W_r=\\
\tag*{\ccd{1}}&-\int_s^T\left[D_1G^n\left(r,X^n_r,Y^n_r,Z^n_r\right)D^2_xX^n_r(k,h)-D_1G\left(r,X_r,Y_r,Z_r\right)D^2_xX_r(k,h)\right]\ud r\\
\tag*{\ccd{2}}&-\int_s^T\left[D^2_{1,1}G^n\left(r,X^n_r,Y^n_r,Z^n_r\right)\left(D_xX^n_rk,D_xX^n_rh\right)-D^2_{1,1}G\left(r,X_r,Y_r,Z_r\right)\left(D_xX_rk,D_xX_rh\right)\right]\ud r\\
\tag*{\ccd{3}}&-\int_s^T\left[D^2_{1,2}G^n\left(r,X^n_r,Y^n_r,Z^n_r\right)\left(D_xY^n_rk,D_xX^n_rh\right)-D^2_{1,2}G\left(r,X_r,Y_r,Z_r\right)\left(D_xY_rk,D_xX_rh\right)\right]\ud r\\
\tag*{\ccd{4}}&-\int_s^T\left[D^2_{1,3}G^n\left(r,X^n_r,Y^n_r,Z^n_r\right)\left(D_xZ^n_rk,D_xX^n_rh\right)-D^2_{1,3}G\left(r,X_r,Y_r,Z_r\right)\left(D_xZ_rk,D_xX_rh\right)\right]\ud r\\
\tag*{\ccd{5}}&-\int_s^T\left[D_2G^n\left(r,X^n_r,Y^n_r,Z^n_r\right)D^2_xY^n_r(k,h)-D_2G\left(r,X_r,Y_r,Z_r\right)D^2_xY_r(k,h)\right]\ud r\\
\tag*{\ccd{6}}&-\int_s^T\left[D^2_{2,1}G^n\left(r,X^n_r,Y^n_r,Z^n_r\right)\left(D_xX^n_rk,D_xY^n_rh\right)-D^2_{2,1}G\left(r,X_r,Y_r,Z_r\right)\left(D_xX_rk,D_xY_rh\right)\right]\ud r\\
\tag*{\ccd{7}}&-\int_s^T\left[D^2_{2,2}G^n\left(r,X^n_r,Y^n_r,Z^n_r\right)\left(D_xY^n_rk,D_xY^n_rh\right)-D^2_{2,2}G\left(r,X_r,Y_r,Z_r\right)\left(D_xY_rk,D_xY_rh\right)\right]\ud r\\
\tag*{\ccd{8}}&-\int_s^T\left[D^2_{2,3}G^n\left(r,X^n_r,Y^n_r,Z^n_r\right)\left(D_xZ^n_rk,D_xY^n_rh\right)-D^2_{2,3}G\left(r,X_r,Y_r,Z_r\right)\left(D_xZ_rk,D_xY_rh\right)\right]\ud r\\
\tag*{\ccd{9}}&-\int_s^t\left[D_3G^n\left(r,X^n_r,Y^n_r,Z^n_r\right)D^2_xZ^n_r(k,h)-D_3G\left(r,X_r,Y_r,Z_r\right)D^2_xZ_r(k,h)\right]\ud r\\
\tag*{\ccd{10}}&-\int_s^T\left[D^2_{3,1}G^n\left(r,X^n_r,Y^n_r,Z^n_r\right)\left(D_xX^n_rk,D_xZ^n_rh\right)-D^2_{3,1}G\left(r,X_r,Y_r,Z_r\right)\left(D_xX_rk,D_xZ_rh\right)\right]\ud r\\
\tag*{\ccd{11}}&-\int_s^T\left[D^2_{3,2}G^n\left(r,X^n_r,Y^n_r,Z^n_r\right)\left(D_xY^n_rk,D_xZ^n_rh\right)-D^2_{3,2}G\left(r,X_r,Y_r,Z_r\right)\left(D_xY_rk,D_xZ_rh\right)\right]\ud r\\
\tag*{\ccd{12}}&-\int_s^T\left[D^2_{3,3}G^n\left(r,X^n_r,Y^n_r,Z^n_r\right)\left(D_xZ^n_rk,D_xZ^n_rh\right)-D^2_{3,3}G\left(r,X_r,Y_r,Z_r\right)\left(D_xZ_rk,D_xZ_rh\right)\right]\ud r\\
\tag*{\ccd{13}}&+D^2\Phi^n\left(X^n_T\right)\left(D_xX^n_Tk,D_xX^n_Th\right)-D^2\Phi\left(X_T\right)\left(D_xX_Tk,D_xX_Th\right)\\
\tag*{\ccd{14}}&+D\Phi^n\left(X^n_T\right)D^2_xX^n_T(k,h)-D\Phi\left(X_T\right)D^2_xX_T(k,h)\ .
\end{align}
We rephrase the above BSDE as
\begin{align}
\nonumber &\Delta^2Y^n_s(k,h)+\int_s^T\Delta^2Z^n_r(k,h)\ud W_r=\\
\nonumber &\ccd{1}+\ccd{2}+\ccd{3}+\ccd{4}\\
\nonumber &-\int_s^T D_2G^n\left(r,X^n_r,Y^n_r,Z^n_r\right)\Delta^2Y^n_r(k,h)\ud r\\
\tag*{\ccd{5a}} &-\int_s^T \left[D_2G^n\left(r,X^n_r,Y^n_r,Z^n_r\right)-D_2G\left(r,X_r,Y_r,Z_r\right)\right]D^2_xY_r(k,h)\ud r\\
\nonumber &+\ccd{6}+\ccd{7}+\ccd{8}\\
\nonumber &-\int_s^TD_3G^n\left(r,X^n_r,Y^n_r,Z^n_r\right)\Delta^2Z^n_r(k,h)\ud r\\
\tag*{\ccd{9a}} &-\int_s^T \left[D_3G^n\left(r,X^n_r,Y^n_r,Z^n_r\right)-D_3G\left(r,X_r,Y_r,Z_r\right)\right]D^2_xZ_r(k,h)\ud r\\
\nonumber &+\ccd{10}+\ccd{11}+\ccd{12}+\ccd{13}+\ccd{14}\ ,
\end{align}
which is of the form
\begin{equation*}
  \Delta^2Y^n_s(k,h)+\int_s^T\Delta^2Z^n_r(k,h)\ud W_r=\bar\eta^n+\int_s^T\bar\alpha^n_r\Delta^2Y^n_r(k,h)\ud r+\int_s^T\bar\beta^n_r\ud r+\int_s^T\bar\gamma^n_r\Delta^2Z^n_r(k,h)\ud r
\end{equation*}
with
\begin{gather*}
  \bar\eta^n:=\ccd{13}+\ccd{14}\ ,\\
  \bar\alpha^n_r:=-D_2G^n\left(r,X^n_r,Y^n_r,Z^n_r\right)\ ,\\
  \bar\gamma^n_r:=-D_3G^n\left(r,X^n_r,Y^n_r,Z^n_r\right)\ ,\\
  \int_t^T\bar\beta^n_r\ud r=\ccd{1}+\ccd{2}+\ccd{3}+\ccd{4}+\ccd{5a}+\ccd{6}+\ccd{7}+\ccd{8}+\ccd{9a}+\ccd{10}+\ccd{11}+\ccd{12}\ . 
\end{gather*}
Defining $V$ as in (\ref{eq:Vn}), the equation being linear, we can apply the same strategy as in the proof of Lemma \ref{step3}. Hence it suffices to check that for some $p\ge 2$
\begin{equation*}
  \bE\left\vert \bar\eta^n\right\vert^p+\bE\left(\int_t^T\left\vert\bar\beta^n_r\right\vert\ud r\right)^p\xrightarrow{n\to\infty}0\ .
\end{equation*}
The convergence of $\bar\eta^n$ under Assumption \ref{ass:dphi} was proved in \cite{flandoli2016infinite} for $p=1$ and $\Phi\in C^{2,\alpha}_b$; the extension to $\Phi$ with polynomial growth relies on Theorem \ref{thm:X} and on the uniform integrability of $\Phi^n(X^n)$ in $L^p$. By H\"older inequality we also obtain convergence of $\bar\eta^n$ in $L^p(\Omega)$ for any $p\ge 2$.

We now show how to deal with some of the addends defining $\bar\beta^n$. The same technique can be used also for the remaining terms.
\begin{align*}
 \bE\left\vert\ccd{1}\right\vert^p&\leq \bE\left(\int_t^T\left\vert D_1G^n\left(r,X^n_r,Y^n_r,Z^n_r\right)\left[D^2_xX^n_r(k,h)-D^2_xX_r(k,h)\right]\right\vert\ud r\right)^p\\
 &+\bE\left(\int_t^T\left\vert \left[D_1G^n\left(r,X^n_r,Y^n_r,Z^n_r\right)-D_1G\left(r,X_r,Y_r,Z_r\right)\right]D^2_xX_r(k,h)\right\vert\ud r\right)^p\\
 &=\bE\left(\left[\overline{\mathrm{C}}_1^n\right]^p+\left[\overline{\mathrm{C}}_2^n\right]^p\right)\ .
\end{align*}
Thanks to the uniform bounds on $X^n$, $Y^n$, $Z^n$ and Assumption \ref{ass:G} we have
\begin{multline*}
  \left[\overline{\mathrm{C}}_1^n\right]^p\lesssim\left(1+\sup_{r\in[t,T]}\left\vert X_r^n\right\vert^{mp}\right)\sup_{r\in[t,T]}\left\vert \left[D^2_xX^n_r-D^2_xX_r\right](k,h)\right\vert^p\\ \times\left(1+\sup_{r\in[t,T]}\left\vert Y^n_r\right\vert^p+\left(\int_t^T\left\vert Z^n_r\right\vert^2\ud r\right)^{\nicefrac{p}{2}}\right)\ ,
\end{multline*}
hence $\bE\left[\overline{\mathrm{C}}_1^n\right]^p\to 0$ by H\"older's inequality and Lemma \ref{step1.b}. 
The convergence of $\overline{\mathrm{C}}_2^n$ is identical to the one of $\int\mathrm{C}_2^n$ in the proof of Lemma \ref{step3}. 
By using the shorthand
\begin{gather*}
  D^2_{ij}G_r(n):=D^2_{ij}G\left(r,J^nX^n_r,Y^n_r,Z^n_r\right)\ , 
  D^2_{ij}G_r(\cdot):=D^2_{ij}G\left(r,X_r,Y_r,Z_r\right)\ , 
\end{gather*}
for $i,j = 1,2,3$, for the term $\ccd{2}$ we get 
\begin{align*}
  \bE\left\vert\ccd{2}\right\vert^p&\lesssim \bE\left(\int_t^T\left\vert D^2_{1,1} G_r(n)\left(J^nD_xX^n_rk,J^nD_xX^nh\right)-D^2_{1,1}G_r(n)\left(J^nD_xX_rk,J^nD_xX_rh\right)\right\vert\ud r\right)^p\\
  &+\bE\left(\int_t^T\left\vert D^2_{1,1}G_r(n)\left(J^nD_xX_rk,J^nD_xX_rh\right)-D^2_{1,1}G_r(\cdot)\left(J^nD_xX_rk,J^nD_xX_rh\right)\right\vert\ud r\right)^p\\
  &+\bE\left(\int_t^T\left\vert D^2_{1,1}G_r(\cdot)\left(J^nD_xX_rk,J^nD_xX_rh\right)-D^2_{1,1}G_r(\cdot)\left(D_xX_rk,D_xX_rh\right)\right\vert\ud r\right)^p\\
  &=\bE\left(\left[\overline{\mathrm{F}}^n_{111}\right]^p+\left[\overline{\mathrm{F}}^n_{112}\right]^p+\left[\overline{\mathrm{F}}^n_{113}\right]^p\right)\ .
\end{align*}
$\overline{\mathrm{F}}_{112}^n$ can be studied again as $\int C^n_2$ in the proof of Lemma \ref{step3}, while
\begin{align*}
\left[\overline{\mathrm{F}}_{111}^n\right]^p&\leq\left(\int_t^T \vert D^2_{1,1}G_r(n)\big[\left(J^nD_xX^n_r,J^nD_xX^n_rh-J^nD_xX_rh\right)\right.\\
&\left.\phantom{aaaaaaaaa\int_t^T}+\left(J^nD_xX_r^nk-J^nD_xX_rk,J^nD_xX_rh\right)\big]\vert\ud r\right)^p\\
  &\lesssim\Bigg[\sup_{r\in[t,T]}\left\vert D_xX^n_rk\right\vert^p\sup_{r\in[t,T]}\left\vert D_xX^n_rh-D_xX_rh\right\vert^p\\
&\phantom{aaaaaaaaaaaaa}+\sup_{r\in[t,T]}\left\vert D_xX_rk\right\vert^p\sup_{r\in[t,T]}\left\vert D_xX^n_rh-D_xX_r^nh\right\vert^p\Bigg]\\
&\phantom{aaaaaaaaaaaaa}\times\left(1+\sup_{r\in[t,T]}\left\vert X_r^n\right\vert^{mp}\right)\left(1+\sup_{r\in[t,T]}\left\vert Y^n\right\vert^p+\left(\int_t^T\left\vert Z^n_r\right\vert^2\ud r\right)^{\nicefrac{p}{2}}\right)\ ,
\end{align*}
and
  \begin{align*}
    \left[\overline{\mathrm{F}}_{113}^n\right]^p&\leq\int_t^T\left\vert D^2_{1,1}G_r(\cdot)\left(J^nD_xX_rk-D_xX_rk,J^nD_xX_rh-D_xX_rh\right)\right\vert^p\ud r\\
    &\phantom{aaa}+\int_t^T\left\vert D^2_{1,1}G_r(\cdot)\left(D_xX_rk,J^nD_xX_rh-D_xX_rh\right)\right\vert^p\ud r\\
    &\phantom{aaa}+\int_t^T\left\vert D^2_{1,1}G_r(\cdot)\left(J^nD_xX_rk-D_xX_rk,D_xX_rh\right)\right\vert^p\ud r \ ;
  \end{align*}
the desired convergence then follows by taking expectation, applying H\"older's inequality and using Assumption \ref{ass:dphi}, the uniform bound (\ref{eq:estDXunif}) and the Vitali convergence theorem.

\begin{align*}
  \bE\left\vert\ccd{3}\right\vert^p&\lesssim \bE\left(\int_t^T\left\vert D^2_{12}G_r(n)\left(D_xY^n_rk,J^nD_xX^n_rh\right)-D^2_{12}G_r(\cdot)\left(D_xY_rk,J^nD_xX_rh\right)\right\vert\ud r\right)^p\\
  &\phantom{\lesssim}+\bE\left(\int_t^T\left\vert D^2_{12}G_r(\cdot)\left(D_xY_rk,J^nD_xX_rh\right)-D^2_{12}G_r(\cdot)\left(D_xY_rk,D_xX_rh\right)\right\vert\ud r\right)^p\\
  &=\bE\left(\left[\overline{\mathrm{F}}_{121}^n\right]^p+\left[\overline{\mathrm{F}}_{122}^n\right]^p\right)
\end{align*}
Choose again any $\epsilon>\max\left\{1-\frac{1}{\alpha p},0\right\}$ and set
\begin{equation*}
  \bar{\nu}=\frac{1}{\alpha p}+\epsilon\ ,\quad \nu=2\frac{1+\epsilon\alpha p}{1+\alpha p(\epsilon-1)}\ ;
\end{equation*}
then by H\"older character of the second derivatives of $G$
\begin{align*}
  \bE\left[\overline{\mathrm{F}}_{121}^n\right]^p&\lesssim\bE\left(\int_t^T\left\vert D^2_{12}G_r(n)\left(D_xY^n_rk,J^nD_xX^n_rh\right)-D^2_{12}G_r(n)\left(D_xY^n_rk,J^nD_xX_rh\right)\right\vert\ud r\right)^p\\
  &\phantom{aaaaa}+\bE\left(\int_t^T\left\vert D^2_{12}G_r(n)\left(D_xY^n_rk,J^nD_xX_rh\right)-D^2_{12}G_r(n)\left(D_xY_rk,J^nD_xX_rh\right)\right\vert\ud r\right)^p\\
  &\phantom{aaaaa}+\bE\left(\int_t^T\left\vert\left[D^2_{12}G_r(n)-D^2_{12}G_r(\cdot)\right]\left(D_xY_rk,J^nD_xX_rh\right)\right\vert\ud r\right)^p\\
  &\lesssim\left(1+\left[\bE\sup_{r\in[t,T]}\left\vert X_r^n\right\vert^{4mp}\right]^{\nicefrac{1}{4}}\right)\left(1+\left[\bE\sup_{r\in[t,T]}\left\vert Y_r^n\right\vert^{4p}\right]^{\nicefrac{1}{4}}\right)\\
  &\phantom{aaaaaaa}\times\left(\left[\bE\sup_{r\in[t,T]}\left\vert D_xY^n_rk\right\vert^4\right]^{\nicefrac{1}{4}}\left[\bE\sup_{r\in[t,T]}\left\vert D_xX^n_rh-D_xX_rh\right\vert^{4p}\right]^{\nicefrac{1}{4}}\right.\\
  &\phantom{aaaaa}+\left.\left[\bE\sup_{r\in[t,T]}\left\vert D_xX_rh\right\vert^4\right]^{\nicefrac{1}{4}}\left[\bE\sup_{r\in[t,T]}\left\vert D_xY^n_rk-D_xY_rk\right\vert^{4p}\right]^{\nicefrac{1}{4}}\right)\\
  &\phantom{aaaaa}+\left[\bE\sup_{r\in[t,T]}\left\vert D_xY_rk\right\vert^{\nu p}\right]^{\nicefrac{1}{\nu}}\left[\bE\sup_{r\in[t,T]}\left\vert D_xX_rh\right\vert^{\nu p}\right]^{\nicefrac{1}{\nu}}\\
  &\phantom{aaaaaaa}\times\left(\left[\bE\sup_{r\in[t,T]}\left\vert Y^n_r-Y_r\right\vert^{\bar{\nu}\alpha p}\right]^{\nicefrac{1}{\bar{\nu}}}+\left[\bE\int_t^T\left\vert J^nX^n_r-X_r\right\vert^{\bar{\nu}\alpha p}\ud r\right]^{\nicefrac{1}{\bar{\nu}}}\right)\\
  &\phantom{aaaaa}+\left[\bE\sup_{r\in[t,T]}\left\vert D_xY_rk\right\vert^{3p}\right]^{\nicefrac{1}{3}}\left[\bE\sup_{r\in[t,T]}\left\vert D_xX_rh\right\vert^{3p}\right]^{\nicefrac{1}{3}}\left[\bE\left(\int_t^T\left\vert Z^n_r-Z_r\right\vert^2\ud r\right)^{\nicefrac{3p}{2}}\right]^{\nicefrac{1}{3}}
\end{align*}
which goes to $0$ by Lemmas \ref{step1.a}, \ref{step1.b} and \ref{step2.a}, thanks to the choice of $\bar \nu$ and $\nu$.
The term $\overline{\mathrm{F}}_{122}^n$ can be controlled as before using Assumption \ref{ass:dphi}.\\
The estimates for $\ccd{4}$ are almost identical. 
Terms \ccd{5a} and \ccd{9a} are treated in the same way as $\beta^n_2$ and $\beta^n_3$ in Lemma \ref{step3}, respectively. 
We have already shown how to deal with \ccd{12} above, the remaining terms \ccd{6}, \ccd{7}, \ccd{8}, \ccd{10}, \ccd{11} are then simple adaptations of \ccd{1}, \ccd{2}, \ccd{3} and \ccd{12}.
\end{proof}

\section*{Acknowledgements}
C. Orrieri and G. Zanco are supported by the GNAMPA Project 2019 ``Ttrasporto ottimo per dinamiche con interazione''.
G. Zanco was partly funded by the Austrian Science Fund (FWF) project F 65.

\bigskip

\makeatletter
\providecommand\@dotsep{5}
\makeatother

\bibliography{mybib}
\bibliographystyle{plain}

\end{document}